\documentclass[12pt, a4paper]{article}
\usepackage{times,amsmath,amsthm,amssymb,graphicx,xspace,epsfig,xcolor}
\usepackage[plain]{fullpage}
\usepackage{upgreek}

\usepackage[utf8]{inputenc}
\usepackage[english]{babel}
\usepackage{tikz}
\usepackage{color,hyperref}
\usepackage{comment}
\usepackage{authblk}
\usepackage{subcaption}
\usepackage{soul}
\usepackage{enumitem}

\usepackage{tikz}
\usetikzlibrary{calc}

\tikzstyle{vertex}=[circle,draw, top color=gray!5, 
	    bottom color=gray!30, minimum size=16pt, scale=1, inner sep=0.5pt]
\tikzstyle{vertex4}=[circle,draw, fill=green, minimum size=16pt, scale=1, inner sep=0.5pt]
\tikzstyle{vertex_nice}=[circle,draw, fill=cyan, minimum size=16pt, scale=1, inner sep=0.5pt]
\tikzstyle{maycoincide}=[circle,draw, fill=black, minimum size=16pt, scale=1,
     inner sep=0.5pt]

\tikzstyle{maycoincide1}=[circle,draw, fill=blue, minimum size=16pt, scale=1,
     inner sep=0.5pt]
     
\tikzstyle{maycoincide2}=[circle,draw, fill=red, minimum size=16pt, scale=1,
     inner sep=0.5pt]

\tikzstyle{arc}=[->, very thick]
\tikzstyle{arc_chelou}=[->, very thick,red]
\tikzstyle{edge}=[very thick]
\tikzstyle{edge_chelou}=[very thick, red]
\tikzstyle{chelou}=[->, very thick, red]
\tikzstyle{nonchelou}=[->, very thick, blue]

\newtheorem{theorem}{Theorem}[section]
\newtheorem*{theorem*}{Theorem}
\newtheorem{corollary}[theorem]{Corollary}
\newtheorem{proposition}[theorem]{Proposition}
\newtheorem{lemma}[theorem]{Lemma}
\newtheorem{claim}{Claim}

\theoremstyle{definition}

\newtheorem{conjecture}[theorem]{Conjecture}
\newtheorem{definition}[theorem]{Definition}

\newenvironment{proofclaim}[1][]{\par\noindent {\it Proof of claim}.\ }{\hfill$\lozenge$\par\vspace{11pt}}

\newcommand{\pth}[1]{\left(#1\right )}
\newcommand{\dic}{\vec{\chi}}
\newcommand{\bid}{\overleftrightarrow}

\newcommand{\bigO}{\mathcal{O}}

\newcommand{\wheelodd}{\overrightarrow{\mathcal{W}}_3}

\newcommand{\oddcycle}{\vec{\mathcal{C}}_{odd}}

\renewcommand{\phi}{\varphi}
\renewcommand{\epsilon}{\varepsilon}

\newcommand{\Vquatre}{\mbox{$V_4^s$}}

\title{On the minimum number of arcs in $k$-dicritical oriented graphs}
\author[1]{Pierre Aboulker}
\author[2]{Thomas Bellitto}
\author[3]{Fr\'ed\'eric Havet}
\author[1]{Cl\'ement Rambaud}

\affil[1]{DIENS, \'Ecole normale sup\'erieure, CNRS, PSL University, Paris, France.}
\affil[2]{Sorbonne Université, CNRS, LIP6, Paris, France}
\affil[3]{CNRS, Universit\'e C\^ote d'Azur, I3S, INRIA, Sophia Antipolis, France}

\begin{document}

\maketitle

\begin{abstract} 
The dichromatic number $\dic(D)$ of a digraph $D$ is the least integer $k$ such that  $D$ can be partitioned into $k$ directed acyclic digraphs. A digraph is $k$-dicritical  if $\dic(D) = k$ and each proper subgraph $D'$ of $D$ satisfies $\dic(D') \leq k-1$. An oriented graph is a digraph with no directed cycle of length $2$. 
For integers $k$ and $n$, we denote by $o_k(n)$ the minimum number of edges of a $k$-critical oriented graph on $n$ vertices (with the convention $o_k(n)=+\infty$ if there is no $k$-dicritical
oriented graph of order $n$). 
The main result of this paper is a proof that $o_3(n) \geq \frac{7n+2}{3}$ together with a construction witnessing that $o_3(n) \leq \left \lceil  \frac{5n}{2} \right \rceil$ for all $n \geq 12$.  
We also give a construction showing that for all sufficiently large $n$ and all $k\geq 3$, $o_k(n) < (2k-3)n$, disproving a conjecture of Hoshino and  Kawarabayashi. 
Finally, we prove that, for all $k\geq 2$, $o_k(n) \geq \pth{ k - \frac{3}{4}-\frac{1}{4k-6}} n + \frac{3}{4(2k-3)}$, improving the previous best known lower bound of Bang-Jensen, Bellitto, Schweser and Stiebitz. 

\end{abstract}

\section{Introduction}
Let $G$ be a graph.

We denote by $V(G)$ its vertex set and by $E(G)$ its edge set, and we set $n(G)=|V(G)|$ and $m(G)=|E(G)|$. 
A {\bf subgraph} of $G$ is a graph $G'$ such that $V(G') \subseteq V(G)$
and $E(G')\subseteq E(G)$. A {\bf proper subgraph} of $G$ is a subgraph $G'$ of $G$ such that $V(G') \neq V(G)$ or $E(G')\neq E(G)$.

A {\bf proper $k$-colouring} of a graph $G$ is a partition of the vertex set of $G$ into $k$ disjoint {\bf stable sets} (i.e. sets of pairwise non-adjacent vertices).  
A graph is {\bf $k$-colourable} if it has a $k$-colouring.  
The  {\bf chromatic number} of a graph $G$, denoted by $\chi(G)$, is the least integer $k$ such that $G$ is $k$-colourable. 
The chromatic number is monotone in the sense that if $G'$ is a subgraph of $G$, then $\chi(G')\leq \chi(G)$.

A graph $G$ is said to be {\bf critical} and {\bf $k$-critical} if every proper subgraph $G'$ of $G$ satisfies $\chi(G') < \chi(G) =k$. Clearly, every graph contains a critical subdigraph with the same chromatic number. Hence many problems concerning the chromatic number can be reduced to critical graphs. The study of critical graphs was initiated by G. A. Dirac in the 1950s and has attracted a lot of attention since then. 
Dirac~\cite{Dirac52,Dirac57, Dirac67} established the basic properties of critical graphs and started to investigate the minimum number of edges possible in a $k$-critical graph of order $n$, denoted by $g_k(n)$. It is easy to show that the minimum degree of a $k$-critical graph is at least $k-1$. Consequently, $g_k(n)\geq \frac{1}{2}(k-1)n$ for all $n\geq k$. Brooks’ Theorem \cite{Bro41} implies that $g_k(n) = \frac{1}{2}(k-1)n$ if and only if $n =k$ or $k=3$ and n is odd. Dirac~\cite{Dirac57} (see also [7]) proved that $g_k(n) \geq \frac{1}{2}(k-1)n + \frac{k -3}{2}$ 
for $k\geq 4$ and $n\geq k+ 2$. Note that there is no $k$-critical graph of order $k+1$. 
Gallai~\cite{Gal63a, Gal63b} proved that the complement of a $k$-critical graph of order at most $2k-2$ is disconnected, and deduced
 the exact values of $g_k(n)$ when $k\geq 4$ and $k+2\leq n \leq 2k-1$.
 Kostochka and Stiebitz~\cite{KoSt99} proved $g_k(n) \geq \frac{1}{2}(k-1)n + k -3$. 
 In 2014, using the potential method, Kostochka and Yancey~\cite{kostochka2014ore} proved that
 $g_4(n) = \left\lceil \frac{5n-2}{3}\right\rceil$ for all $n]\geq 4$, $n\neq 5$.
 Furthermore, they~\cite{KoYa14} established the lower bound $g_k(n) \geq \left\lceil  \frac{(k+1)(k-2)n -k(k-3)}{2k-2}\right\rceil$ 
 which is sharp when $k\geq 4$ and $n \equiv 1 \mod (k-1)$.
 In particular, they proved that $g_5(n) \geq \frac{9}{4} n - \frac{5}{4}$ with equality when $n \equiv 1 \mod 4$.

 The {\bf girth} of a graph $G$ is the minimum length of a cycle in $G$ or $+\infty$ if $G$ is acyclic.
 Also using the potential method, Liu and Postle~\cite{liu2017minimum} showed that the minimum number of edges of a $4$-critical graphs is larger than $g_4(n)$ if we impose the graph to have girth $5$:
 If $G$ is a $4$-critical graph of girth at least $5$, then $m(G) \geq \frac{ 5 n(G)+2}{3}$.
 Likewise, Postle~\cite{postle2017minimum} showed that every $5$-critical graph with girth at least $4$ satisfies $m(G) \geq (\frac{9}{4} +\epsilon) n(G) - \frac{5}{4}$ for $\epsilon=\frac{1}{84}$.

 \medskip
 
 Let $D$ be a digraph.

We denote by $VD)$ its vertex set and by $A(D)$ its arc set, and we set $n(D)=|V(D)|$ and $m(D)=|A(D)|$.
A {\bf subdigraph} of $D$ is a digraph $D'$ such that $V(D') \subseteq V(D)$
and $A(D')\subseteq A(D)$. A {\bf proper subdigraph} of $D$ is a subdigraph $D'$ of $D$ such that $V(D') \neq V(D)$ or $A(D')\neq A(D)$.

A {\bf $k$-dicolouring} of a digraph is a partition of its vertex set into $k$ subsets inducing acyclic subdigraphs. Alternatively, it is a $k$-colouring $\phi:V(D) \rightarrow [k]$ such that $D[\phi^-(c)]$ is acyclic for every colour $c\in [k]$.
A digraph is {\bf $k$-dicolourable} if it has a $k$-dicolouring.  
The {\bf dichromatic number} of a digraph $D$, denoted by $\vec{\chi}(D)$, is the least integer $k$ such that $D$ is $k$-dicolourable. 
This notion was introduced and investigated by Neumann-Lara~\cite{Neu82}. 
It can be seen as a generalization of the chromatic number. Indeed, for a graph $G$, the {\bf bidirected graph} $\bid{G}$ is the digraph obtained from $G$ by replacing each edge by a {\bf digon}, that is a pair of oppositely directed arcs between the same end-vertices. Observe that $\chi(G) = \vec{\chi}(\bid{G})$ since any two adjacent vertices in $\bid{G}$ induce a directed cycle of length $2$. A \textbf{tournament} is an orientation of a complete graph. 

Similarly to the chromatic number, the dichromatic number is monotone: if $D'$ is a subdigraph of $D$, then $\dic(D')\leq \dic(D)$. A digraph $D$ is said to be {\bf dicritical} and {\bf $k$-dicritical} if every proper subdigraph $D'$ of $D$ satisfies $\dic(D') < \dic(D) =k$. Clearly, every digraph contains a dicritical subdigraph with the same dichromatic number. 
Dicritical digraphs were introduced in Neumann-Lara's seminal paper~\cite{Neu82}.
Clearly, $\dic(D) =1$  if and only if $D$ is acyclic. As a consequence, a digraph $D$ is $2$-dicritical if and only if $D$ is a directed cycle. Bokal, Fijav\v{z}, Juvan, Kayll and Mohar~\cite{Bokal20003} proved that deciding whether a given digraph is $k$-dicolourable is NP-complete for all $k\geq 2$. Hence a characterization of the class of $k$-dicritical digraphs with fixed $k\geq 3$ is unlikely. However, it might be possible to derive bounds on the minimum number of arcs in a $k$-dicritical digraph.
Kostochka and Stiebitz~\cite{kostochka2020minimum} deduced the following from a Brooks-type result for digraphs due to Mohar~\cite{Mohar10}: if $D$ is a $3$-dicritical digraph of order $n\geq 3$, then $m(D)\geq 2n$ and equality holds if and only if n is odd and D is a bidirected odd cycle.

For integers $k$ and $n$, let $d_k(n)$ denote the minimum number of arcs in a $k$-dicritical digraph of order $n$. 
By the above observations, $d_2(n) = n$ for all $n\geq 2$, and $d_3(n) \geq 2n$ for all possible $n$, and equality holds if and only if $n$ is odd and $n \geq 3$.

If $G$ is a $k$-critical graph, then $\bid{G}$ is $k$-dicritical, so
$d_k(n) \leq 2g_k(n)$ provided that there is a $k$-critical graph of order $n$. It is known that, for $k\geq 4$, there is a $k$-critical graph of order $n$ if and only if $n \geq k$ and $n\neq k+1$. Moreover, there is a $k$-dicritical digraph of order $k+1$ for all $k\geq 3$ : take the disjoint union of a directed $3$-cycle $\vec{C}_3$ and a bidirected complete graph on $k-2$ vertices $\bid{K}_{k-2}$, and add a digon between each vertex of $3$-cycle $\vec{C}_3$ and each vertex of $\bid{K}_{k-2}$.


Kostochka and Stiebitz proved that if $D$ is a $4$-dicritical digraph then $m(D) \geq \frac{10}{3} n(D) - \frac{4}{3}$. This bound is sharp if $n(D) \equiv 1 \mod 3$ or $n(D) \equiv 2 \mod 3$ and $n \neq 5$.
They also proposed the following conjecture.
\begin{conjecture}[Kostochka and Stiebitz~\cite{kostochka2020minimum}] If $D$ is a $k$-dicritical digraph of order $n$ with $k\geq 4$ and $n\geq k$, then $m(D)\geq 2g_k(n)$ and equality implies that $D$ is a bidirected $k$-critical graph. As a consequence, $d_k(n) = 2g_k(n)$ when $n \geq k$ and $n \neq k+1$.
\end{conjecture}

 Kostochka and Stiebitz~\cite{kostochka_number_2000} showed that if a $k$-critical $G$
 is triangle-free (that is has no cycle of length $3$), then $m(G)/n(G) \geq k - o(k)$
 as $k \to + \infty$. Informally, this means that
the minimum average degree of a $k$-critical triangle-free graph is (asymptotically) twice the minimum average degree of a $k$-critical graph.
Similarly to this undirected case, it is expected that the  minimum number of arcs in a $k$-dicritical digraph
of order $n$ is larger than $d_k(n)$ if we impose this digraph to have no short directed cycles, and in particular if the digraph is an {\bf oriented graph}, that is a digraph with no digon.
Let $o_k(n)$ denote the minimum number of arcs in a $k$-dicritical oriented graph of order $n$ (with the convention $o_k(n)=+\infty$ if there is no $k$-dicritical
oriented graph of order $n$). 
Clearly $o_k(n) \geq d_k(n)$.

\begin{conjecture}[Kostochka and Stiebitz~\cite{kostochka2020minimum}] 
There is a constant $c>1$ such that $o_k(n) > c \cdot d_k(n)$ for $k\geq 3$ and $n$ sufficiently large.
\end{conjecture}

We now describe the results that we improve.  
Hoshino and Kawarabayashi~\cite{HoKa15} observed that, using iteratively an analogue of Haj\'os construction for oriented graphs, for each $k\geq 3$, one can construct an infinite family of sparse $k$-dicritical oriented graphs $D$ such that $m(D) < \frac{1}{2}(k^2-k+1)n(D)$.
Consequently, $o_k(n) < \frac{1}{2}(k^2-k+1) n$ for infinitely many values of $n$.
When $k=3$, a better result can be obtained using
the unique $3$-dicritical oriented graph with $20$ arcs. It yields $3$-dicritical oriented graphs with $n$ vertices and $\frac{19n}{6}+1$ arcs for all $n \equiv 1 \mod 6$.
Consequently, $o_3(n) \le \frac{19n}{6}+1$ for all $n \equiv 1 \mod 6$. 
In~\cite{bang2019haj}, it is proved that $o_k(n) \geq (k- \frac{5}{6} - \frac{1}{6(3k-2)})n$.
\medskip 

\subsection*{Our results}
In this paper, we improve both the lower and the upper bound on
$o_k(n)$ for $k \geq 3$ and in particular $o_3(n)$.
First, in Subsection~\ref{sec:lower3}, we show that $o_3(n) \leq \lceil\frac{5}{2}n\rceil$ for every $n\geq 12$, by describing a $3$-dicritical oriented graph with $n$ vertices and  $\lceil\frac{5}{2}n\rceil$ arcs for all $n \geq 12$.
Then, in Subsection~\ref{sec:lowerk}, we  extend this construction to prove that for
every $k \geq 3$, $o_k(n) < (2k-3)n$ for every $n$ large enough.
This disproves a conjecture of Hoshino and Kawarabayashi~\cite{HoKa15}
stating that any dicritical oriented
graph $D$ satisfies $\frac{m(D)}{n(D)} \geq \frac{k^2}{2}- \bigO(k)$. 
For the lower bounds, we prove 
$o_k(n) \geq (k-\frac{3}{4}-\frac{1}{4k-6})n+\frac{3}{2(4k-6)}$ for every $k \geq 2$
in Section~\ref{sec:lower_ok(n)}
and in the case $k=3$ we show $o_3(n) \geq \frac{7n+2}{3}$, which constitutes our main
theorem.

\begin{theorem}\label{thm:main_thm_oriented_critical}
If $D$ is a $3$-dicritical oriented graph, then 
\[
m(D) \geq \frac{7n(D)+2}{3} .
\]
\end{theorem}

To prove Theorem~\ref{thm:main_thm_oriented_critical} we use the so-called
potential method introduced by Kostochka and Yancey 
\cite{kostochka2014ore, kostochka2020minimum}.
We also use some ideas introduced by Liu and Postle \cite{postle2017minimum, liu2017minimum}
who were interested in the minimum number of edges in triangle-free critical graphs.

We  actually prove a more general result than Theorem~\ref{thm:main_thm_oriented_critical}. It holds for every $3$-dicritical digraph and takes into account the number of digons.
For any digraph $D$, let $B(D)$ be the (undirected) graph with vertices the vertices of $D$
that are incident to at least one digon, and with edges all pairs of vertices
linked by a digon in $D$. 
Recall that a {\bf matching} in a graph is a set of edges without common end-vertices.
Let $\pi(D)$ be the size of a maximum matching of $B(D)$.
In particular, $\pi(D)=0$ if and only if $D$ is an oriented graph.

\begin{definition}
If $D$ is a digraph and $R \subseteq V(D)$, the {\bf potential}
of $R$ in $D$ is
\[
\rho_D(R) = 7|R| - 3m(D[R]) - 2\pi(D[R])
\]
and we write $\rho(D) = \rho_D(V(D))$.
\end{definition}

Theorem~\ref{thm:main_thm_oriented_critical} is equivalent to the statement
\emph{every $3$-dicritical oriented graph has potential at most $-2$}.
We prove that this statement holds not only
for oriented graphs, but for all $3$-dicritical digraphs except a few exceptions
which have digons.
The first family of exceptions are the bidirected odd cycles, which have
potential $1$. The second family of exceptions are the odd $3$-wheels.
An {\bf odd $3$-wheel} is a digraph obtained by connecting a vertex $c$ to a directed
$3$-cycle $(x,y,z,x)$ by three bidirected odd paths.
It is easy to check that any odd 3-wheel $D$ is a $3$-dicritical digraph
with $m(D)=2n(D)+1$ and $\pi(D) = \frac{n(D)-2}{2}$, so $\rho(D) =-1$.
These are the only exceptions.


\begin{theorem}\label{thm:main_potential}
If $D$ is a $3$-dicritical digraph, then
\begin{itemize}
\item $\rho(D) = 1$ if $D$ is a bidirected odd cycles,
\item $\rho(D) = -1$ if $D$ is an odd $3$-wheel,
\item $\rho(D) \leq -2$ otherwise.
\end{itemize}
\end{theorem}

This result implies Theorem~\ref{thm:main_thm_oriented_critical} because  $\pi(D)=0$ for every oriented graph $D$, and bidirected odd cycles and odd $3$-wheels have digons.

The proof of Theorem~\ref{thm:main_potential} is spread into three sections. We introduce a few tools in Section~\ref{sec:tools}; then we study some structural properties and the potential of some particular digraphs that pop-up during the proof; finally we give the proof of the theorem in Section~\ref{sec:mainproof}.

\section{Notations}

A {\bf forest} is an acyclic graph. A {\bf tree} is a connected forest.
The vertices of degree at most $1$ (resp. at least $2$) in a forest are called {\bf leaves} (resp. {\bf internal vertices}).

\medskip

Let $D$ be a digraph.

The {\bf underlying graph} of $D$ is the graph with vertex set $VD)$ in which two verticea are joined by an edge if and only if there is at least one arc between them.

 A {\bf  digon} is a pair of arcs in opposite directions between the same vertices.
The digon $\{xy,yx\}$ is denoted by $[x,y]$. 
The {\bf digon graph} of $D$, denoted by $B(D)$, is the graph with vertex set $V(D)$ in which $uv$ is an edge if and only if $[u,v]$ is a digon on $D$.

A digraph is a {\bf bidirected graph} if every arc is in a digon.
A {\bf bidirected path} (resp. {\bf bidirected cycle}, {\bf bidirected tree}, 
{\bf bidirected complete graph}) is a bidirected graph $D$ such that
$B(D)$ is a path (resp. cycle, tree, complete graph). In other words, it is 
a digraph obtained from a  path (resp. cycle, tree, complete graph) by replacing every edge by a digon.
The {\bf leaves} (resp. {\bf  internal vertices}) of a bidirected tree $T$ are the leaves (resp.  internal vertices) of $B(T)$.

A digraph is an {\bf oriented graph} if is has no digon.
An {\bf oriented forest} is an oriented graph whose underlying graph is a forest.

The {\bf  directed path} $(x_1, \dots , x_n)$ is the oriented graph with vertex set $\{x_1, \dots , x_n\}$ and arc set $\{x_1x_2, \dots , x_{n-1}x_n\}$.

The {\bf  directed cycle} $(x_1, \dots , x_n, x_1)$ is the oriented graph with vertex set $\{x_1, \dots , x_n\}$ and arc set $\{x_1x_2, \dots , x_{n-1}x_n, x_nx_1\}$.

For short, we often abbreviate directed path into {\bf dipath} and 
directed cycle into {\bf dicycle}.

Let $D$ be a digraph. We use the following notations
\begin{itemize}
\item Given a graph $G$, we denote by $\mu(G)$ the maximum size of a matching in $G$.
We set $\pi(D) =\mu(B(D))$.

\item For any $v \in V(D)$, $d(v)$ is the number of arcs incident to $v$.
\item For any $v \in V(D)$, $n(v)$ is its number of neighbours.
\item $D[X]$ is the subdigraph of $D$ induced by $X \cap V(D)$.
\item For any $X \subseteq V(D)$, $D-X$ is the subdigraph induced by $V(D)\setminus X$. We abbreviate $D-\{x\}$ into $D-x$.
\item For any $X$ such that $X\cap V(D) =\emptyset$, $D+X$ is the digraph $(V(D)\cup X,A(D))$. We abbreviate $D+\{x\}$ into $D+x$.
\item For any $F\subseteq  \binom{V(D)}{2}$, $D\setminus F$ is the subdigraph $(V(D), A(D)\setminus F)$ and $D\cup F$ is the digraph $(V(D), A(D)\cup F)$.

\item The {\bf converse} of $D$ is the digraph $\overleftarrow{D}$ obtained by reversing the direction of all its arcs: $V(\overleftarrow{D}) = V(D)$ and $A(\overleftarrow{D})= \{yx \mid xy \in A(D) \}$.

\end{itemize}

\bigskip

\section{Properties of \texorpdfstring{$k$}{k}-dicritical digraphs}\label{sec:tools}

A graph $G$ is {\bf non-separable} if it is connected and $G-v$ is connected for all $v\in V(G)$.
A {\bf block} of a graph $G$ is a subgraph which is non-separable and is maximal with respect to this property. 
A {\bf block} of a digraph is a block in its underlying graph.

\begin{lemma}[Neumann-Lara~\cite{Neu82}]\label{lemma:degeneracy}
If $D$ is a $k$-dicritical digraph, then for every vertex $v$,
$d^+(v) \geq k-1$ and $d^-(v) \geq k-1$.
\end{lemma}

\begin{theorem}[Bang-Jensen et al.~\cite{bang2019haj}]
\label{thm:directed_gallai_subdigraph}
If $D$ is a $k$-dicritical digraph, then any block
of the subdigraph induced by vertices of degree $2(k-1)$
is either:
\begin{itemize}
\item an arc, or
\item a directed cycle, or
\item a bidirected odd cycle, or
\item a bidirected complete graph. 
\end{itemize}
\end{theorem}

Forest and trees will be important in the proof because of the followings two lemmas.

\begin{lemma}\label{lem:forestB}
Let $D$ be a $3$-critical digraph.
If $D$ is not a bidirected odd cycle, then $B(D)$ is a forest. 
\end{lemma}
\begin{proof} 
Assume for contradiction that $D$ is a $3$-critical digraph, $D$ is not a bidirected odd cycle, and $D$ contains a bidirected cycle $C$. 
Since bidirected cycles of odd length are $3$-critical, $C$ must have  even length. 
 Let $xy$ be an arc in $C$. As $D$ is $3$-dicritical,
$D\setminus xy$ has a $2$-dicolouring $\phi$. 
Since $x$ and $y$ are linked by a bidirected path of odd length,  
we have $\phi(x) \neq \phi(y)$.
Thus $\phi$ is a $2$-dicolouring of $D$, a contradiction.
\end{proof}

\begin{lemma}\label{lemma:vertex_no_almost_only_digons}
Let $D$ be a $k$-dicritical digraph.
There is no vertex $v$ in $D$ with only one neighbour not connected
to $v$ by a digon.
\end{lemma}

\begin{proof}
Suppose for a contradiction that such a vertex $v$ exists and
let $w$ be its unique neighbour not connected to $v$ by a digon.
Let $D'$ be obtained from $D$ by removing the (unique) arc connecting $v$ and $w$.
As $D$ is $3$-dicritical, $D'$ has a $(k-1)$-dicolouring $\phi$.
In this dicolouring, all neighbours of $v$ in $D'$ have a colour different from $\phi(v)$ because they are connected to $v$ by a digon.
Hence adding the arc between $v$ and $w$ cannot create a monochromatic dicycle and thus $\phi$ is a $(k-1)$-dicolouring of $D$, a contradiction. 
\end{proof}

\section{Dicritical oriented graphs with few arcs}\label{sec:lower}

\subsection{\texorpdfstring{$3$}{3}-dicritical oriented graphs with few arcs}\label{sec:lower3}

The {\bf knob of height 1} is the tournament $\vec{K}_1$ defined by:
\begin{eqnarray*}
V(\vec{K}_1) & = &  \{x_1, x_2, y_1, y_2, y_3\}, \\ 
A(\vec{K}_1) & = &  \{x_1x_2, y_1y_2, y_2y_3, y_3y_1,  y_1x_1, y_2x_1, y_3x_1, x_2y_1, x_2y_2, x_2y_3\}.
\end{eqnarray*}
Hence $(y_1, y_2, y_3, y_1)$ and $(x_1, x_2, y_i, x_1)$ for $i\in [3]$ are directed $3$-cycles.
The {\bf base} of the knob is the arc $x_1x_2$.

For all integers $i\geq 2$, the {\bf knob of height $i$}, denoted by $\vec{K}_i$, is the oriented graph obtained from $\vec{K}_{i-1}$ by adding two new vertices $z_1z_2$ and the arcs of the two directed $3$-cycle $(z_1, z_2, x, z_1)$ for all end-vertex $x$ of the base of $\vec{K}_{i-1}$.
The {\bf base} of $\vec{K}_i$ is the arc $z_1z_2$.

We also define a knob with an even number of vertices as follows.
Let $\vec{K}'_1$ be the oriented graph defined by
\[
\begin{split}
    V(\vec{K}'_1) &= \{x_1,x_2,y_1,y_2,y_3,y_4\} \\
    A(\vec{K}'_1) &= \{x_1x_2, y_1y_2,y_2y_3,y_3y_4,y_4y_1, y_1x_1,y_2x_1,y_3x_1,y_4x_1, x_2y_1,x_2y_2,x_2y_3,x_2y_4\} \\
\end{split}
\]
and we call the arc $x_1x_2$ its {\bf base}.
Informally, this is a knob $\vec{K}_1$ where the $3$-cycle $(y_1,y_2,y_3,y_1)$
is replaced by a $4$-cycle $(y_1,y_2,y_3,y_4,y_1)$.

A {\bf knob} is either $\vec{K}'_1$ or 
a knob of height $i$ for some positive integer $i$.
The following proposition is easy and the proof is left to the reader.
\begin{proposition}\label{prop:knob}
Let $\vec{K}$ be a knob.
\begin{itemize}
\item[(i)] Every precolouring of the two vertices of its base can be extended into a $3$-dicolouring of $\vec{K}$.
\item[(ii)] In every $2$-dicolouring of $\vec{K}$, the two end-vertices of its base are coloured differently. 
\item[(iii)] For every arc $a\in A(\vec{K})$, there is a $2$-dicolouring of $\vec{K}\setminus a$ such that the two end-vertices of the base are coloured the same. 
\end{itemize}
\end{proposition}

Let ${\cal O}_3$ be the family of the oriented graphs that are obtained from an odd directed cycle by adding a copy of a knob with base $a$ for every arc $a$ of this cycle. See Figure~\ref{fig:3critical_ad_5}.

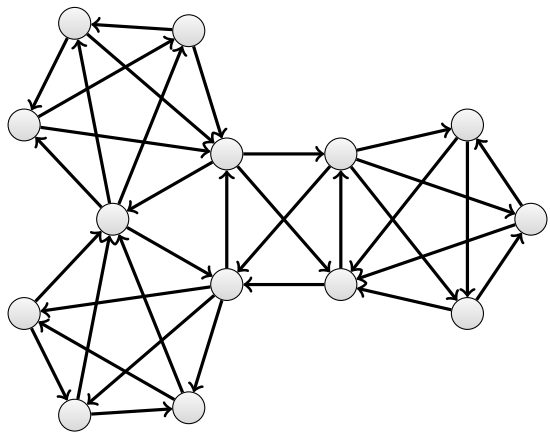
\begin{figure}[ht]
\begin{center}
\begin{tikzpicture}[scale=1]
	\tikzstyle{vertexX}=[circle,draw, top color=gray!5, 
	    bottom color=gray!30, minimum size=12pt, scale=1, inner sep=0.5pt]
	\tikzstyle{arc}=[->, very thick]
	\node (a) at (60:1) [vertexX] {};
	\node (b) at (180:1) [vertexX] {};
	\node (c) at (300:1) [vertexX] {};

    \node (xc) at (90:2.5) [vertexX] {};
    \node (yc) at (120:3) [vertexX] {};
    \node (zc) at (150:2.5) [vertexX] {};

\begin{scope}[xshift=1.5cm]
    \node (a1) at (60:1) [vertexX] {};
	\node (c1) at (300:1) [vertexX] {};
    \node (xb) at (330:2.5) [vertexX] {};
    \node (yb) at (360:3) [vertexX] {};
    \node (zb) at (390:2.5) [vertexX] {};
\end{scope}    

    \node (za) at (-90:2.5) [vertexX] {};
    \node (ya) at (-120:3) [vertexX] {};
    \node (xa) at (-150:2.5) [vertexX] {};
    
    \draw[arc] (a) to (b);
    \draw[arc] (b) to (c);
    \draw[arc] (c) to (a);
    \draw[arc] (c1) to (a1);
    
    \draw[arc] (a) to (a1);
    \draw[arc] (a) to (c1);
    \draw[arc] (a1) to (c);
    \draw[arc] (c1) to (c);
    
    \draw[arc] (xc) to (a);
    \draw[arc] (yc) to (a);
    \draw[arc] (zc) to (a);
    \draw[arc] (b) to (xc);
    \draw[arc] (b) to (yc);
    \draw[arc] (b) to (zc);

    \draw[arc] (xa) to (b);
    \draw[arc] (ya) to (b);
    \draw[arc] (za) to (b);
    \draw[arc] (c) to (xa);
    \draw[arc] (c) to (ya);
    \draw[arc] (c) to (za);

    \draw[arc] (xb) to (c1);
    \draw[arc] (yb) to (c1);
    \draw[arc] (zb) to (c1);
    \draw[arc] (a1) to (xb);
    \draw[arc] (a1) to (yb);
    \draw[arc] (a1) to (zb);
    
    \draw[arc] (xa) to (ya);
    \draw[arc] (ya) to (za);
    \draw[arc] (za) to (xa);

    \draw[arc] (xb) to (yb);
    \draw[arc] (yb) to (zb);
    \draw[arc] (zb) to (xb);
    
    \draw[arc] (xc) to (yc);
    \draw[arc] (yc) to (zc);
    \draw[arc] (zc) to (xc);
\end{tikzpicture}
\end{center}
\caption{\label{fig:3critical_ad_5}A digraph of ${\cal O}_3$ of order 14.}
\end{figure}

Since there are knobs of any odd order at least $5$, and $K'_1$ has order 6, there are elements of ${\cal O}_3$ of every order at least $12$.

\begin{proposition}\label{prop:O}
If $D \in {\cal O}_3$, then $D$ is $3$-dicritical.
\end{proposition}
\begin{proof}
Let $D$ be an oriented graph in ${\cal O}_3$ constructed on an odd directed cycle $\vec{C}$.
It is clear from the construction that $m(D) = \frac{5}{2} n(D)$.

 Using a proper $3$-colouring of $\vec{C}$ and extending to each knob by Proposition~\ref{prop:knob}~(i), we obtain a $3$-dicolouring of $D$.
 Suppose for a contradiction that $D$ admits a $2$-dicolouring.
 Then there is an arc $a$ of the directed cycle with both end-vertices coloured the same.
 Thus the knob with base $a$ contradicts Proposition~\ref{prop:knob}~(ii).
 Consequently, $\dic(D) =3$.
 
 Consider now an arc $a$ of $D$.
 A $2$-dicolouring of $D\setminus a$ can be obtained as follows.
 Take a $2$-dicolouring of the knob from which $a$ is removed such that the end-vertices of its base are coloured the same (it exists by Proposition~\ref{prop:knob}~(iii)), and $2$-dicolourings of the other knobs (for which the end-vertices of the base are coloured differently), and make sure that they agree on the vertices of the directed cycle. This is possible because $\vec{C}$ is odd. Then the union of those $2$-dicolourings is a $2$-dicolouring of $D\setminus a$.

Altogether, this proves that $D$ is $3$-dicritical.
\end{proof}

With the observation that for every even integer $n\geq 12$, there exists
a digraph in $\mathcal{O}_3$ with $n$ vertices and $\frac{5}{2}n$ arcs
(by taking the three knobs of the form $\vec{K}_i$),
and that for every odd integer $n \geq 13$, there exists a digraph in
$\mathcal{O}_3$ with $n$ vertices and $\frac{5n+1}{2}$ arcs
(by taking two knobs of the form $\vec{K}_i$, and for the last one
$\vec{K}'_1$), we deduce the following corollary.

\begin{corollary}\label{coro:construction_3_dicritical_odd_and_even_order}
For every integer $n\geq 12$, there is a $3$-dicritical oriented graph
with $n$ vertices and $\lceil \frac{5}{2}n\rceil$ arcs.
In other words, $o_3(n) \leq \lceil \frac{5}{2}n \rceil$ for every $n \geq 12$.
\end{corollary}

\subsection{\texorpdfstring{$k$}{k}-dicritical oriented graphs with few arcs}\label{sec:lowerk}

In~\cite{HoKa15}, the authors conjectured that for any $k$-dicritical oriented
graph $D$, $\frac{m(D)}{n(D)} \geq \frac{k^2}{2}- \bigO(k)$.
We disprove this conjecture by generalising the construction of the 
previous subsection.

\begin{theorem}\label{thm:general_construction}
For any integer $k \geq 2$, there is an integer $N_k$ such that
for every $n \geq N_k$, there exists a $k$-dicritical oriented
graph with $n$ vertices and $m$ arcs,
such that $m \leq (2k - 3) n$. Moreover, this last inequality is strict if $k>2$. 
\end{theorem}

First we need to define a generalisation of the knob:
for any oriented graph $D$, we define the {\bf $D$-knob with base $z_1z_2$},
denoted by $\vec{K}(D)$,
to be the oriented graph with vertex set $V(D) \cup \{z_1,z_2\}$ and 
arc set $A(D) \cup \{z_1z_2\} \cup \{z_2u, uz_1 \mid u \in V(D)\}$.
For example, $\vec{K}_1$ is the $\vec{C}_3$-knob, 
and $\vec{K}'_1$ is the $\vec{C}_4$-knob, where $\vec{C}_k$ is the directed cycle of order $k$ for every $k\geq 3$.

\begin{lemma}\label{lem:general_knob}
Let $D$ be a $k$-dicritical oriented graph and $\vec{K}(D)$ be the $D$-knob
with basis $z_1z_2$.
\begin{enumerate}[label=(\roman*)]
    \item $\dic(\vec{K}(D)) =k$, and  $z_1$ and $z_2$
        receive different colours in every $k$-dicolouring of $\vec{K}(D)$. \label{item:knob_a_b_different_colour}
    \item For any arc $a$ of $\vec{K}(D)$, there exists a $k$-dicolouring of
        $\vec{K}(D)\setminus a$ where $z_1$ and $z_2$ receive the same colour.
        \label{item:knob_critical}
\end{enumerate}
\end{lemma}

\begin{proof}
(i) $\dic(\vec{K}(D)) \geq \dic(D) =k$. Now every $k$-dicolouring of $D$ can be extended into a $k$-dicolouring of $\vec{K}(D)$ by assigning to $z_1$ and $z_2$ two distinct colours because all directed cycles containing one of those vertices contains the arc $z_1z_2$. Hence $\dic(\vec{K}(D)) =k$.

Suppose for contradiction that $\vec{K}(D)$ has a $k$-dicolouring $\phi$
with $\phi(z_1)=\phi(z_2)$. Without loss of generality, $\phi(z_1)=\phi(z_2)=k$. Then for any vertex $u$ in the copy of $D$,
$(u,z_1,z_2,u)$ is a directed $3$-cycle, so $\phi(u) \neq \phi(z_1)=\phi(z_2)=k$.
Thus $\phi$ induces a $(k-1)$-dicolouring of $D$, contradiction.

\medskip
(ii) Let $a$ be an arc of $\vec{K}(D)$.
If $a=z_1z_2$, then $z_1$ and $z_2$ participate in no dicycle of $\vec{K}(D)\setminus a$.
So one can choose any $k$-dicolouring of the copy of $D$ in $\vec{K}(D)\setminus a$, and give the same colour (any of them) to $z_1$ and $z_2$. 

If $a=uz_1$ for some vertex $u$ in the copy of $D$. Then consider a $(k-1)$-dicolouring of $D-u$, 
and colour $z_1,z_2,u$ with colour $k$. This yields the desired $k$-dicolouring of $D\setminus a$.
By directional duality, we get the result if $a$ is of the form $z_2u$.

If $a$ is in the copy of $D$, then taking a $(k-1)$-dicolouring of $D\setminus a$, and colouring $z_1$ and $z_2$ with colour $k$, we get the desired $k$-dicolouring.
\end{proof}

We are now ready to prove Theorem~\ref{thm:general_construction}.

\begin{proof}[Proof of Theorem~\ref{thm:general_construction}]
For all $i \geq 1$ and all $k\geq 2$, we  construct a $k$-dicritical oriented graph $\vec{G}^i_k$ in such a way that $\vec{G}^{i+1}_k$  has one more vertex than $\vec{G}^i_k$. We first define $\vec{G}^i_2$ for all $i \geq 1$, and explain how to construct $\vec{G}^i_{k+1}$ from $\vec{G}^i_k$. We then prove that the $\vec{G}^i_k$ satisfy the statement of the theorem. 
\medskip 

For every $i \geq 1$, let $\vec{G}^i_2$ be the directed cycle of length $i+2$.  
Let $i \geq 1$ and  $k \geq 2$. Suppose $\vec{G}^i_k$ is already constructed, and define $\vec{G}^i_{k+1}$ as follows: start with a tournament $T$ on $k+1$ vertices, and  glue on every arc $z_1z_2$ of $T$ a copy of $\vec{K}(\vec{G}^1_{k})$
with base $z_1z_2$, except on a unique arc on which we  glue a copy of
$\vec{K}(\vec{G}^i_k)$.
\medskip

We now prove that $\dic(\vec{G}^i_{k+1}) \geq k+1$. 
Since $T$ has $k+1$ vertices, there is a monochromatic arc in any $k$-dicolouring of $T$. 
So, if $\vec{G}^i_{k+1}$ has a $k$-dicolouring $\phi$, then
$\vec{G}^i_{k+1}$ contains a $\vec{G}^1_{k}$-knob or $\vec{G}^i_k$-knob with base $z_1z_2$ such that
$\phi(z_1)=\phi(z_2)$, a contradiction to  Lemma~\ref{lem:general_knob}~\ref{item:knob_a_b_different_colour}. 
Hence $\dic(\vec{G}^i_{k+1}) \geq k+1$.

Let us now prove that $\vec{G}^i_{k+1}$ is $(k+1)$-dicritical. 
Consider  an arc $a$ of $\vec{G}^i_{k+1}$. A $k$-dicolouring of
$\vec{G}^i_{k+1} \setminus a$ can be obtained as follows.
Let $\vec{K}$ be the $\vec{G}^1_{k}$-knob or $\vec{G}^i_k$-knob containing $a$ and $z_1z_2$ its base.
Let $c$ be a colouring of $V(T)$ such that the $k-1$ vertices of $V(T)\setminus \{z_1,z_2\}$ receive  pairwise distinct colours from $[k-1]$ and $c(z_1)=c(z_2)=k$.
For each arc $xy$ of $A(T)\setminus z_1z_2$, take a $k$-dicolouring of the $\vec{G}^1_{k}$-knob or $\vec{G}^i_k$-knob  with base $xy$. By Lemma~\ref{lem:general_knob}~\ref{item:knob_a_b_different_colour}, the colours
of $x$ and $y$ are distinct. Permute the colours so that $x$ is coloured $c(x)$ and $y$ is coloured $c(y)$.
By Lemma~\ref{lem:general_knob}~\ref{item:knob_critical}, $\vec{K}\setminus a$
has a $k$-dicolouring such that $z_1$ and $z_2$ are coloured the same. Permute the colours so that they are coloured $k$. 
Now, the dicolourings of the $\vec{G}^1_{k}$-knobs and $\vec{G}^i_k$-knob agree with $c$ on $V(T)$.
This easily implies that the union of these dicolourings is a $k$-dicolouring of $\vec{G}^i_{k+1} \setminus a$. Hence $\vec{G}^i_{k+1}$ is $(k+1)$-dicritical.
\medskip

For all $i \geq 1$ and $k \geq 2$, let $n^i_k=n(\vec{G}^i_{k})$ and  $m^i_k=m(\vec{G}^i_{k})$. 
For $k=2$, $n(\vec{G}^i_2)= m(\vec{G}^i_2)=i+2$, and thus $\frac{m(\vec{G}^i_2)}{n(\vec{G}^i_2)}=1 = 2k-3$. 
So, to finish the proof, it suffices to prove that for all $i\geq 1$ and $k \geq 3$, $\frac{m(\vec{G}^i_{k})}{n(\vec{G}^i_{k})} < 2k-3$. 

We have the following relations: 

\[
\begin{split}
&\left\{
\begin{array}{r l l}
n^i_2 =& i+2 & \mbox{for all}~i\geq 1 \\
n^i_k =& (\binom{k}{2}-1)n^1_{k-1} + n^i_{k-1} + k & \mbox{for all}~i\geq 1~\mbox{and}~k \geq 3 \\
\end{array}
\right. \\
&\\
&\left\{
\begin{array}{r l l}
m^i_2 =& i+2 & \mbox{for all}~i\geq 1 \\
m^i_k =& (\binom{k}{2}-1)(2n^1_{k-1} + m^1_{k-1} +1)
    + (2n^i_{k-1} + m^i_{k-1}+1)& \mbox{for all}~i\geq 1~\mbox{and}~k \geq 3 \\
\end{array}
\right.
\end{split}
\]

It follows by induction that $n^i_k = n^1_k + (i-1)$ and $m^i_k = m^1_k + (2k-3)(i-1)$ for every $i \geq 1, k \geq 2$. So we can set $N_k = n^1_k$.

Let first prove $\frac{m^1_k}{n^1_k} < 2k-3$.
For $k=3$, the result is straightforward: $\frac{5}{2} < 3$.
One can easily show by induction that $n^1_k \geq 2^{k-1}$ for every 
$k \geq 2$. Now for $k \geq 4$
\[
\begin{split}
\frac{m^1_k}{n^1_k} & < \frac{m^1_k}{\binom{k}{2}n^1_{k-1}} \\
&\leq 2  + \frac{m^1_{k-1}}{n^1_{k-1}} + \frac{1}{n^1_{k-1}} \\
&\leq 2(k-3) + \frac{m^1_3}{n^1_3} + \sum_{j=3}^{k-1} \frac{1}{n^1_j} \\
&< 2(k-3) + \frac{5}{2} + \sum_{j=3}^{+\infty} 2^{-j+1}  \\
&= 2k - 6 + \frac{5}{2}+ 2^{-1} = 2k-3. \\
\end{split}
\]
It follows that for every $i \geq 1$ 
\[
\begin{split}
\frac{m^i_k}{n^i_k} 
&= \frac{m^1_k + (2k-3)(i-1)}{n^1_k + (i-1)} \\
&< \frac{(2k-3)n^1_k + (2k-3)(i-1)}{n^1_k + (i-1)} = 2k-3 \\
\end{split}
\]
as claimed.
\end{proof}

\section{Improved lower bound on \texorpdfstring{$o_k(n)$}{ok(n)}} \label{sec:lower_ok(n)}

 A  {\bf Gallai forest} is a digraph whose blocks are either a directed cycle or an arc.

\begin{lemma}\label{lem:oriented-gallai-deg}
Let $H$ be an oriented Gallai forest, then 
\[ m(H) \le \frac{3}{2}(n(H)-1).\]
\end{lemma}

\begin{proof}
We prove the result by induction. 
If $H$ is not connected, then the induction applied to each of the connected components yields the result for $H$.
Henceforth, we may assume that $H$ is connected.

If $H$ is non-separable, then $H$ is either an arc or a directed cycle of length at least $3$.
In the first case, $m(H) =1$ and $n(H)=2$ and in the second $m(H)=n(H)\geq 3$. In both cases $m(H) \le \frac{3}{2}(n(H)-1)$.

Assume now that $H$ is separable.
There exists $v$ such that $H-v$ is not connected.
Let $C_1$ be a connected component of $H-v$ and let $H_1=H[V(C_1)\cup\{v\}]$ and $H_2=H-V(C_1)$.
 By the induction hypothesis, $m(H_1)\leq \frac{3}{2}(n(H_1)-1)$ and  $m(H_2)\le \frac{3}{2}(n(H_2)-1)$. The conclusion holds since $n(H)-1 = n(H_1)-1 + n(H_2)-1$, and $m(H)=m(H_1)+m(H_2)$. 
\end{proof}

\begin{theorem}\label{theorem:improved_lower_bound_on_ok}
Let $D$ be a $k$-dicritical oriented graph where $k\ge 3$. Then 
\[
m(D) \geq \pth{ k - \frac{3}{4}-\frac{1}{4k-6}} n(D) + \frac{3}{4(2k-3)}.
\]
\end{theorem}

\begin{proof}
Let $S$ be the set of vertices $v\in V(D)$ such that $d^+(v)=d^-(v)=k-1$. 
By Theorem~\ref{thm:directed_gallai_subdigraph}, $D[S]$ is an oriented Gallai forest,
and by Lemma~\ref{lem:oriented-gallai-deg}, 
\[ m(D[S]) \leq \frac{3}{2}(|S|-1)\]
arcs. Note that $(2k-2)|S|$ is the number of arcs of $D$ incident with vertices of
S, counting those in $D[S]$ twice. Hence,
\begin{align}
m(D) \ge (2k-2)|S| - m(D[S]) \ge \pth{2k-\frac{7}{2}}|S| + \frac{3}{2}. \label{eq:arcS}
\end{align}
All vertices in $V(D-S)$ have degree at least $2k-1$, so
\begin{equation}
2m(D) \ge (2k-1)n(D-S) + (2k-2)|S| = (2k-1)n(D)-|S|. \label{eq:arcG}
\end{equation}
By doing $(2k-\tfrac{7}{2})$ \eqref{eq:arcG} + \eqref{eq:arcS}, we obtain
\begin{align*}
\pth{4k-6} m(D) &\ge (2k-1)\pth{2k-\frac{7}{2}}n(D) + \frac{3}{2}\\
&\ge \pth{4k^2-9k+\frac{7}{2}} n(D) \frac{3}{2}\\
m(D) &\ge \pth{ \frac{4k^2-9k+\frac{7}{2}}{4k-6}} n(D) + \frac{3}{2(4k-6)} \\
&\geq  \pth{k - \frac{3}{4} - \frac{1}{4k-6}} n(D) + \frac{3}{2(4k-6)}.
\end{align*}
\end{proof}


%

\section{Properties of some digraphs} \label{sec:prop}

In this section, we prove some properties of some particular digraphs that are of importance in the proof of Theorem~\ref{thm:main_potential}.
We first consider the exceptional digraphs in this theorem, namely bidirected cycles  and odd $3$-wheels. We then examine the potential of some particular digraphs that pop up in the proof.

We start with a useful lower bound on the size of a maximum matching in a tree.
This result is certainly already known, but we give its easy proof for sake of completeness.

\begin{lemma}\label{matching-tree}
Every tree $T$ with 
$n\geq 2$ vertices and $f$ leaves has a matching of order at least 
$\frac{1}{2} (n - f +1)$.
\end{lemma}
\begin{proof}
The proof is by induction on $n$.
If $n = 2,3$, then $f \geq n- 1$ and the results holds.

Now suppose that $n\geq 4$.
Let $v$ be a leaf in $T$ and $u$ its neighbour in $T$.
Note that $d(u)\ge 2$ because $n\geq 4$.

If $u$ has degree at least $3$, then $T-v$ has $f-1$ leaves.
Thus, by the induction hypothesis,
$T-v$ and so $T$ has a matching of size $\frac{1}{2}((n-1) - (f-1) + 1) = 
\frac{1}{2}(n-f+1)$ as wanted.

If $u$ has degree $2$, then $T - \{u,v\}$ is a tree with at most $f$ leaves.
Thus by the induction hypothesis, $T - \{u,v\}$ has a matching $M$ of size
$\frac{1}{2}((n-2)-f+1)$. Hence $M\cup\{uv\}$ is a matching of size
$\frac{1}{2}(n-f+1)$ in $T$.
\end{proof}

\subsection{The exceptional \texorpdfstring{$3$}{3}-dicritical digraphs}

We denote by $\oddcycle$ the class of bidirected cycles of odd length. Observe that every bidirected cycle of odd length is $3$-dicritical and has potential $1$.

\medskip

We denote by $\wheelodd$ the class of odd $3$-wheels.
Recall that an odd $3$-wheel is a digraph
obtained by connecting a vertex $c$ to a directed $3$-cycle $(x,y,z,x)$
by three bidirected odd paths that are pairwise disjoint except in $c$.
We call $c$ the {\bf center} of $D$, these three bidirected paths the {\bf spikes} of $D$,
and the directed $3$-cycle $(x,y,z,x)$ the {\bf rim} of $D$.

It is straightforward to check that any odd 3-wheel $D$ is
$3$-dicritical, that $m(D)=2n(D)+1$ and $\pi(D) = \frac{n(D)-2}{2}$.
As a consequence any digraph in $\wheelodd$ has potential
$7n(D) - 3(2n(D)+1) - 2\frac{n(D)-2}{2} = -1$.

\begin{lemma}\label{lemma:digon_3_wheel}
Let $D$ be an odd $3$-wheel with center $c$. For any digon $[x,y]$ in $D$, $\pi(D-\{x,y\}) \geq \pi(D) - 1$.
\end{lemma}
\begin{proof}
One can easily check that in an odd $3$-wheel, every digon $[x,y]$ is contained in a maximum matching $M$ of digons. Then  $M \setminus \{[x,y]\}$ is a matching of digons of $D-\{x,y\}$ of size $\pi(D)-1$.
%
%
%
%
%
%
%
\end{proof}

\begin{lemma}\label{lemma:wheel_find_odd_cycle}
Let $D$ be an odd $3$-wheel.
\begin{enumerate}
\item[(i)] For any $e\in A(D)$,  $D\setminus e$ contains some bidirected odd cycle minus one arc.
\item[(ii)] If $v$ is not the center of $D$, then $D-v$ contains a bidirected odd cycle minus one arc.
\end{enumerate}
\end{lemma}
\begin{proof}
(ii) Let $v$ be a vertex which is not the center of $D$,
in particular, $v$ is in exactly one spike $P$.
Then the union of the two other spikes induces a bidirected
odd cycle minus one arc.
This proves~(ii).

\medskip

(i) Let $e=uv \in A(D)$, and suppose without loss of generality
that $v$ is not the center of $D$. Then $D-v \subseteq D \setminus e$
contains a bidirected odd cycle minus one arc by~(ii).
This proves~(i).
\end{proof}

\subsection{Potential of some particular digraphs}

In this section, we prove that some particular digraphs have potential at most $3$.
This will result in forbidden configurations in a minimal counterexample
to Theorem~\ref{thm:main_potential}.

\medskip

A {\bf purse} is a digraph obtained from the oriented graph with vertex set $\{x,y,z, y_1, y_2\}$ and arc set $\{xy, zy, yy_1, yy_2, y_1x, y_1z, y_2x, y_2z\}$ by adding a bidirected path of odd length between $y_1$ and $y_2$ with vertices disjoint from $\{x,y,z\}$.
See Figure~\ref{fig:purse}. We say $\{x,y,z\}$ is the {\bf bottom} of the purse.

\begin{figure}[!hbtp],
\begin{center}
\begin{tikzpicture}[scale=1.25, rotate=90]

\node (x) at (0,-0.75) [vertex] {$x$};
\node (y) at (-0.25,0.5) [vertex] {$y$};
\node (z) at (0,1.75) [vertex] {$z$};
\node (y1) at (1,1) [vertex] {$y_1$};
\node (y2) at (1,0) [vertex] {$y_2$};

\draw[arc] (x) to (y);
\draw[arc] (y) -- (y1);
\draw[arc] (y) -- (y2);
\draw[arc] (y1) -- (z);
\draw[arc] (y2) -- (z);
\draw[arc] (y1) -- (x);
\draw[arc] (y2) -- (x);
\draw[arc] (z) to (y);

\node (a) at ($(y1) + (30:1)$) [vertex] {};
\node (b) at ($(a) + (-60+30:1)$) [vertex] {};
\node (c) at ($(b) + (-60-60+30:1)$) [vertex] {};
\node (d) at ($(c) + (-60-60-60+30:1)$) [vertex] {};

\draw[arc, bend right] (y1) to (a);
\draw[arc, bend right] (a) to (b);
\draw[arc, bend right] (b) to (c);
\draw[arc, bend right] (c) to (d);
\draw[arc, bend right] (d) to (y2);

\draw[arc, bend right] (a) to (y1);
\draw[arc, bend right] (b) to (a);
\draw[arc, bend right] (c) to (b);
\draw[arc, bend right] (d) to (c);
\draw[arc, bend right] (y2) to (d);

\end{tikzpicture}
\end{center}
\caption{\label{fig:purse} A purse.}
\end{figure}

\begin{lemma}\label{lemma:purse}
(i) Every purse has potential $3$.

(ii) Every $3$-dicritical digraph having  a spanning purse  has potential at most $-3$.
\end{lemma}
\begin{proof}
(i) Let $H$ be a purse with a bidirected path of length $2\ell+1$.
It has $2\ell+5$ vertices, $4\ell +10$ arcs, and $\pi(H) =\ell +1$.
Hence $\rho(H) = 7(2\ell+5) - 3(4\ell +10) - 2(\ell+1) =3$.

(ii) Let $D$ be a  $3$-dicritical digraph having a spanning purse $H$.
Since $D$ is $3$-dicritical, every vertex of $D$ have in- and out-degree  at least $2$ by Lemma~\ref{lemma:degeneracy}. 
So in $A(D)\setminus A(H)$, there is an arc leaving $x$ and an arc leaving $z$.
Hence $D$ has at least two more arcs  than $H$, and so $\rho(D) \leq \rho(H) - 6 \leq -3$.
\end{proof}

A {\bf handcuff} is a digraph obtained from the undirected graph with vertex set 
$\{x,x',y,z,z'\}$
and arc set 
$\{xy, x'y, yz,yz',z'x',zx',z'x\}$
and by adding two bidirected paths $P_1$ and $P_2$ of odd length such that $P_1$ links $z$ and $z'$, $P_2$ links $x$ and $x'$, and  $V(P_1 - \{z,z'\})$ and $V(P_2 \setminus \{x,x'\})$ may intersect but are disjoint from $\{x,x',y,z,z'\}$.

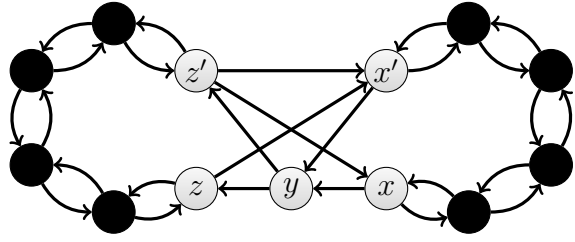
\begin{figure}[!hbtp]
\begin{center}
\begin{tikzpicture}[scale=1.25]

\node (z) at (-1,0) [vertex] {$z$};
\node (y) at (0,0) [vertex] {$y$};
\node (z') at (-1,1.25) [vertex] {$z'$};
\node (x') at (1,1.25) [vertex] {$x'$};
\node (x) at (1,0) [vertex] {$x$};

\draw[arc] (x) -- (y);
\draw[arc] (y) -- (z);
\draw[arc] (z') to (x');
\draw[arc] (y) -- (z');
\draw[arc] (x') -- (y);
\draw[arc] (z) -- (x');
\draw[arc] (z') -- (x);

\node (a) at ($(x') + (30:1)$) [maycoincide] {};
\node (b) at ($(a)  + (-60+30:1)$) [maycoincide] {};
\node (c) at ($(b)  + (-60-60+30:1)$) [maycoincide] {};
\node (d) at ($(c)  + (-60-60-60+30:1)$) [maycoincide] {};

\node (aa) at ($(z')  + (180-30:1)$) [maycoincide] {};
\node (bb) at ($(aa) + (60+180-30:1)$) [maycoincide] {};
\node (cc) at ($(bb) + (60+60+180-30:1)$) [maycoincide] {};
\node (dd) at ($(cc) + (60+60+60+180-30:1)$) [maycoincide] {};

\draw[arc, bend right] (x') to (a);
\draw[arc, bend right] (a) to (b);
\draw[arc, bend right] (b) to (c);
\draw[arc, bend right] (c) to (d);
\draw[arc, bend right] (d) to (x);

\draw[arc, bend right] (a) to (x');
\draw[arc, bend right] (b) to (a);
\draw[arc, bend right] (c) to (b);
\draw[arc, bend right] (d) to (c);
\draw[arc, bend right] (x) to (d);

\draw[arc, bend right] (z') to (aa);
\draw[arc, bend right] (aa) to (bb);
\draw[arc, bend right] (bb) to (cc);
\draw[arc, bend right] (cc) to (dd);
\draw[arc, bend right] (dd) to (z);

\draw[arc, bend right] (aa) to (z');
\draw[arc, bend right] (bb) to (aa);
\draw[arc, bend right] (cc) to (bb);
\draw[arc, bend right] (dd) to (cc);
\draw[arc, bend right] (z) to (dd);

\end{tikzpicture}
\end{center}
\caption{\label{fig:handcuff}A handcuff. The two paths of black vertices may intersect but contain no grey vertices. 
} 
\end{figure}


\begin{lemma}\label{lemma:handcuff}
Every handcuff has a subdigraph with potential at most $-2$.
\end{lemma}
\begin{proof}
Let $H$ be a handcuff. 
Let $P_1$ (resp. $P_2$) be the bidirected path of $H$ between $z$ and $z'$ (resp. $x$ and $x'$).
For $i\in \{1,2\}$, let $2\ell_i+1$ be the length of $P_i$.

Assume first that $P_1$ and $P_2$ do not intersect. Then $H$ has $2\ell_1 +2\ell_2+5$ vertices, $4\ell_1+4\ell_2 +11$ arcs, and $\pi(H) =\ell_1 +\ell_2 +2$. Hence $\rho(H) = 7(2\ell_1 +2\ell_2+5) - 3(4\ell_1 +4\ell_2 +11) - 2(\ell_1+\ell_2 +2) =-2$.

Assume now that the two bidirected paths intersect.
Then $H$ has a subdigraph $T$ which is a bidirected tree containing $x,\ x',\ z$ and $z'$
and with leaves in $\{x,x',z,z'\}$. By Lemma~\ref{matching-tree},
and because $T$ has at most $4$ leaves, $\pi(T) \geq \frac{1}{2}(n(T)-3)$.
Let $H'$ be the handcuff $H[V(T)\cup\{y\}]$.
It has $n(T) +1$ vertices, at least $2n(T) +5$ arcs and 
$\pi(H')\geq\pi(T) \geq n(T) - 3/2$.
Hence $\rho(H') \leq 7(n(T)+1) - 3(2n(T)+5) - 2((n(T)-3)/2) = -5$.
%
\end{proof}

A {\bf basket} is a digraph obtained from the oriented graph with vertex set $\{x_1, x_2,y,y_0, y_1, y_2\}$ and arc set $\{x_1y, x_2y, yy_0, yy_1, yy_2, y_1x_1, y_2x_2, y_0x_1, y_0x_2\}$
by adding a bidirected path of odd length between $y_0$ and $y_1$ and a bidirected path of odd length between $y_0$ and $y_2$. 
Those two bidirected paths may intersect, but they are always disjoint from $\{x_1,y,x_2\}$,
and $y_0,\ y_1,\ y_2,\ x_1,\ x_2$ are pairwise distinct.
See Figure~\ref{fig:basket}. 


\begin{figure}[!hbtp]
\begin{center}

\begin{tikzpicture}[scale=1.25]

\node (x) at (0,0) [vertex] {$y$};
\node (y1) at (-1,0) [vertex] {$x_1$};
\node (y2) at (1,0) [vertex] {$x_2$};
\node (z') at (0,1) [vertex] {$y_0$};

\draw[arc] (y1) -- (x);
\draw[arc] (y2) -- (x);
\draw[arc] (x) -- (z');
\draw[arc] (z') -- (y1);
\draw[arc] (z') -- (y2);

\node (y1in) at ($(y1) + (-0.66,1)$) [vertex] {$y_1$};
\node (y2in) at ($(y2) + (0.66,1)$) [vertex] {$y_2$};

\draw[arc] (y1in) -- (y1);
\draw[arc] (y2in) -- (y2);

\node (d1) at ($(y1in) + (0,1)$) [maycoincide] {};
\node (d2) at ($(d1) + (1,0)$) [maycoincide] {};
\node (d3) at ($(y2in) + (0,1)$) [maycoincide] {};
\node (d4) at ($(d3) + (-1,0)$) [maycoincide] {};

\draw[arc, bend right] (y1in) to (d1);
\draw[arc, bend right] (d1) to (d2);
\draw[arc, bend right] (d2) to (z');
\draw[arc, bend right] (d1) to (y1in);
\draw[arc, bend right] (d2) to (d1);
\draw[arc, bend right] (z') to (d2);

\draw[arc, bend right] (y2in) to (d3);
\draw[arc, bend right] (d3) to (d4);
\draw[arc, bend right] (d4) to (z');
\draw[arc, bend right] (d3) to (y2in);
\draw[arc, bend right] (d4) to (d3);
\draw[arc, bend right] (z') to (d4);

\draw[arc] (x) -- (y1in);
\draw[arc] (x) -- (y2in);

\end{tikzpicture}
\end{center}
\caption{\label{fig:basket} A basket. The two bidirected paths with black internal vertices may intersect, but they are always disjoint from $\{x_1,y,x_2\}$,
and $y_0,\ y_1,\ y_2,\ x_1,\ x_2$ are pairwise distinct.}
\end{figure}
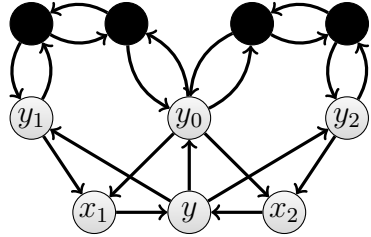

\begin{lemma}\label{lemma:basket}
Let $D$ be a $3$-dicritical digraph with a basket $H$ as a subdigraph.
Then $H$ has a subdigraph $H'$ such that
\begin{enumerate}[label=(\roman*)]
    \item $H'$ has potential at most $2$, and
    \item if $H'$ spans $D$, then $D$ has potential at most $-4$.
\end{enumerate}


\end{lemma}

\begin{proof}




(i) Let $T$ be a minimal bidirected tree included in $H$ containing $y_1,y_0,y_2$, and consider $H'$ the subdigraph of $H$ induced by $V(T) \cup \{x_1,x_2,y\}$.
Observe that the set of leaves of $T$ included in $\{y_0,y_1,y_2\}$, so by Lemma~\ref{matching-tree}, $T$ has a matching of size at least $\frac{1}{2}(n(T) -2)$.
Then $n(H') = n(T) + 3$, $m(H') \geq m(T) + 9 = 2n(T) +7$, and 
$\pi(H') \ge \pi(T) \geq  \frac{1}{2}(n(T) -2)$. As a consequence
$\rho(H') \leq 7(n(T)+3) - 3(2n(T)+7) - 2(n(T)-2)/2 = 2$.

\medskip

(ii) Suppose that $H'$ spans $D$.
Since $D$ is $3$-dicritical, every vertex has in- and out-degree at least $2$ in $D$. So in $A(D)\setminus A(H')$, there is an arc leaving $x_1$ and an arc leaving $x_2$.
Hence $D$ has at least two arcs more than $H'$, so $\rho(D) \leq \rho(H') -6 \leq -4$.
\end{proof}

A {\bf bag} is a digraph obtained from the oriented graph with
vertex set $\{y,x_1,x_2,y_1,y_2,y_3,y_4\}$ and arc set
$\{x_1y,x_2y,yy_1,yy_2,yy_3,yy_4,y_1x_1,y_2x_1,y_3x_2,y_4x_2\}$
by adding a bidirected path of odd length between $y_1$ and $y_2$,
and a bidirected path of odd length between $y_3$ and $y_4$.
These two paths may intersect, but are always disjoint from $x_1,\ y,\ x_2$ and $y_1,\ y_2,\ y_3,\ y_4$ are always pairwise distinct. 
See Figure~\ref{fig:bag}.

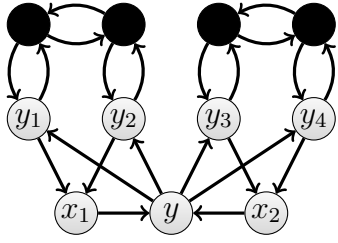
\begin{figure}[!hbtp]
\begin{center}

\begin{tikzpicture}[scale=1.25]

\node (x) at (0,0) [vertex] {$y$};
\node (y1) at (-1,0) [vertex] {$x_1$};
\node (y2) at (1,0) [vertex] {$x_2$};
\node (z') at (-0.5,1) [vertex] {$y_2$};

\node (z'') at (0.5,1) [vertex] {$y_3$};

\draw[arc] (y1) -- (x);
\draw[arc] (y2) -- (x);
\draw[arc] (x) -- (z');
\draw[arc] (x) -- (z'');
\draw[arc] (z') -- (y1);
\draw[arc] (z'') -- (y2);

\node (y1in) at ($(z') + (-1,0)$) [vertex] {$y_1$};
\node (y2in) at ($(z'') + (1,0)$) [vertex] {$y_4$};

\draw[arc] (y1in) -- (y1);
\draw[arc] (y2in) -- (y2);

\node (d1) at ($(y1in) + (0,1)$) [maycoincide] {};
\node (d2) at ($(d1) + (1,0)$) [maycoincide] {};
\node (d3) at ($(y2in) + (0,1)$) [maycoincide] {};
\node (d4) at ($(d3) + (-1,0)$) [maycoincide] {};

\draw[arc, bend right] (y1in) to (d1);
\draw[arc, bend right] (d1) to (d2);
\draw[arc, bend right] (d2) to (z');
\draw[arc, bend right] (d1) to (y1in);
\draw[arc, bend right] (d2) to (d1);
\draw[arc, bend right] (z') to (d2);

\draw[arc, bend right] (y2in) to (d3);
\draw[arc, bend right] (d3) to (d4);
\draw[arc, bend right] (d4) to (z'');
\draw[arc, bend right] (d3) to (y2in);
\draw[arc, bend right] (d4) to (d3);
\draw[arc, bend right] (z'') to (d4);

\draw[arc] (x) -- (y1in);
\draw[arc] (x) -- (y2in);

\end{tikzpicture}
\end{center}
\caption{\label{fig:bag} A bag. The two bidirected paths with black internal vertices may intersect, but are always disjoint from $x_1,\ y,\ x_2$ and $y_1,\ y_2,\ y_3,\ y_4$ are always pairwise distinct.}
\end{figure}

\begin{lemma}\label{lemma:bag}
Let $D$ be a $3$-dicritical digraph with a bag $H$ as a subdigraph. Then
$H$ has a subdigraph $H'$ such that
\begin{itemize}
\item[(i)] $H'$ has potential at most $3$, and
\item[(ii)]  if $H'$ spans $D$, then $D$ has potential at most $-3$.
\end{itemize}
\end{lemma}

\begin{proof}
(i) Let $P_1$ be the bidirected path between $y_1$ and $y_2$, and $P_2$ the bidirected path between $y_3$ and $y_4$.
For $i\in [2]$, let $2\ell_i+1$ be the length of $P_i$.

Assume first  that $P_1$ and $P_2$ do not intersect.
Then $n(H) = 2\ell_1 + 2\ell_2 +7$, 
$m(H) = 4\ell_1 + 4\ell_2 + 14$ and $\pi(B) = \ell_1 + \ell_2+2$.
Hence $\rho(H) = 7(2\ell_1 + 2\ell_2 +7) - 3(4\ell_1 + 4\ell_2 + 14) - 
2(\ell_1 + \ell_2+2) = 49 - 42 - 4 = 3$.

Assume now that $P_1$ and $P_2$ intersect.
Let $T$ be a minimal bidirected tree included in $H$ that contains $y_1,y_2,y_3,y_4$.
Then the leaves of $T$ are in $\{y_1,y_2,y_3,y_4\}$, and so by Lemma~\ref{matching-tree}, $T$ has a matching of size at least $n(T)/2 - 3/2$.
Let $H'$ be the subdigraph of $H$ induced by $V(T) \cup \{x_1, x_2, y\}$.
Then $n(H') = n(T) + 3$, $m(H') \geq m(T) + 10 = 2n(T) +8$, and $\pi(H') \geq \pi(T) \geq  n(T)/2  -3/2$.
Thus $\rho(H') \le 7 \times (n(T)+3) - 3\times (2n(T)+8) - 2 ( n(T)/2  -3/2) = 21 -24 +3 = 0$.



\medskip

(ii) Suppose that $H'$ spans $D$.
Every vertex has in- and out-degree at least $2$ in $D$. So in $A(D)\setminus A(H')$, there is an arc leaving $x_1$ and an arc leaving $x_2$.
Hence $D$ has at least two arcs more than $H'$, so $\rho(D) \leq \rho(H') -6 \leq -3$.
\end{proof}

A {\bf turtle} is a digraph obtained from the oriented graph
with set of vertices $\{y,x_1,x_2,z_1,z_2,z_3,z_4\}$ and set of arcs
$\{ x_1y, x_2y, yz_1, yz_2, yz_3, yz_4, z_2x_2, z_3x_2, z_4x_2, z_1x_1,
z_2x_1\}$ by
\begin{enumerate}[label=(\roman*)]
\item adding a bidirected path $P$ of odd length between $z_1$
     and $z_2$, and
\item adding a directed $3$-cycle $(u_2,u_3,u_4,u_2)$, and
\item for $i=2,3,4$, adding a bidirected path $P_i$ of even length
    with extremities $z_i$, $u_i$.
\end{enumerate}
Note that the bidirected paths $P_2,\ P_3,\ P_4$  and $P$ may intersect 
but they are always disjoint from $\{x_1,y,x_2\}$, and
$y,\ x_1,\ x_2,\ z_1,\ z_2,\ z_3,\ z_4$ are always pairwise distinct.
Moreover, $P_2$ (resp. $P_3$, $P_4$) may have length $0$, and in this case $u_2=z_2$ (resp. $u_3=z_3$, $u_4=z_4$).
See Figure~\ref{fig:turtle}.

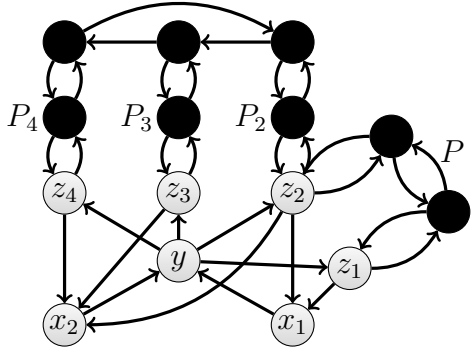
\begin{figure}[!hbtp]
\begin{center}

\begin{tikzpicture}[scale=1, rotate=180]

\node (x1) at (-1.5,0.75) [vertex] {$x_1$};
\node (x2) at (1.5,0.75) [vertex] {$x_2$};

\node (y) at (0,-0.1) [vertex] {$y$};

\node (z1) at (-2.25,0) [vertex] {$z_1$};
\node (z2) at (-1.5,-1) [vertex] {$z_2$};
\node (z3) at (0,-1) [vertex] {$z_3$};
\node (z4) at (1.5,-1) [vertex] {$z_4$};

\draw[arc] (x1) to (y);
\draw[arc] (x2) to (y);

\draw[arc] (y) to (z1);
\draw[arc] (y) to (z2);
\draw[arc] (y) to (z3);
\draw[arc] (y) to (z4);


\draw[arc, bend left] (z2) to (x2);
\draw[arc] (z3) to (x2);
\draw[arc] (z4) to (x2);

\node (a2) at ($(z2)-(0,1)$) [maycoincide] {};
\node (a3) at ($(z3)-(0,1)$) [maycoincide] {};
\node (a4) at ($(z4)-(0,1)$) [maycoincide] {};

\node (b2) at ($(a2)-(0,1)$) [maycoincide] {};
\node (b3) at ($(a3)-(0,1)$) [maycoincide] {};
\node (b4) at ($(a4)-(0,1)$) [maycoincide] {};

\draw[arc, bend right] (z2) to (a2);
\draw[arc, bend right] (a2) to (z2);
\draw[arc, bend right] (z3) to (a3);
\draw[arc, bend right] (a3) to (z3);
\draw[arc, bend right] (z4) to (a4);
\draw[arc, bend right] (a4) to (z4);

\draw[arc, bend right] (b2) to (a2);
\draw[arc, bend right] (a2) to (b2);
\draw[arc, bend right] (b3) to (a3);
\draw[arc, bend right] (a3) to (b3);
\draw[arc, bend right] (b4) to (a4);
\draw[arc, bend right] (a4) to (b4);

\draw[arc] (b2) to (b3);
\draw[arc] (b3) to (b4);
\draw[arc, bend left] (b4) to (b2);

\node (label_P2) at ($(a2)+(0.55,0)$) {$P_2$};
\node (label_P3) at ($(a3)+(0.55,0)$) {$P_3$};
\node (label_P4) at ($(a4)+(0.55,0)$) {$P_4$};


\draw[arc] (z1) to (x1);
\draw[arc] (z2) to (x1);

\node (u1) at ($(z1) - (30:1.5)$) [maycoincide] {};
\node (u2) at ($(z2) - (30:1.5)$) [maycoincide] {};

\node (labelP) at ($(u2) + (-0.8,0.2)$) {$P$};

\draw[arc, bend right] (z1) to (u1);
\draw[arc, bend right] (u1) to (z1);

\draw[arc, bend right] (z2) to (u2);
\draw[arc, bend right] (u2) to (z2);

\draw[arc, bend right] (u1) to (u2);
\draw[arc, bend right] (u2) to (u1);

\end{tikzpicture}
\end{center}
\caption{A turtle. The bidirected paths $P_2,\ P_3,\ P_4$  and $P$ may intersect but they are always disjoint from $\{x_1,y,x_2\}$, and
$y$, $x_1$, $x_2$, $z_1$, $z_2$, $z_3$, $z_4$ are always distinct.}\label{fig:turtle}
\end{figure}

\begin{lemma}\label{lemma:turtle}
Let $D$ be a $3$-dicritical digraph $D$ with a turtle $H$ as a subdigraph.
Then there exists a subdigraph $H'$ of $H$ such that
\begin{enumerate}[label=(\roman*)]
\item $H'$ has potential at most $1$ or is a bidirected cycle minus
    one arc,
\item if $H'$ spans $D$, then $\rho(D) \leq -5$.
\end{enumerate}
\end{lemma}

\begin{proof}
(i)
Observe that if $P_i$ and $P_j$ have a non empty intersection
from some distinct $i,j \in \{2,3,4\}$ then $H$ contains
a bidirected cycle minus one arc.
Similarly, if $P$ intersects $P_3$ or $P_4$, then $H$ contains
a bidirected cycle minus one arc again.
Now suppose that $P_2,P_3,P_4$ are pairwise disjoint and $P$ intersects only
$P_2$. Moreover, $B(H)$ has no cycle by Lemma~\ref{lem:forestB}.
Let $P'$ be the subpath of $P$ joining $z_1$ and
the first vertex in $V(P) \cap V(P_2)$ along $P$.
Let $\ell_2,\ell_3,\ell_4,\ell'$ be the lengths of, respectively,
$P_2,P_3,P_4,P'$.

Let $H'$ be the subdigraph of $H$ induced by $V(P')\cup V(P_2)\cup V(P_3) \cup V(P_4) \cup \{x_1, x_2, y\}$. Note that $n(H')= \ell' + \ell_2 +\ell_3+\ell_4 +6$ and $m(H') = 2(\ell' + \ell_2 + \ell_3 +\ell_4) + 14 = 2n(H') +2$.

Observe that $P_2 \cup P'$ has a matching of digons
of size at least $\frac{\ell_2}{2} + \frac{\ell'-1}{2}$,
and $P_3$ (resp. $P_4$) has a matching of $\frac{\ell_3}{2}$
(resp. $\frac{\ell_4}{2}$) digons.
We deduce that $\pi(H') \geq \frac{\ell_2 + \ell_3 + \ell_4 + \ell' - 1}{2}
= \frac{n(H') - 7}{2}$.
Hence we have
\[
\begin{split}
\rho(H) &\leq 7n(H') - 3(2n(H')+2) - 2\frac{n(H') - 7}{2} \\
&= - 3 \times 2 + 7 = 1
\end{split}
\]
as claimed.

\medskip

(ii) If $H'$ spans $D$, then it is not a bidirected cycle minus one arc,
so by~(i) we have $\rho(H') \leq 1$.
Moreover, every vertex has in- and out-degree at least $2$ in $D$. So in $A(D)\setminus A(H')$, there is an arc leaving $x_1$ and an arc leaving $x_2$.
Hence $D$ has at least two arcs more than $H'$, so $\rho(D) \leq \rho(H') -6 \leq -5$.
.
\end{proof}

\section{Proof of Theorem~\ref{thm:main_potential}} \label{sec:mainproof}

The goal of this section is to prove Theorem~\ref{thm:main_potential}.

\subsection{Some properties of the potential function}

We recall the definition of the {\bf potential}:
if $D$ is a digraph and $R$ a set of vertices of $D$, we set
$\rho_D(R) = 7n(D[R]) - 3m(D[R]) - 2\pi(D[R])$.
If $R = V(D)$, we write for short $\rho(D)=\rho_D(V(D))$.

We start with an easy observation:
\begin{lemma}\label{lemma:potential_subgraph_not_induced}
Let $H$ be a subdigraph of $D$. Then $\rho(H) \geq \rho_D(V(H))$,
and if $H$ is not an induced subdigraph of $D$, then
$\rho(H) \geq \rho_D(V(H)) + 3$.
\end{lemma}

\begin{proof}
We have
\[
\begin{split}
\rho(H) &= \rho_D(V(H)) + 3(m(D[V(H)]) - m(H)) + 2(\pi(D[H]) - \pi(H)) \\
&\geq \rho_D(V(H)) + 3(m(D[V(H)]) - m(H))
\end{split}
\]
Hence $\rho(H) \leq \rho_D(V(H))$ and if $H$ is not an induced subdigraph,
then $m(D[H]) \geq m(H) +1$ and $\rho(H) \geq \rho_D(V(H)) + 3$.
\end{proof}

Let $D$ be a digraph.
Let $k\geq 2$ be an integer.
A {\bf  $k$-thread} in $D$ is a bidirected path of length $k$ whose internal vertices have degree $4$ in $D$.
The digraph obtained from $D$ by {\bf  contracting a $3$-thread} $[w,x,y,z]$ is the digraph $D'$ obtained by replacing the $3$-thread $[w,x,y,z]$ by the digon $[w,z]$, that is $D'=D-\{x,y\} \cup [w,z]$.

Recall that, given a digraph $D$, its potential is $\rho(D) = 7n(D) - 3m(D) - 2\pi(D)$.
\begin{lemma}\label{lemma:subdiv_digon}
Let $D'$ be a digraph obtained from a digraph $D$ by contracting a $3$-thread.
\begin{enumerate}
\item[(i)] $\rho(D') \geq \rho(D)$.
\item[(ii)] If $D$ is not $2$-dicolourable, then $D'$ is not $2$-dicolourable.
\end{enumerate}
\end{lemma}
\begin{proof}
Let $[w,x,y,z]$ be the $3$-thread in $D$ whose contraction results in $D'$.

Note that $n(D') = n(D) -2$ and $m(D') \leq m(D) -4$ (equality does not hold  when an arc of $[w,z]$ was already in $A(D)$).

Let us now prove $\pi(D') \leq \pi(D)-1$. 
Let $M'$ be a matching of digons in $D'$.
If it does not contain $[w,z]$, then $M'\cup \{[x,y]\}$ is a matching of digons in $D$, and if it contains  $[w,z]$ then $(M'\setminus \{[w,z]\}) \cup \{[w,x], [y,z]\}$ is a matching of digons in $D$. Hence  $\pi(D') \leq \pi(D)-1$.
Now, $n(D') = n(D) -2$, $m(D') \leq m(D) -4$ and $\pi(D') \leq \pi(D)-1$ directly imply (i).

\medskip

Let us now prove (ii).
If $D'$ has a $2$-dicolouring $\phi$, then  $\phi(w) \neq \phi(z)$, and thus, setting
 $\phi(y)=\phi(w)$ and $\phi(x) =\phi(z)$ results in a $2$-dicolouring of $D$.
\end{proof}

The potential method is based on the following definition and lemma, which, given a $3$-dicritical digraph $D$ and a set $R$ of vertices,
allow to construct a smaller $3$-dicritical digraph. 

\begin{definition}
Let $D$ be a digraph and $R \subseteq V(D)$.
If $\phi$ is a $2$-dicolouring of $D[R]$, we define $D/(R,\phi,X)$
where $X=\{x_1,x_2\}$ as the digraph obtained by contracting 
each $\phi^{-1}(i)$ into a single vertex $x_i$, adding a digon between 
$x_1$ and $x_2$, and removing loops and multiple arcs.  
\end{definition}

The following lemma roughly says that, in a digraph of dichromatic number $3$, contracting  the colour classes of a $2$-dicoloured subdigraph  results in a digraph of dichromatic number at least $3$. 

\begin{lemma}\label{lemma:contract_dont_decrease_chi}
Let $D$ be a digraph and $R \subsetneq V(D)$ such that $D[R]$ is $2$-dicolourable.
If $\vec{\chi}(D) \geq 3$, then for any $2$-dicolouring  $\phi$ of  $D[R]$,
$\vec{\chi}(D/(R,\phi,X)) \geq 3$. 
\end{lemma}

\begin{proof}
Suppose for a contradiction that $D'=D/(R,\phi,X)$ has a
$2$-dicolouring $\phi'$.
As $x_1$ is linked via a digon to $x_2$, we have $\phi'(x_1) \neq \phi'(x_2)$, so we may assume without
loss of generality that $\phi'(x_1)=1$ and $\phi'(x_2)=2$.
Define a $2$-colouring $\phi''$ of $D$ as follows:
$\phi''(v)=\phi'(v)$ if $v \not\in R$ and
$\phi''(v)=\phi(v)$ if $v \in R$.
We claim that $\phi''$ is a $2$-dicolouring of
$D$, a contradiction to the fact that $\vec{\chi}(D) \geq 3$.
Indeed, if there is a monochromatic directed cycle $C$
in $D$ coloured by $\phi''$, then $C$
must intersect both $R$ and $V(D) \setminus R$, as the restrictions
of $\phi''$ to $R$ and $V(D) \setminus R$ are $2$-dicolourings.
But then, we can contract all vertices in $C \cap R$
and we get a monochromatic directed cycle in $D'$ coloured by $\phi'$,
contradicting the fact that $\phi'$ is a $2$-dicolouring of $D'$. 
\end{proof}

\begin{lemma}\label{lemma:potential_after_contraction}
Let $D$ be a digraph, $R\subsetneq V(D)$ and $\phi$ a $2$-dicolouring of $R$. 
Let $\Tilde{D}$ be a subdigraph of $D'=D/(R,\phi,X)$, $\tilde{X} = V(\tilde{D}) \cap X$ and
$R' = (V(\Tilde{D}) \setminus \tilde{X}) \cup R$. 
Then the following holds:
\[
\rho_D(R') \leq \rho(\Tilde{D}) + \rho_D(R) - 7|\tilde{X}| + 
                3 m(\Tilde{D}[\tilde{X}]) +
                2 t -
                3\left(m(D'[V(\Tilde{D})]) - m(\Tilde{D})\right)
\] 
where $t=\pi(\Tilde{D}) + \pi(D[R]) - \pi(D[R']) \leq 2$.\\ 
Moreover, equality holds only if for every $i\in\{1,2\}$ such that $x_i \in \tilde{X}$, vertices in $\Tilde{D}-\tilde{X}$ have at most one in-neighbour and at most one out-neighbour coloured $i$.
\end{lemma}

\begin{proof}
We have
\begin{itemize}
\item $|R'| = n(\Tilde{D}) + |R| - |\tilde{X}|$.
\item $m(D[R']) \geq m(\Tilde{D})  + m(D[R]) - m(\Tilde{D}[\tilde{X}]) + (m(D'[V(\Tilde{D})]) - m(\Tilde{D}))   $
    with equality only if no multiple arc is created between vertices in $\Tilde{D}$ during the contraction, i.e. for every $i\in\{1,2\}$ 
    such that $x_i \in \tilde{X}$, vertices in $\Tilde{D}-\tilde{X}$ have at most one in-neighbour and at most one out-neighbour coloured $i$. 
\item By definition of $t$, $\pi(D[R']) = \pi(\Tilde{D}) + \pi(D[R]) - t$.
    Note that we always have $t \leq 2$ because we can construct 
    a matching in $B(D[R'])$ of size $\pi(\Tilde{D}) - 2 + \pi(D[R])$
    by taking the union of a maximum matching in $B(\Tilde{D})$ minus
    the (at most two) edges incident with a vertex of $\tilde{X}$, with a maximum  matching in $B(D[R])$.
\end{itemize}
Finally, we get the result by summing these inequalities.
\end{proof}

The previous lemma will be used as follows:

\begin{corollary}\label{coro:potential}
Let $D$ be a digraph, $R\subsetneq V(D)$ and $\phi$ a $2$-dicolouring of $R$. Let $\Tilde{D}$ be a subdigraph of $D'=D/(R,\phi,X)$, $\tilde{X} = V(\tilde{D}) \cap X$ and let
$R' = (V(\Tilde{D}) \setminus X) \cup R$.\\
If $|\tilde{X}| = 1$, then: 
$$ \rho_D(R) \geq 7 + 3\left(m(D'[V(\Tilde{D})]) - m(\Tilde{D})\right) + \rho_D(R') - \rho(\Tilde{D}) - 2t $$ 
If $|\tilde{X}| = 2$, then:
$$ \rho_D(R) \geq 8 + 3\left(m(D'[V(\Tilde{D})]) - m(\Tilde{D})\right) + \rho_D(R') - \rho(\Tilde{D}) - 2t $$
And in the latter case, equality holds only if, for $i=1,2$, vertices in 
$\Tilde{D} - X$ have at most one in-neighbour and at most one out-neighbour
coloured $i$ in $R$.
\end{corollary}

We are now ready to prove  Theorem~\ref{thm:main_potential}.

\begin{theorem*}[Theorem~\ref{thm:main_potential} restated]
If $D$ is a $3$-dicritical digraph, then
\begin{itemize}
\item $\rho(D) = 1$ if $D$ is in $\oddcycle$,
\item $\rho(D) = -1$ if $D$ is in $\wheelodd$,
\item $\rho(D) \leq -2$ otherwise.
\end{itemize}
\end{theorem*}

\subsection{Properties of a minimal counterexample}

We consider for a contradiction a minimum counterexample $D$ with respect to the number of vertices. So
$D$ is $3$-dicritical,  $D\notin \oddcycle \cup \wheelodd$, $\rho(D) \geq -1$, and for any 
$3$-dicritical digraph $D'$ with $n(D') < n(D)$, 
either $D'$ is in $\oddcycle \cup \wheelodd$ or $\rho(D') \leq -2$.
Note also that, by Lemma~\ref{lemma:subdiv_digon}, $D$ does not contain
any $3$-thread.

\begin{claim}\label{claim:high_potential_number_one}
Let $R \subseteq V(D)$ such that $3\leq |R| \leq n(D)-1$.
Then $\rho_D(R) \geq 4$ and equality 
holds only if, either there is a vertex $z$ such that $V(D) = R\cup \{z\}$ and $d(z) = 4$, or $D-R$ is a bidirected path of length $2$ whose vertices have degree $4$ in $D$. 
\end{claim}

\begin{figure}[hbtp]
\begin{center}
\begin{tikzpicture}[scale=1.5]

\node (x) at (0,0.25) [vertex] {};

\node (xi) at ($(x) + (-0.65-0.33,-1.25)$) {};
\node (xo) at ($(x) + (-0.65+0.33,-1.25)$) {};
\node (zi) at ($(x) + (+0.65-0.33,-1.25)$) {};
\node (zo) at ($(x) + (+0.65+0.33,-1.25)$) {};

\draw[arc] (xi) to (x);
\draw[arc] (x) to (xo);
\draw[arc] (zi) to (x);
\draw[arc] (x) to (zo);

\draw (0,-1) ellipse (2 and 0.75);

\node (label_D_R) at (-0.65,-1.5) {$R$};
\end{tikzpicture}
\hspace{1cm}
\begin{tikzpicture}[scale=1.5]

\node (x) at (-1,0) [vertex] {};
\node (y) at (0,0.5) [vertex] {};
\node (z) at (1,0) [vertex] {};

\node (xi) at ($(x) + (-0.4,-1)$) {};
\node (zi) at ($(z) + (-0.4,-1)$) {};
\node (xo) at ($(x) + (0.4,-1)$) {};
\node (zo) at ($(z) + (0.4,-1)$) {};

\draw[arc] (xi) to (x);
\draw[arc] (x) to (xo);
\draw[arc] (zi) to (z);
\draw[arc] (z) to (zo);

\draw[arc, bend right] (x) to (y);
\draw[arc, bend right] (z) to (y);
\draw[arc, bend right] (y) to (x);
\draw[arc, bend right] (y) to (z);

\draw (0,-1) ellipse (2 and 0.75);

\node (label_D_R) at (-0.65,-1.5) {$R$};
\end{tikzpicture}
\end{center}

\caption{\label{fig:equality_case_rho_4} The two equality cases in Claim~\ref{claim:high_potential_number_one}. Claim~\ref{claim:high_potential} will show that only the first one can happen.}
\end{figure}
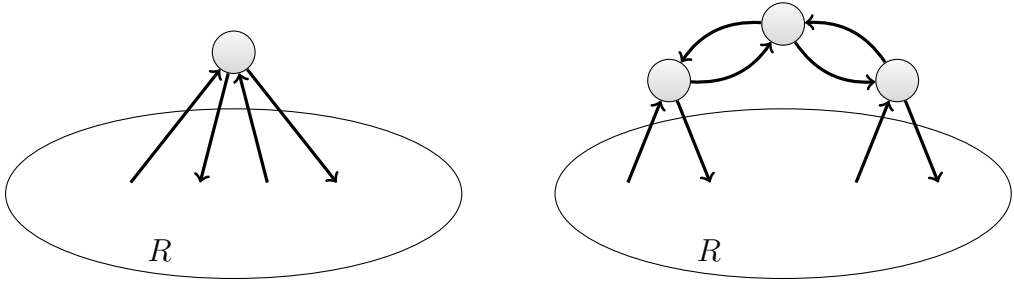

\begin{proofclaim}
We proceed by induction on $n(D) - |R|$. Suppose $|R| \leq n(D) - 1$ and that the result holds for every
set of vertices larger than $R$. Let $\phi$ be a $2$-dicolouring of $D[R]$, 
set $D' = D/(R,\phi,X)$ and let $\Tilde{D}$ be a $3$-dicritical subdigraph of $D'$ (it exists by Lemma~\ref{lemma:contract_dont_decrease_chi}).
Let $R' = (V(\Tilde{D}) \setminus X) \cup R$.
Note that $V(\Tilde{D}) \setminus X \neq \emptyset$ and thus $R \subsetneq R'$.  
Observe that either $R' = V(D)$ and then $\rho_D(R') = \rho(D) \geq -1$, or $R' \neq V(D)$ but $R'$ is larger than $R$  and thus $\rho_D(R') \geq 4$ by the induction hypothesis.
Set $t=\pi(\Tilde{D}) + \pi(D[R]) - \pi(D[R'])$. 
Remember that $t \leq 2$ by Lemma~\ref{lemma:potential_after_contraction}.
\smallskip

\noindent \textbf{Case 1:} $|V(\Tilde{D}) \cap X| = 1$. 
Observe that $t\leq 1$
 because a maximum matching of $\Tilde{D}$ has at most one digon that 
 intersects $X$.
By Corollary~\ref{coro:potential} and since $\rho_D(R') \geq -1$, we get: 
 \[
 \rho_D(R) \geq 6 - \rho(\Tilde{D}) - 2t
 \]
 Assume first that $\Tilde{D}$ is a bidirected odd cycle. 
 Then $\rho(\Tilde{D}) = 1$ and $t \leq 0$ because  $\Tilde{D}$
 has a maximum matching disjoint from $X$. So $\rho_D(R) \geq
 6-1= 5$.

 If $\Tilde{D}$ is an odd $3$-wheel, then $\rho(\Tilde{D})=-1$, so
 $\rho_D(R) \geq 6 + 1 -2 = 5$.

 Assume now that $\Tilde{D}$ is neither a bidirected odd cycle,
 nor an odd $3$-wheel. 
 Then $\rho(\Tilde{D}) \leq -2$ by minimality of $D$.  
 So $\rho_D(R) \geq 6+2-2 = 6$.
\smallskip

\noindent \textbf{Case 2}: $|V(\Tilde{D}) \cap X| = 2$.
By Corollary~\ref{coro:potential}, and since $\rho_D(R') \geq -1$, we get: 
$$ \rho_D(R) \geq 7 - \rho(\Tilde{D}) - 2t $$

Assume first that $\Tilde{D}$ is a bidirected odd cycle. In particular $\rho(\Tilde{D})=1$.
If $\Tilde{D}$ is not an induced subdigraph of $D'$, then by Corollary~\ref{coro:potential},
$\rho_D(R) \geq 10 - \rho(\Tilde{D}) - 2t \geq 10 - 1 - 2 \times 2 = 5$ as $t \leq 2$.
Now suppose $\Tilde{D}$ is an induced subdigraph of $D'$. In particular
every internal vertex of $\Tilde{D}-X$ (which is a path of even length) has degree $4$ in $D'$, and so degree $4$ in $D$ too.
By Lemma~\ref{lemma:subdiv_digon}, $D$ has no $3$-thread  and thus $\Tilde{D}$ is a bidirected $3$-cycle or a bidirected $5$-cycle. 
Then $t \leq 1$ because $\tilde{D}$ has a matching disjoint from $X$ of size $\pi(\tilde{D}) -1$. Hence we get $\rho_D(R) \geq 4$. 
Moreover, equality holds only if, $\rho_D(R') = -1$, that is 
$V(R') = V(D) = (V(\tilde{D})\setminus X) \cup R$, and no multiple arc has been created while contracting $R$. So there are at most $4$ arcs between $R$ and $V(\Tilde{D}) \setminus X$.
If  $\Tilde{D}$ is a bidirected $3$-cycle, this implies that the unique vertex $z$ of $V(\Tilde{D}) \setminus X$ has degree $4$ in $D$.
If  $\Tilde{D}$ is a bidirected $5$-cycle, this implies that all vertices of $V(\Tilde{D}) \setminus X$ have degree $4$ in $D$.

Now assume that $\Tilde{D}$ is in $\wheelodd$, then $\rho(\Tilde{D})=-1$ and by
Lemma~\ref{lemma:digon_3_wheel} we have $t \leq 1$, thus
$\rho_D(R) \geq 7 + 1 - 2 = 6$. 

Now assume that $\Tilde{D}$ is neither an odd wheel nor a bidirected odd cycle,
then $\rho(\Tilde{D})\leq -2$ and as $t \leq 2$ we get  $\rho_D(R) \geq 7 + 2 - 4 = 5$. 
\end{proofclaim}

\begin{claim}\label{claim:cycle_minus_one_arc}
Let $H$ be a digraph in 
$\oddcycle \cup \wheelodd$.
For any arc $e$ in $H$, $D$ does not contain a (not necessarily induced) copy of $H\setminus e$.
\end{claim}

\begin{proofclaim}
By Lemma~\ref{lemma:wheel_find_odd_cycle}~(i), it is enough to prove that $D$ contains no bidirected odd cycle minus one arc as a subdigraph.

Assume for a contradiction  that $D$ contains a (not necessarily induced) copy $F$ of a bidirected odd cycle minus one arc.
Let $xy$ be the arc such that $F\cup xy$ is a  bidirected odd cycle.
Set $H= D[V(F)]$. 
Observe that $m(H) \geq 2n(H)-1$ and $\pi(H) = \frac{n(H)-1}{2}$. 
Hence,  $\rho_D(H) \leq 7n(H) - 3(2n(H)-1) - 2 \frac{n(H)-1}{2} = 4$. 

If $D = H$, then $D$ has at most one more arc than $F$, for otherwise $\rho(D) \leq 4-6=-2$, a contradiction. But then,  either $D$ is a bidirected odd cycle, or $D$ is $2$-colourable, a contradiction in both cases. 

So $|V(H)| < |V(D)|$. 
By Claim~\ref{claim:high_potential_number_one}, $\rho_D(H) = 4$, which implies that $H = F$ (i.e. $F$ is an induced subdigraph of $D$) and either $(a)$ $F =D-z$ for some vertex $z$ of degree $4$
or $(b)$ $D-F$ is a bidirected path $[a,b, c]$ with $d_D(a)=d_D(b)=d_D(c)=4$. 
Observe that $F$  has a unique $2$-dicolouring $\phi$ with the following two properties:
$x$ and $y$ receive the same colour, say $1$ ;
There is a bidirected path of even length between every pair of vertices $u,v$ such that $\phi(u) \neq \phi(v)$.

Assume first that we are in case $(a)$. Then $z$ cannot be coloured $2$, so $z$ is linked via a digon to a vertex coloured $2$. Similarly, $z$ cannot be coloured $1$, but if $z$ is linked via a digon to a vertex coloured $1$, then $D$ contain a bidirected odd cycle, and if $z$ forms a directed $3$-cycle with $x$ and $y$, then $D$ is an odd $3$-wheel, a contradiction in both cases. 

Assume now that we are in case $(b)$. 
Assume without loss of generality that $a$ cannot be coloured $1$, and $c$ cannot be coloured $2$. So $c$ is linked via a digon to a vertex coloured $2$. Then $a$ cannot be linked via a digon to a vertex coloured $1$, for otherwise $D$ contains a bidirected odd cycle, so $a$ forms a directed $3$-cycle with $x$ and $y$. But then $D$ is an odd $3$-wheel, a contradiction. 
\end{proofclaim}

\begin{claim}\label{claim:no_purse_handcuff_crib_basket}
$D$ contains no purse, no handcuff, no basket, no bag, no turtle, nor any converse of those digraphs as a subdigraph.
\end{claim}

\begin{proofclaim}
Note that by directional duality, if $D$ contains no $H$ as a subdigraph then it also contains no converse of $H$ as a subdigraph. There it suffices to prove that $D$ contains no purse, no handcuff, no basket, no bag, and no turtle.

Let $H$ be a subdigraph of $D$.

If $H$ is a purse, then by Lemma~\ref{lemma:purse}~(i), 
$D$ has potential at most $3$.
Hence by Claim~\ref{claim:high_potential_number_one} this purse must be
spanning, which is impossible by Lemma~\ref{lemma:purse}~(ii).

If $H$ is a handcuff, then
by Lemma~\ref{lemma:handcuff}, a subdigraph $H'$ of $H$ has potential at most $-2$.
So by Claim~\ref{claim:high_potential_number_one}, $H'$ spans $D$,
but then $\rho(D) \leq -2$, a contradiction.


If $H$ is a basket, then by Lemma~\ref{lemma:basket}~(i),
$H$ (and thus $D$) has a subdigraph $H'$ with potential at most $2$. Hence by Claim~\ref{claim:high_potential_number_one},
$H'$ and thus $H$ spans $D$. But by Lemma~\ref{lemma:basket}~(ii), $D$ has potential
at most $-4$, a contradiction.

If $H$ is a bag, then by Lemma~\ref{lemma:bag}~(i),
$H$ has has a subdigraph $H'$ with potential at most $3$. Hence by Claim~\ref{claim:high_potential_number_one},
$H'$ and thus $H$ spans $D$. But by Lemma~\ref{lemma:bag}~(ii), $D$ has potential
at most $-3$, a contradiction.

Finally, if $H$ is a turtle, then by Lemma~\ref{lemma:turtle},
$H$ has a subdigraph $H'$ such that $\rho(H') \leq 1$ or
is a bidirected odd cycle minus one arc, but this latter case
is impossible by Claim~\ref{claim:cycle_minus_one_arc}.
Then by Claim~\ref{claim:high_potential_number_one},
$H'$ must spans $D$, and Lemma~\ref{lemma:turtle}~(ii) implies
that $\rho(D) \leq -5$, a contradiction.
\end{proofclaim}

Next claim will be used many times during the proof. 

\begin{claim}\label{claim:2colouring_minus_v_plus_e}
Let $x,y,z$ be three distinct vertices in $D$,
then $D-x \cup yz$ is $2$-dicolourable,
\end{claim}

\begin{proofclaim}
Suppose by contradiction that $D-x\cup yz$ contains a $3$-dicritical digraph $\Tilde{D}$. Then $\Tilde{D}$ is neither a bidirected odd
cycle nor a digraph in $\wheelodd$ because otherwise 
$\Tilde{D} \setminus yz \subsetneq D$ would contradict Claim~\ref{claim:cycle_minus_one_arc}. 
So $\rho(\Tilde{D}) \leq -2$ by minimality of $D$, and
thus $\rho_D(V(\Tilde{D})) \leq -2 + 3 + 2 = 3$ contradicting
Claim~\ref{claim:high_potential_number_one}.
\end{proofclaim}

\begin{claim}\label{claim:degree_4_not_3_neighbours}
If $v$ has degree $4$, then $n(v) \in \{2, 4\}$.
\end{claim}

\begin{proofclaim}
Let $v$ be a vertex of degree $4$ and
assume for a contradiction that $n(v)  = 3$. 
Set $N^+(v) = \{u,x\}$ and
$N^-(v) = \{u,y\}$.
Then, by Claim~\ref{claim:2colouring_minus_v_plus_e},
$D-v \cup yx$ has a $2$-dicolouring $\phi$.
\begin{itemize}
\item If $\phi(x) \neq \phi(y)$, then assigning to $v$ the colour
    different from $\phi(u)$ yields a $2$-dicolouring of $D$, a contradiction.
\item If $\phi(x)=\phi(y)=\phi(u)$ then, again, assigning to $v$ the colour
    different from $\phi(u)$ yields a $2$-dicolouring of $D$, a contradiction.
    \item If $\phi(x) = \phi(y) \neq \phi(u)$, then assigning to $v$ the colour $\phi(x)=\phi(y)$ yields a $2$-dicolouring of $D$, a contradiction.
    Indeed, if there were a monochromatic dicycle $C$, it would have to contain $v$.  So $C-v$ would be a monochromatic dipath from $x$ to $y$
    in $D-v$, whose union with $yx$ would be a monochromatic dicycle in $D-v \cup xy$.
\end{itemize}
\end{proofclaim}

Now we can characterize more precisely the equality case in
Claim~\ref{claim:high_potential_number_one}.

\begin{claim}\label{claim:high_potential}
Let $R \subseteq V(D)$ such that $3 \leq |R| \leq n(D)-1$.
If $\rho_D(R)=4$, then there is a vertex $z$ such that $D-z=D[R]$ and $d(z)=4$.
\end{claim}

\begin{proofclaim}
Suppose for a contradiction that $\rho_D(R)=4$ but there is no vertex $z$ such that $D-z=D[R]$ and $d(z)=4$.
Then, by Claim~\ref{claim:high_potential_number_one},
$D-R$ is a bidirected path $[x,y,z]$ such that $d(x)=d(y) =d(z)=4$.
By Lemma~\ref{lemma:subdiv_digon} $D$ contains no $3$-thread, so $n(x)=3$, a contradiction to~Claim~\ref{claim:degree_4_not_3_neighbours}.
\end{proofclaim}

\begin{claim}\label{claim:degree_4_and_out_neighbours_dominated}
There are no vertices $x,y,z$ in $D$ such that:
\begin{enumerate}
\item $\{xy,zx,zy\} \subseteq A(D)$, and
\item $d(x)=4$.
\end{enumerate}
\end{claim}
\begin{proofclaim}
Suppose that such a configuration exists. Let $y'$ be the out-neighbour
of $x$ distinct from $y$, and let $z'$ be in-neighbour
of $x$ distinct from $z$

Assume first that $z\neq y'$.
Then by~Claim~\ref{claim:2colouring_minus_v_plus_e}, $D' = D-x\cup zy'$
has a  $2$-dicolouring $\phi$.
If $z$ and $z'$ (resp. $y$ and $y'$) have the same colour, then assigning to $x$ the opposite colour yields a
$2$-dicolouring of $D$, a contradiction.
Hence, $\phi(z) \neq \phi(z')$ and $\phi(y) \neq \phi(y')$.  
Set $\phi(x)=\phi(z)$. We claim that this results in a $2$-dicolouring of $D$.
If not, there must be a monochromatic dicycle $C$ in $D$, which must go through $x$.
It must use the dipath $(z,x,y^*)$, with $y^*$ the vertex in $\{y,y'\}$ coloured $\phi(z)$.
But, since $zy^*$ is an arc in $D'$, the dicycle $C'$ obtained from $C$ by replacing the dipath $(z,x,y^*)$ by $(z,y^*)$ is a monochromatic dicycle in $D'$ coloured by $\phi$, a contradiction. 

Assume now that $z = y'$. 
Then by~Claim~\ref{claim:degree_4_not_3_neighbours}, $n(x)=2$ and so $y=z'$.
Hence $D[\{x,y,z\}]$ is either a bidirected odd cycle or a bidirected odd cycle
minus one arc, contradicting~Claim~\ref{claim:cycle_minus_one_arc}.
\end{proofclaim}

\begin{claim}\label{claim:digon_degre_5_plus}
There are no vertices $y,y_1,y_2$ in $D$ such that $d^+(y)=d^+(y_1)=2$,
$yy_1,yy_2 \in A(D)$ and $[y_1,y_2]\subset A(D)$.
\end{claim}

\begin{proofclaim}
Assume for a contradiction that there are three such vertices $y,y_1,y_2$.
If $[y,y_1]\subset A(D)$ or $[y,y_2]\subset A(D)$, then $D[\{y,y_1,y_2\}]$ is a bidirected $3$-cycle
minus at most one arc, contradicting Claim~\ref{claim:cycle_minus_one_arc}.
So we may assume that $y_1y\not\in A(D)$ and similarly $y_2y \notin A(D)$.

Let $z$ be the unique out-neighbour of $y_1$ different from $y_2$. It is
also distinct from $y$ as $y_1y \not\in A(D)$.
By Claim~\ref{claim:2colouring_minus_v_plus_e}, $D-y +zy_1$ has a 
 $2$-dicolouring $\phi$.
As $[z,y_1]$ and $[y_1,y_2]$ are digons, we can suppose without loss of generality
that $\phi(z)=\phi(y_2)=1$ and $\phi(y_1)=2$.
Then we set $\phi(y)=2$ and claim that this is $2$-dicolouring of $D$.
Indeed, assume for a contradiction that there is monochromatic cycle $C$ in $D$. It must contains $y$ and so be of colour $2$. Since $\phi(y_2) =1$, $C$ contains $y_1$. But the out-neighbours of $y_1$, namely $y_2$ and $z$ are coloured $1$, a contradiction.
\end{proofclaim}

\begin{claim}\label{claim:not_wheel_3}
There is no directed $3$-cycle $(x,y,z,x)$  such that $d^-(x)=d^-(y)=d^-(z)=2$
and $x,y$ and $z$ have a common in-neighbour $v$.
\end{claim}

\begin{proofclaim}
Suppose for a contradiction that there is a directed cycle $(x,y,z,x)$ as in the statement.
By Claim~\ref{claim:2colouring_minus_v_plus_e},
$D'=D-x \cup zy$ has a  $2$-dicolouring $\phi$.
Since  $[y,z]$ is a digon in $D'$, we may assume $\phi(y) =1$ and $\phi(z)=2$.  
If $\phi(v)=2$, then setting $\phi(x) = 1$, we obtain a $2$-dicolouring of $D$,  because the two in-neigbours of $x$ are coloured $2$.
If $\phi(v)=1$, then set $\phi(x) = 2$.  Let us prove that there is no monochromatic directed cycle in $D$. If there is one, it must contain $x$ and so be a colour $2$. But the unique in-neighbour of $x$ coloured $2$ is $z$, which has no in-neighbour coloured $2$, a contradiction. 
Hence in both cases, we obtain a $2$-dicolouring of $D$, a contradiction.
\end{proofclaim}

We denote by $\Vquatre$ the set of vertices of degree $4$ incident to no digon. 

\begin{claim}\label{claim:ind-cycle}
There is no induced dicycle $C$ in $D$ such that
\begin{enumerate}
\item for all $v \in V(C)$, $d^+(v)=2$, and
\item $C$ intersects $\Vquatre$. 
\end{enumerate}
\end{claim}
\begin{proofclaim}
Assume for a contradiction that $D$ contains an induced dicycle $C$ as in the statement.
For every $v\in V(C)$, let $v^+$ be the unique vertex such 
that $vv^+\in A(D) \setminus A(C)$. 
Since $C$ is induced, for every $v \in V(C)$,  we have $v^-, v^+ \notin V(C)$.
Let $u \in V(C)\cap \Vquatre$, and let $u^-$ be the unique vertex such 
that $u^-u\in A(D) \setminus A(C)$.

By~Claim~\ref{claim:2colouring_minus_v_plus_e}, there is a $2$-dicolouring $\phi$ of $(D-C)\cup u^-u^+$.

If the two colours appear on $\{v^+ \mid v\in V(C) \}$, that is  $\{\phi(v^+) \mid v\in V(C) \} = \{1,2\}$, then assign to $v$ the colour distinct from the one of $v^+$ (i.e.  $\phi(v) = 3-\phi(v^+)$) for all $v\in V(C)$.
This results in a $2$-dicolouring of $D$. Indeed consider a dicycle $C'$.
If $V(C')$ does not intersect $V(C)$, then it is not monochromatic, because $\phi$ is a $2$-dicolouring of $D-C$.
If $C'=C$, then it is not monochromatic because $\phi(V(C)) = \{3-\phi (v^+) \mid v\in V(C) \} = \{1,2\}$ as the two colours appear on $\{v^+ \mid v\in V(C) \}$.
If $C'\neq C$ and $V(C')$ intersects $V(C)$, then it must contain an arc $vv^+$ and this is not monochromatic.

Assume now that all the $v^+$ are coloured the same. Without loss of generality, $\phi(v^+)=1$ for all $v\in V(C)$.
Set $\phi(u)=1$ and $\phi(v)=2$ for all $v\in V(C)\setminus \{u\}$.
Let us prove that this results in a $2$-dicolouring of $D$. 
Suppose for a contradiction, that there is a monochromatic dicycle $C'$. As above,  $C'\neq C$ and $C'$ intersects $C$, and thus $C'$ contains an arc $vv^+$ for some $v \in V(C)$. This arc must be $uu^+$ because the other such arcs are not monochromatic.
Hence the vertices of $C'$ are coloured $1$, and so $C'$ must contain the dipath $(u^-,u,u^+)$.
But then the dicycle $C''$ obtained from $C'$ by replacing the subdipath $(u^-,u,u^+)$ by $(u^-,u^+)$ is a monochromatic dicycle in $(D-C)\cup u^-u^+$ coloured by $\phi$, a contradiction.
\end{proofclaim}





\begin{claim}\label{claim:V4_tree}
$D[\Vquatre]$ is an oriented forest.
\end{claim}

\begin{proofclaim}
By Theorem~\ref{thm:directed_gallai_subdigraph}, every block of $D[\Vquatre]$ is either an arc or a dicycle. By Claim~\ref{claim:ind-cycle}, $D[\Vquatre]$ has no dicycle, so all its blocks are arcs, hence it is a forest. 
\end{proofclaim}

We say that an arc $xy$ is {\bf chelou} if
$d^+(x)=d^-(y)=2$ and at least one of $y$ and $x$ is incident to no digon. 
If $xy$ is chelou, we say that $y$ (resp. $x$) is a {\bf chelou neighbour} of $x$
(resp. $y$). 

Given a chelou arc $xy$, we say that it is {\bf out-chelou} if $y$ is incident to no digon, {\bf in-chelou} if $x$ is incident with no digon, and {\bf full chelou} if it is both in- and out-chelou. In particular, any arc in $D[\Vquatre]$ is full-chelou. 

A vertex $x$ is {\bf nice} if either it is incident to at least two chelou arcs, or it is incident to one chelou arc and $d(x) \geq 5$.
A vertex $x$ is {\bf bad} if it is incident with a unique chelou arc and $d(x) = 4$. 
A nice vertex $y$ linked with a vertex $x$ via a chelou arc is said to be a {\bf nice chelou neighbour} of $x$. 

An {\bf A$^{\mbox{\bf +}}$-configuration} on an out-chelou arc $xy$  is a subdigraph which is the union of a digraph $A$ with vertex set $\{x,y,z,y_1,y_2\}$ and arc set $\{xy, zy, yy_1,yy_2, y_1z, y_2z\}$ and a bidirected path of odd length from $y_1$ to $y_2$ whose internal vertices are disjoint from $V(A)$.
See Figure~\ref{subfig:odd_cycle}.

A {\bf B$^{\mbox{\bf +}}$-configuration} on an out-chelou arc $xy$  is a subdigraph which is the union of a digraph $B$ with vertex set $\{x,y,z,y_1,y_2, y_3\}$ and arc set $\{xy, zy, yy_1,yy_2, yy_3, y_1z, y_2z, y_3z\}$, a directed cycle $C=(c_1,c_2, c_3, c_1)$ and three disjoint bidirected paths  of even length $P_i$, $i\in \{1,2,3\}$, from $y_i$ to $c_i$ whose internal vertices are disjoint from $V(B)\cup V(C)$. Each $P_i$ may have length $0$, that is $c_i$ and $y_i$ might be the same vertex.
See Figure~\ref{subfig:3_wheel}.

An {\bf A$^{\mbox{\bf -}}$-configuration} (resp. {\bf B$^{\mbox{\bf -}}$-configuration}) on an in-chelou arc $yx$ is the converse of an A$^+$-configuration (resp. B$^+$-configuration).

\begin{figure}[hbtp]

\begin{center}
\begin{subfigure}[t]{0.45\textwidth}
\begin{center}
\begin{tikzpicture}[scale=1.5]

\node (x) at (-1,0) [vertex_nice] {$x$};
\node (y) at (0,0) [vertex] {$y$};
\node (z) at (0,1) [vertex] {$z$};
\node (y1) at (1,1) [vertex] {$y_2$};
\node (y2) at (1,0) [vertex] {$y_1$};

\draw[arc_chelou] (x) -- (y);
\draw[arc] (y) -- (y1);
\draw[arc] (y) -- (y2);
\draw[arc] (y1) -- (z);
\draw[arc] (y2) -- (z);
\draw[arc] (z) -- (y);

\node (a) at ($(y1) + (30:1)$) [vertex] {};
\node (b) at ($(a) + (-60+30:1)$) [vertex] {};
\node (c) at ($(b) + (-60-60+30:1)$) [vertex] {};
\node (d) at ($(c) + (-60-60-60+30:1)$) [vertex] {};

\draw[arc, bend right] (y1) to (a);
\draw[arc, bend right] (a) to (b);
\draw[arc, bend right] (b) to (c);
\draw[arc, bend right] (c) to (d);
\draw[arc, bend right] (d) to (y2);

\draw[arc, bend right] (a) to (y1);
\draw[arc, bend right] (b) to (a);
\draw[arc, bend right] (c) to (b);
\draw[arc, bend right] (d) to (c);
\draw[arc, bend right] (y2) to (d);

\end{tikzpicture}
\end{center}
\caption{\label{subfig:odd_cycle} An A$^+$-configuration on $xy$.}
\end{subfigure}
\hspace{1cm}
\begin{subfigure}[t]{0.45\textwidth}
\begin{center}
\begin{tikzpicture}[scale=1.5]

\node (x) at (-1,0) [vertex_nice] {$x$};
\node (y) at (0,0) [vertex] {$y$};
\node (z) at (0,1) [vertex] {$z$};
\node (y3) at (1.35,1.25) [vertex] {$y_3$};
\node (y2) at (1,0.5) [vertex] {$y_2$};
\node (y1) at (1.35,-0.25) [vertex] {$y_1$};

\draw[arc_chelou] (x) -- (y);
\draw[arc] (y) -- (y1);
\draw[arc] (y) -- (y2);
\draw[arc] (y) -- (y3);
\draw[arc] (y1) -- (z);
\draw[arc] (y2) -- (z);
\draw[arc] (y3) -- (z);
\draw[arc] (z) -- (y);

\node (y1b) at ($(y1)+(1,0)$) [vertex] {};
\draw[arc, bend right] (y1) to (y1b);
\draw[arc, bend right] (y1b) to (y1);

\node (y1c) at ($(y1b)+(1,0)$) [vertex] {$c_1$};
\draw[arc, bend right] (y1b) to (y1c);
\draw[arc, bend right] (y1c) to (y1b);

\node (y2b) at ($(y2)+(1,0)$) [vertex] {};
\draw[arc, bend right] (y2) to (y2b);
\draw[arc, bend right] (y2b) to (y2);

\node (y2c) at ($(y2b)+(1,0)$) [vertex] {$c_2$};
\draw[arc, bend right] (y2b) to (y2c);
\draw[arc, bend right] (y2c) to (y2b);

\node (y3b) at ($(y3)+(1,0)$) [vertex] {};
\draw[arc, bend right] (y3) to (y3b);
\draw[arc, bend right] (y3b) to (y3);

\node (y3c) at ($(y3b)+(1,0)$) [vertex] {$c_3$};
\draw[arc, bend right] (y3b) to (y3c);
\draw[arc, bend right] (y3c) to (y3b);

\draw[arc] (y1c) -- (y2c);
\draw[arc] (y2c) -- (y3c);
\draw[arc] (y3c) -- (y1c);

\end{tikzpicture}
\end{center}
\caption{\label{subfig:3_wheel} A B$^+$-configuration on $xy$. }
\end{subfigure}
\end{center}

\caption{\label{fig:chelou_edge}The two possible configurations when $xy$ is out-chelou and  $x$ is nice.}
\end{figure}

\begin{claim}\label{claim:B+pot}
The digraph induced by a B$^+$-configuration has potential $9$.
\end{claim}
\begin{proofclaim}
Let $H$ be a  B$^+$-configuration.
For $i\in [3]$, let $2\ell_i$ be the length of $P_i$.
Then $n(H)= 2\ell_1 +2 \ell_2 +2\ell_3 +6$, $m(H) = 4\ell_1+4\ell_2 +4\ell_3 + 11$, and $\pi(H) = \ell_1+\ell_2+\ell_3$.
Hence $\rho(H) = 7( 2\ell_1 + 2\ell_2 +2\ell_3 +6) - 3( 4\ell_1+4\ell_2 +4\ell_3 + 11) - 2(\ell_1+\ell_2+\ell_3) = 9$.
\end{proofclaim}

Next claim is one of the main tools of the rest of the proof. We often use the dual proposition that is the one about in-chelou arcs obtained by reversing the direction of all arcs.

\begin{claim}\label{claim:chelou_arc_one}
Let $xy$ be an out-chelou arc in $D$. The following hold:
\begin{enumerate}[label=(\roman*)]
\item $y$ is in a directed $3$-cycle not containing $x$, and \label{it:triangle} 
\item if $x$ is nice, then there is an A$^+$-configuration or a B$^+$-configuration on $xy$. \label{it:nice}
 \end{enumerate}
\end{claim}

\begin{proofclaim}
Let $N^+(y) = \{y_1, \dots, y_l\}$ and $N^-(y)=  \{x,z\}$.
Consider $D' = D - \{x ,y\} \cup \{ zy_1, \dots, zy_{\ell}\}$. 
Since $y$ is not in a digon,  $z \notin \{y_1, \dots, y_{\ell}\}$, so $D'$ is well-defined (it has no loop). 

Let us prove that $\vec{\chi}(D') \geq 3$. Assume for a contradiction 
that there is a $2$-dicolouring $\phi$  of $D'$. 
Let $c \in \{1,2\}$ be the colour of the unique vertex in $N^+(x) \setminus\{y\}$.
We extend $\phi$ to $D$ by setting $\phi(x) = 3-c$ and $\phi(y) = c$. A monochromatic dicycle must contain $x$ or $y$.  
The colour of $x$ is distinct from the colour of its two out-neighbours, so $x$ is not in a monochromatic dicycle. If $y$ is in a monochromatic dicycle, it must contain $z$ and one of the $y_i$, which implies a monochromatic dicycle in $D'$, a contradiction. So $\vec{\chi}(D') \geq 3$.

So $D'$ contains a $3$-dicritical digraph $\Tilde{D'}$ as a subdigraph. 
Since $\tilde{D'}$ cannot be a subdigraph of $D$, $\tilde{D'}$ must contain an arc $zy_i$ for some $i$, in particular $z \in V(\tilde{D'})$. 
Let $U = V(\tilde{D'}) \cup \{y\} \subseteq V(D)\setminus \{x\}$ and let us compare $\rho_D(U)$ with $\rho(\tilde{D'})$.  
By construction of $D'$ and since the arc $zy$ is in $D[U]$ but not in $\tilde{D'}$, we have:
$$m(D[U]) - m(\tilde{D'}) \geq 1$$ 
Moreover, since a digon in $\tilde{D'}$ but not in $D$ must be incident to $z$, we have:  
$$\pi(D[U])-\pi(\Tilde{D'}) \geq -1$$
Finally, it is clear that 
$$n(D[U]) - n(\tilde{D'}) = 1 $$ 
All together, this yields: 
\begin{align}
\rho_D(U) 
&=  7n(D[U]) - 3m(D[U]) - 2\pi(D[U]) \nonumber \\
&= \rho(\Tilde{D}) + 7(n(D[U]) - n(\tilde{D'})) - 3(m(D[U]) - m(\Tilde{D})) - 2(\pi(D[U])-\pi(\Tilde{D'}))  \label{eq:potentiel} \\
&\leq \rho(\Tilde{D'}) + 7 - 3 + 2   \nonumber\\
&\leq \rho(\Tilde{D'}) + 6  \label{eq:+6}
\end{align}

By Claim~\ref{claim:high_potential_number_one},  $4 \leq \rho_D(U)$ and thus $\rho(\Tilde{D'}) \geq -2$ by Equation~\eqref{eq:+6}. 
Hence, by the induction hypothesis, we are in one of the three following cases: either $\rho(\tilde{D'}) = -2$, or $\rho(\tilde{D'}) = -1$ and $\tilde{D'} \in \wheelodd$, or $\rho(\tilde{D'}) = 1$ and $\tilde{D'}$ is a bidirected odd cycle.

\medskip
\noindent{\bf Case 1:} $\rho(\Tilde{D'}) = -2$: 

In this case, we shall prove that $x$ is bad, that is $d(x)=4$ and $x$ is incident with a unique chelou arc, and that $y$ is in a directed $3$-cycle not containing $x$.

Since $\rho(\tilde{D'}) = -2$, by Claims~\ref{claim:high_potential_number_one} and Equation~\eqref{eq:+6}, we must have $\rho_D(U) = 4$. 
Hence, we must have $\pi(D[U])- \pi(\Tilde{D'}) = -1$ by Equation~\eqref{eq:potentiel}, so adding the arcs
$zy_1, \dots, zy_l$  created at least one digon. 
Hence $y$ is in a directed $3$-cycle not containing $x$ (namely together with $z$ and one of the $y_i$).

By~Claim~\ref{claim:high_potential}, as $\rho_D(U)=4$, $D[U]=D-x$ and $d(x)=4$.
In particular, $V(\tilde{D'}) = V(D) \setminus \{x,y\}$. 

It now remains to prove that $x$ is incident with a unique chelou arc, namely $xy$. 

Let $y'$ be the out-neighbour of $x$ distinct from $y$.
Observe that $yy' \not\in A(D)$ by Claim~\ref{claim:degree_4_and_out_neighbours_dominated}.
Hence $d^-_{D}(y') = d^-_{\Tilde{D'}}(y') +1$.
But $y'$ is a vertex of $\tilde{D'}$ which is $3$-dicritical, so $d^-_{\Tilde{D'}}(y') \geq 2$.
Hence $d^-_{D}(y') \geq 3$, so $xy'$ is not chelou.
Let $w$ be an in-neighbour of $x$. If $d^-(x) \geq 3$, then $wx$ is not a chelou arc. Suppose now $d^-(x)=2$.
Note that $w\neq z$ by Claim~\ref{claim:degree_4_and_out_neighbours_dominated}.
Hence $d^+_{D}(w) = d^+_{\Tilde{D'}}(w) + 1$.
But $w$ is a vertex of $\tilde{D'}$ which is $3$-dicritical, so $d^+_{\Tilde{D'}}(w) \geq 2$.
Hence $d^+_{D}(w) \geq 3$ and $wx$ is not chelou.



\medskip
\noindent{\bf Case 2:} $\rho(\tilde{D'}) = -1$ and $\Tilde{D'} \in \wheelodd$.

In this case, we prove that there is a B$^+$-configuration (Figure~\ref{subfig:3_wheel})  on $xy$. So we only need to show that $z$ is the center of $\tilde{D'}$. 
If it is not the case, then by Lemma~\ref{lemma:wheel_find_odd_cycle}~(ii), $\tilde{D'} - z$  contains a bidirected odd cycle minus one arc, and thus $D$ contains a bidirected odd cycle minus one arc, a contradiction to Claim~\ref{claim:cycle_minus_one_arc}.

\medskip
\noindent{\bf Case 3:} $\rho(\tilde{D'}) = 1$ and $\Tilde{D'}$ is a bidirected odd cycle.

In this case, by Claim~\ref{claim:cycle_minus_one_arc}, $\Tilde{D}$ contains at least two arcs in $\{zy_1, \dots zy_{\ell}\}$, and one easily sees that there is an A$^+$-configuration (Figure~\ref{subfig:odd_cycle}) on $xy$. 
\end{proofclaim}




\begin{claim}\label{claim:path_nice}
If $y$ is a vertex in $\Vquatre$ and has a nice chelou neighbour $z$,
then $y$ is bad, i.e. it has no other chelou neighbour.
\end{claim}

\begin{proofclaim}
Suppose for contradiction that $y$ has another chelou neighbour
$x$, and assume $xy \in A(D)$ (the case where $yx \in A(D)$ is symmetric). 
Since $x$ is incident to at least one chelou arc, it is either nice or bad. 

Consider the case where $x$ is nice.
Assume first $yz \in A(D)$. Let $z'$ be the out-neighbour of $y$ distinct from $z$ and $x'$ the in-neighbour of $y$ distinct from $x$. 
By~Claim~\ref{claim:chelou_arc_one}~\ref{it:nice} applied
to the out-chelou arc $xy$, and because  $d^+(y) = 2$, there is an A$^+$-configuration on $xy$.
Similarly, by the dual of Claim~\ref{claim:chelou_arc_one}~\ref{it:nice} applied to the in-chelou arc $yz$, and because $d^-(y) = 2$, there is an A$^-$-configuration on $yz$. The union of those two configurations forms a handcuff $H$ (see Figure~\ref{fig:handcuff}),
which contradicts Claim~\ref{claim:no_purse_handcuff_crib_basket}.

Assume now $zy \in A(D)$. Set $N^+(y) = \{y_1, y_2\}$. 
By~Claim~\ref{claim:chelou_arc_one}~\ref{it:nice} applied on the two out-chelou arcs $xy$ and $zy$ and because $d^+(y) = 2$, 
there is an A$^+$-configuration on $xy$ and an A$^+$-configuration on $zy$.
The union of those two configurations form a purse $P$ (Figure~\ref{fig:purse}) with bottom $\{x,y,z\}$, which contradicts Claim~\ref{claim:no_purse_handcuff_crib_basket}.
\medskip

So we may assume that $x$ is bad, and so $x$ is in $\Vquatre$.

Assume first $zy \in A(D)$ and set $N^+(y) = \{y_1, y_2\}$. 
Note that $y$ is nice as it is incident to at least two chelou arcs. 
By~Claim~\ref{claim:chelou_arc_one}~\ref{it:nice} applied on the out-chelou
arc $zy$ and because $d^+(y) = 2$, there is an A$^+$-configuration on $zy$. Observe in particular that it implies that $N^-(x) = \{y_1, y_2\}$. 
Now, by the dual of~Claim~\ref{claim:chelou_arc_one}~\ref{it:nice} applied to the in-chelou arc $xy$, and because $y$ is nice, there is an A$^-$-configuration on $yx$.
The union of those two configurations contains the converse of a purse with bottom $\{y,x,y'\}$ where $y'$ is the out-neighbour of $x$ distinct from $y$. 
This yields a contradiction as above.

Assume now $yz \in A(D)$. 
Then by~Claim~\ref{claim:chelou_arc_one}~\ref{it:nice} applied on the in-chelou
arc $yz$, $x$ is incident to a digon, which
contradicts the fact that $x\in \Vquatre$.
\end{proofclaim}

Note that every arc with both extremities in $\Vquatre$ is full-chelou.

\begin{claim}\label{claim:nice_chelou_neighbour_of_degree_5}
A vertex of degree $5$ or $6$ incident to no digon has at most one nice chelou neighbour.
\end{claim}

\begin{proofclaim}
Let $y$ be a vertex of degree $5$ or $6$ incident to no digon. 
 Suppose for a contradiction that $y$ has two nice chelou  neighbours $x_1$ and $x_2$.
Since $y$ is incident with chelou arcs, it must have in- or out-degree $2$. Without loss of generality, we may assume $d^+(y)\in\{3,4\}$ and $d^-(y)=2$. Hence $x_1$ and $x_2$ are the in-neighbours of $y$.

By Claim~\ref{claim:chelou_arc_one}~\ref{it:nice}, there is an A$^+$- or a B$^+$-configuration on each of the out-chelou arcs
$x_1y$ and $x_2y$. 

Assume first that there is a B$^+$-configuration $H$ on $x_1 y$.
By Claim~\ref{claim:B+pot}, this configuration has potential $9$.

If the A$^+$-configuration or B$^+$-configuration on $x_2y$
contains two out-neighbours of $y$ which are already in the $B^+$-%
configuration on $x_1y$, then $D[V(H)]$ has at least two more arcs than $H$.
Hence $\rho_D(V(H)) \leq 9-6 =3$.
By Claim~\ref{claim:high_potential_number_one}, this implies that 
$H$ is a spanning subdigraph of $D$.
But as $d^+(x_1), d^+(x_2) \geq 2$, $D$ has at least two more arcs (that is at least four more than $H$).
Thus $\rho(D) \leq 3 - 3 \times 2 = -3$, a contradiction.

So we know that there is only one out-neighbour of $y$ in the intersection
of these two configurations. In particular, there in an A$^+$-configuration
on $x_2y$. But then we precisely have a turtle (see Figure~\ref{fig:turtle}), and this contradicts Claim~\ref{claim:no_purse_handcuff_crib_basket}.
As a consequence, there is no $B^+$ configuration on $x_1y$.

Similarly, we may assume that there is no B$^+$-configuration on $x_2y$.

Assume now that there is  an A$^+$-configuration $H_1$ on $x_1 y$ and an A$^+$-configuration $H_2$ on $x_2 y$.
If the out-neighbourhoods of $y$ in $H_1$ and $H_2$ have two (resp. one, no) vertices in common, then $H_1\cup H_2$ is a purse
(resp. basket, bag).
In any case, this contradicts Claim~\ref{claim:no_purse_handcuff_crib_basket}.
\end{proofclaim}


\subsection{Discharging} \label{subsec:discharging}

Recall that our goal is to get the following contradiction: $\rho(D) = 7n(D) - 3m(D) - 2\pi(D) < -1$. 
In this subsection, we are going to define an initial charge on each vertex in such a way that $\rho(D)$ is at most the sum of the charges of the vertices. A natural way to define this charge is: $w(v) = 7 - \frac{3}{2}d(v)$. 
But such a definition does not take in account the $-2\pi(D)$ that appears in the definition of the potential. 
So we rectify this charge by adding an integer  $\epsilon(v)$ (defined below) to the charge of each vertex, that depends on the number of digon incident to $v$. 
We then prove that the potential of $D$ is at most the sum of the charges (we actually prove something slightly more subtle,  see Claim~\ref{claim:epsilon_correct}). Then we define some discharging rules so that, without changing the sum of the charges, each vertex gets a new charge which is non-positive,  see Claim~\ref{claim:result_discharging}. 
Hence, at the end of this section, we  will have proved $\rho(D) \leq 0$. 
This is not enough, but we are getting closer to the goal.
\medskip

Note that for any vertex $v$, $d(v)-n(v)$ is the number of digons
incident to $v$, in other words it is the degree of $v$ in $B(D)$. 
For every vertex $v$ in $D$, we define $\epsilon(v)$ as follows:
\begin{itemize}
\item $\epsilon(v) = \frac{1}{2}$ if $d(v)-n(v)=1$,
\item $\epsilon(v) = 2 - \frac{d(v) - n(v)}{2}$ if $d(v)-n(v) \geq 2$,
\item $\epsilon(v)=0$ otherwise.
\end{itemize}

For every vertex $v$, we define its {\bf initial charge} as
$$w(v) = 7 - \frac{3}{2}d(v) - \epsilon(v).$$
Observe that the only vertices of positive initial charge are  the vertices in $\Vquatre$ which have charge equal to $1$.

We denote by $\Sigma(D)$ the total charge : 
\begin{equation}\label{eq-charge}
\Sigma(D) = \sum_{v \in V(D)} w(v) = 7n(D) - 3m(D) - \sum_{v\in V(D)} \epsilon(v)
\end{equation}

\begin{claim}\label{claim:epsilon_correct}
$\rho(D) \leq \Sigma(D) - \gamma(D)$, where $\gamma(D)$ is the number of connected components of $B(D)$ which are odd paths.
\end{claim}

\begin{proofclaim} 
By Equation~\eqref{eq-charge}, it is enough to prove that $\sum_{v\in V(D)} \epsilon(v) \leq 2\pi(D) - \gamma(D)$.
Let ${\cal C}_2$ be the set of components of $B(D)$ of size at least $2$ and let $T$ be a connected component in ${\cal C}_2$.
By Lemma~\ref{lem:forestB}, $T$ is a tree.
Let $f$ be the number of leaves in $T$ and $n_T$ its number of  vertices.
We have 
$$\sum_{v \in V(T)} \epsilon(v) = \sum_{v \in V(T)} \left(2 - \frac{1}{2}d_T(v)\right) - f  = 2n_T - (n_T-1) -f = n_T - f +1.$$
So by Lemma~\ref{matching-tree}, $\sum_{v \in V(T)} \epsilon(v) \leq 2\mu(T)$.
Moreover, if $T$ is a path
of odd length $2\ell+1$, then $\sum_{v \in V(T)} \epsilon(v) = 
\frac{1}{2} + 2\ell + \frac{1}{2} = 2\ell + 1$ and $\pi(T)= \ell+1$ so
$\sum_{v \in V(T)} \epsilon(v) = 2\mu(T) - 1$.

Now, summing over all the connected components of ${\cal C}_2$, we get $$\sum_{v\in V(D)} \epsilon(v) = \sum_{T\in {\cal C}_2} \sum_{v \in V(T)} \epsilon (v) \leq \sum_{T\in {\cal C}_2}2\mu(T) - \gamma(D) \leq 2\pi(D) - \gamma(D) $$
\end{proofclaim}

From now on, we want to prove that $\Sigma(D) - \gamma(D) < -1$ which implies $\rho(D) < -1$, a contradiction. 
In order to do so, we shall redistribute the charges using some discharging rules.

Let $u$ and $v$ be two adjacent vertices.
We say that $u$ is a {\bf simple} in-neighbour (resp. out-neighbour) of $v$ if $u$ is an  in-neighbour (resp. an out-neighbour) of $v$ and there is no digon between $u$ and $v$ (i.e. $[u,v]\not\subset A(D)$). 

We say that $u$ is a {\bf binary} in-neighbour (resp. out-neighbour) of $v$ if $u$ is an in-neighbour (resp. an out-neighbour) of $v$ and $d^-(v) = 2$ (resp. $d^+(v) = 2$). Otherwise, $u$ is a {\bf  standard} in-neighbour (resp. out-neighbour) of $v$.

If $u$ is a simple (resp. binary, standard) out-neighbour or 
simple (resp. binary, standard) in-neighbour of $v$, we sometimes  simply say that $u$ is a {\bf  simple} (resp. {\bf  binary}, {\bf  standard}) neighbour of $v$.
\medskip


We can now state our discharging rules. For each vertex $v$:
\medskip

\hfill
\begin{minipage}{0.95\textwidth}
{\it
\begin{itemize}
\item[\rm (R1):] if $d(v) \geq 5$ and $v$ is incident to at least one digon, then $v$ receives $\frac{1}{2}$ from each of its simple standard neighbour incident to no digon,
\item[\rm (R2):] if $d(v) \geq 5$ and $v$ is incident to no digon, then $v$ receives $\frac{1}{3}$ from each of its standard neighbour incident to no digon.     
    \item[\rm (R3):] if $v$ is a bad vertex (that is of degree $4$, incident to no digon, and which has a unique chelou neighbour), then $v$ receives $\frac{1}{3}$ from its unique chelou neighbour. 
\end{itemize}
}
\end{minipage}
\bigskip

Observe that a vertex of degree $4$ has only binary neighbours, in other words has no standard neighbour. So, in rules (R1) and (R2), it is redundant to ask for $v$ to be of degree at least $5$. Anyway, we leave it like this for clarity. 

These three discharging rules can equivalently be stated as follows. 
Given two adjacent vertices $x$ and $v$, we have:
\medskip 

\hfill
\begin{minipage}{0.95\textwidth}
{\it
\begin{itemize}
\item[\rm (R1):] $v$ sends $\frac{1}{2}$ to $x$ if:
\begin{itemize}
\item $d(x) \geq 5$, 
\item $x$ is incident to a digon,
\item $v$ is incident to no digon, and
\item $v$ is a simple standard neighbour of $x$ and

\end{itemize}
\item[\rm (R2):] $v$ sends $\frac{1}{3}$ to $x$ if:
\begin{itemize}
\item $d(x) \geq 5$, 
\item $x$ is incident to no digon,
\item $v$ is incident to no digon, and
\item $v$ is a simple standard neighbour of $x$. 
\end{itemize}
\item[\rm (R3):] $v$ sends $\frac{1}{3}$ to $x$ if:
\begin{itemize}
\item $x$ is bad (that is $d(x) = 4$, $x$ is incident to no digon and with a unique chelou arc),
\item $v$ is the unique chelou neighbour of $x$. 
\end{itemize}
\end{itemize}
}
\end{minipage}
\medskip

A few useful observations on these three rules: 
\begin{itemize}
\item[(OBS1)] No charge is sent via digons. 
\item[(OBS2)] Each extremity of a chelou arc  is a binary neighbour of the other extremity. Hence, no charge is sent through a chelou arc by (R1) and (R2). 
\item[(OBS3)] A vertex in $\Vquatre$ does not receive any charge, except if it is bad: then it receives $\frac{1}{3}$ from its unique chelou neighbour by (R3), and nothing else.  
\item[(OBS4)] A vertex in $\Vquatre$ sends $\frac{1}{3}$ or $\frac{1}{2}$ to each of its non-chelou neighbour. Indeed, given two vertices $x,v$ such that $d(v)=4$ and $x$ is a non-chelou neighbour of $v$, $v$ must be a standard neighbour of $x$, and thus $v$ sends $\frac{1}{2}$ or $\frac{1}{3}$ to $x$ by (R1) or (R2), depending whether $x$ is incident to a digon or not. 
\end{itemize}


Let $w^*(v)$ be the new charge of a vertex $v$ after performing these rules.

\begin{claim}\label{claim:result_discharging}
Let $v$ be a vertex of $D$. 
The following hold:
\begin{enumerate}[label=(\roman*)]
\item $w^*(v) \leq 0$; \label{cas1} 

\item If $v\in \Vquatre$ and $v$ has no chelou neighbour, then
    $w^*(v) \leq - \frac{1}{3}$; \label{deg4-nodigon-nochelou}

\item If $v\in \Vquatre$ and $v$ has at least two chelou neighbours, then
    $w^*(v) \leq - \frac{1}{3}$; \label{deg4-nodigon-2chelou}
    
\item If $d(v) = 5$  and $v$ is incident to no digon, then $w^*(v) \leq - \frac{1}{6}$. \label{deg5-nodigon}

\item If $d(v)=5$, $v$ is incident to no digon, has a standard neighbour
    incident to a digon, no chelou neighbour, and is a standard neighbour
    of a vertex incident to a digon, then $w^*(v) \leq 
    -\frac{2}{3}$; \label{deg5-nodigon-particular-case}

\item If $d(v) =6$ and  $v$ is incident to at least two digons, then $w^*(v) < -1$;  \label{deg6-2digons}

\item If $d(v) = 6$ and $v$ is incident to one digon, then $w^*(v) \leq -\frac{1}{2}$; \label{deg6-1digon}

\item If $d(v) = 6$, $v$ is incident to no digon, and $\min\{d^-(v),d^+(v)\}=2$, then $w^*(v) < -1$. \label{deg6-in2} 

\item If $d(v) = 6$, $d^+(v) = d^-(v) = 3$, $v$ is incident to no digon, $v$ is a simple standard neighbour of $\ell_1$ vertices incident to no digon, $v$ is a simple standard neighbour of $\ell_2$ vertices incident to a digon, and $v$ has $\ell_3$ simple standard neighbours incident to a digon, then $w^*(v) \leq -\frac{\ell_1}{3} - \frac{\ell_2}{2} - \frac{\ell_3}{3}$. \label{deg6-nodigon}

\item If $d(v) \geq 7$, then $w^*(v) < -1$; \label{7digon}

\end{enumerate}
\end{claim}


\begin{proofclaim}
We distinguish several cases according to the degree of $v$.

{\medskip \noindent \bf Case 1:} $d(v)=4$.

If $v$ is not in $\Vquatre$, then by Claim~\ref{claim:degree_4_not_3_neighbours}, $v$ is incident to two digons and has no simple neighbour. So, by (OBS1), it does not receive nor send any charge. Moreover, $\epsilon(v) = 1$. Hence  $w^*(v) = w(v) =0$. 

Suppose now that $v$ is in $\Vquatre$. We have $w(v) = 1$. 
We distinguish subcases according to the number of chelou neighbours of  $v$. 
Set $N^+(v) = \{y_1,y_2\}$ and $N^-(v) = \{x_1, x_2\}$. 
Recall that $v$ receives no charge, except if $v$ is bad (see (OBS3)), and, by (OBS4) $v$ sends $\frac{1}{2}$ or $\frac{1}{3}$ to each of its non-chelou neighbours, depending whether it is incident to a digon or not. 
\begin{itemize}
\item Assume $v$ has no chelou neighbour.\\ 
Then $v$ sends at least $\frac{1}{3}$ to each of its neighbours. Moreover, since $v$ is not bad, it receives no charge. So $w^*(v) \leq 1 - 4 \times \frac{1}{3} = -\frac{1}{3}$. This proves \ref{deg4-nodigon-nochelou}.

\item Assume $v$ has a unique chelou neighbour (so  $v$ is bad), and this chelou neighbour is nice.\\
Assume without loss of generality that $x_1$ is the unique chelou neighbour of $v$. 
So $v$ receives $\frac{1}{3}$ from $x_1$ and receives no other charge. Let us now see what it sends. 
Since $x_1$ is assumed to be nice and $d^+(v) = 2$, by Claim~\ref{claim:chelou_arc_one}~\ref{it:nice} there is  an A$^+$-configuration on  $x_1v$.  So both $y_1$ and $y_2$ are incident to a digon. 
Hence, $v$ sends $\frac{1}{2}$ to $y_1$ and to $y_2$ by (R2). Moreover it sends at least $\frac{1}{3}$ to $x_2$. 
Altogether, we get $w^*(v) \leq 1 + \frac{1}{3} -  2\times \frac{1}{2}  - \frac{1}{3}  = 0$. 

\item Assume $v$ has a unique chelou neighbour (so  $v$ is bad), and this chelou neighbour is bad.\\
Assume without loss of generality that $x_1$ is the unique chelou neighbour of $v$. 
Since both $x_1$ and $v$ are assumed to be bad, $v$ receives and sends $\frac{1}{3}$ from/to $x_1$ by (R3), and $v$ does not receive any other charge. Moreover, $v$ sends at least $\frac{1}{3}$ its three other neighbours (because they are non-chelou neighbours). 
Hence $w^*(v) \leq 1 - 3\times \frac{1}{3} = 0$.

\item Assume $v$ has at least two chelou neighbours.\\ 
Then $v$ is nice so it does not receive any charge.
Let $\ell$ be the number of chelou neighbours of $v$.
By Claim~\ref{claim:path_nice}, because $v$ is nice,
all the $\ell$ chelou neighbours of $v$ are bad, 
so $v$ sends $\frac{1}{3}$ to each of them by (R3).
Moreover, $v$ sends at least $\frac{1}{3}$ to each of its $4-\ell$ non-chelou neighbours.
Hence, $w^*(v) \leq 1 - \ell \frac{1}{3} - (4-\ell) \frac{1}{3}
=-\frac{1}{3}$. This proves~\ref{deg4-nodigon-2chelou}.
\end{itemize}

{\medskip \noindent \bf Case 2:} $d(v)=5$ and $v$ is incident to a digon.

By Lemma~\ref{lemma:vertex_no_almost_only_digons}, $v$ is incident to exactly one digon.
So $\epsilon(v) =\frac{1}{2}$, and $w(v) = -1$. Moreover, $v$ has exactly two simple standard neighbours from each of which it receives $\frac{1}{2}$ by (R1).
Thus $w^*(v) \leq -1 + 2\times  \frac{1}{2} = 0$.

{\medskip \noindent \bf Case 3:} $d(v)=5$ or $d(v)=6$, 
$\min\{d^-(v),d^+(v)\}=2$ and $v$ is incident to no digon.
We are going to prove \ref{deg5-nodigon} and \ref{deg6-in2}

If $d(v)=5$, then $w(v) = -\frac{1}{2}$, and if $d(v)=6$, then $w(v) = -2$.

We may assume without loss of generality $d^-(v)=2$ and $d^+(v) \in \{3,4\}$.
Then  $v$ receives $\frac{1}{3}$ from each of its $d^+(v)$ standard neighbours that are incident to no digon by (R2) and does not receive any other charge.

Let $x_1,x_2$ be the in-neighbours of $v$. 
Let $x \in \{x_1,x_2\}$. If $x$ is a non-chelou neighbour of $v$, then $d^+(x) \geq 3$ and thus $v$ is a simple standard  neighbour of $x$, so $v$ sends at least $\frac{1}{3}$ to $x$ by (R1) or (R2). If $x$ is a bad chelou neighbour of $v$, then $v$ sends $\frac{1}{3}$ to $x$ by (R3). If $v$ sends at least $\frac{1}{3}$ to both $x_1$ and $x_2$, then $w^*(v) \leq w(v) + 3 \times \frac{1}{3} - 2 \times \frac{1}{3} = -\frac{1}{6}$
if $d(v)=5$ and
$w^*(v) \leq w(v) + 4 \times \frac{1}{3} - 2 \times \frac{1}{3}  < -1$
if $d(v)=6$. So we may assume that it does not happen.

Hence, we may assume that at least one of $x_1$ or $x_2$, say $x_1$, is neither a non-chelou neighbour nor a bad chelou neighbour of $v$. 
So $x_1$ is a nice chelou neighbour of $v$. 
By Claim~\ref{claim:nice_chelou_neighbour_of_degree_5},
$x_2$ is not a nice chelou neighbour of $v$. So $x_2$ is either a non-chelou neighbour or a bad chelou neighbour of  $v$ and thus $v$ sends $\frac{1}{3}$ to $x_2$ as explained above. 

Assume that one of the out-neighbours of $v$, say $y$, is incident to a digon. 
Then $v$ does not receive any charge from $y$, and thus it receives $\frac{1}{3}$ from at most $d^+(v)-1$ of its standard neighbours by (R2).
Hence $w^*(v) \leq w(v) + 2\times \frac{1}{3} - \frac{1}{3} = -\frac{1}{6}$
if $d(v)=5$ and
$w^*(v) \leq w(v) + 3\times \frac{1}{3} - \frac{1}{3}<-1$. 
So we may assume from now on that all out-neighbours of $v$ are incident to no digon.

Let us apply Claim~\ref{claim:chelou_arc_one}~\ref{it:nice} on the out-chelou arc $x_1 v$. Since the out-neighbours of $v$ are incident to no digon, 
there must be a B$^+$-configuration on $x_1v$ with the three bidirected paths of length $0$.
Let $y_1,y_2,y_3$ be the out-neighbours of $v$ which are in this B$^+$-configuration, ordered such that $(y_1, y_2, y_3, y_1)$ is a directed $3$-cycle.
By Claim~\ref{claim:not_wheel_3}, one of the $y_i$, say $y_1$, has in-degree at least $3$.
So  $v$ sends $\frac{1}{3}$ to $y_1$ by (R2). Hence,  
$w^*(v) \leq w(v) + 3\times \frac{1}{3} - 2\times \frac{1}{3} = -\frac{1}{6}$ if $d(v)=5$ and $w^*(v) \leq w(v) + 4\times \frac{1}{3} - 2\times \frac{1}{3} = -\frac{4}{3}$ if $d(v)=6$.

{\medskip \noindent \bf Case 4:} $d(v) = 5$, $v$ is incident to no digon, $v$ has no chelou neighbour, $v$ has a standard neighbour incident to a digon, and $v$ is a standard neighbour of a vertex incident to a digon. 
We are going to prove~\ref{deg5-nodigon-particular-case}, i.e. we prove that $w^*(v) \leq -\frac{2}{3}$.

By (R2), $v$ receives $\frac{1}{3}$ from each of its simple standard neighbours incident to no digon, and receives no other charge. It has three simple standard neighbours, and at least one of them is incident to a digon, so $v$ receives at most $\frac{2}{3}$. 

Moreover, as $v$ is a standard neighbour of at least two vertices  (indeed, since $v$ has no chelou neighbour, $v$ is a simple standard neighbour of its two binary neighbours),
at least one of which is incident to a digon, $v$ sends $\frac{1}{2}$ 
to the latter one  by (R1), and at least $\frac{1}{3}$ to the other ones
by (R1) or (R2). Thus $w^*(v) \leq w(v) + 2\frac{1}{3} - \frac{1}{2}
-\frac{1}{3} = -\frac{2}{3}$. 

{\medskip \noindent \bf Case 5:} $d(v)=6$.\\
We are going to prove~\ref{deg6-2digons}~\ref{deg6-1digon} and~\ref{deg6-nodigon}. By Case 3, we may assume $d^+(v) = d^-(v) = 3$. 

If $v$ is incident to two digons, then $w(v) = -3$ and  $v$ can only receive charge from its two simple neighbours by (R2). Thus $w^*(v) \leq w(v) + 2\times \frac{1}{2} = -2 <-1$. If $v$ is incident to three digons, then it receives no charge at all.
So $w^*(v) = w(v) = -\frac{5}{2} < -1$.
This proves \ref{deg6-2digons}.

If $v$ is incident to one digon, then $w(v) =-\frac{5}{2}$.
Moreover, $v$ can only receive charge from its at most four simple standard neighbours incident to no digon by (R1). Thus $w^*(v) \leq w(v) + 4\times \frac{1}{2} = -\frac{1}{2}$. This proves \ref{deg6-1digon}.

Assume now that $v$ is as in~\ref{deg6-nodigon}, that is 
$v$ is incident to no digon, $v$ is a simple standard neighbour of $\ell_1$ vertices incident to no digon, $v$ is a simple standard neighbour of $\ell_2$ vertices incident to a digon, and $v$ has $\ell_3$ simple standard neighbours incident to a digon. 
Then $w(v) = -2$.
Now $v$ receives at most $\frac{1}{3}$ from each of its $6-\ell_3$ simple standard neighbours incident to no digon.
Moreover, by (R2), $v$ sends $\frac{1}{3}$  $\ell_1$ of its neighbours and, by (R1,) $v$ sends $\frac{1}{2}$ $\ell_2$ of its neighbours. 
Altogether, we get $w^*(v) \leq -2 + \frac{6-\ell_3}{3} - \frac{\ell_1}{3} - \frac{\ell_2}{2}
=-\frac{\ell_1}{3} -\frac{\ell_2}{2} -\frac{\ell_3}{3}$. 
This proves \ref{deg6-nodigon}.

{\medskip \noindent \bf Case 6:} $d(v)\geq 7$.

Let $b = d(v) - n(v)$ be the number of digons incident to $v$. 

If $b=0$, then  $w(v)  = 7 -\frac{3}{2}d(v)$. Moreover, $v$ only receives by (R2), so it receives at most $\frac{1}{3}d(v)$. So
 $w^*(v) \leq w(v)  +  \frac{1}{3}d(v) = 7 - \frac{7}{6}d(v) < -1$.

If $b=1$, then $w(v) = \frac{13}{2} -\frac{3}{2}d(v)$. 
Moreover, $v$ only receives by (R1) through its at most $d(v) - 2$ simple standard neighbours, so it receives at most  $\frac{1}{2}(d(v)-2)$.
 So $w^*(v) \leq w(v)  +  \frac{1}{2} (d(v) - 2) = \frac{11}{2} - d(v) < -1$.

If $b \geq 2$, then $w(v) = 5 - \frac{3}{2}d(v) + \frac{b}{2}$. 
Moreover, $v$ only receives by (R1) through its at most $d(v) - 2b$ simple standard neighbours, so it receives at most  $\frac{1}{2}(d(v)-2b)$.
So $w^*(v) \leq w(v) \frac{1}{2}(d(v)-2b) = 5 - d(v) - \frac{b}{2} < -1$ 
This proves \ref{7digon}.
\medskip

Altogether, the above cases imply \ref{cas1}.
\end{proofclaim}


\subsection{Some more tools before colouring}\label{subsec:moretools}

In this subsection, we use Claims~\ref{claim:epsilon_correct} and~\ref{claim:result_discharging} to prove some more structural properties on $D$. More, precisely we prove: 
\begin{itemize}
    \item Some restrictions on the possible degrees of the vertices. Claim~\ref{claim:degree_prop_after_discharging}. 
    \item $D$ has no A- nor B-configurations. Claims~\ref{claim:no_A_conf} and~\ref{claim:no_B_conf}.
    \item The connected components  of $B(D)$ are paths of length $2$ and some strong structural properties are forced around them. Claims~\ref{claim:component_of_B(D)}, \ref{claim:2thread}, \ref{claim:apex} and \ref{claim:no6digon}. 
    \item Vertices of degree $6$ are incident to no digon. Claim~\ref{claim:no6digon}
\end{itemize}

A digon (resp. an arc) is {\bf isolated}  if its end-vertices are incident to no other digon (resp. no other arc).

\begin{claim} \label{claim:degree_prop_after_discharging}
For every vertex $v$ of $D$, we have:
\begin{enumerate}[label=(\roman*)]
\item $d(v) \in \{4, 5, 6\}$; \label{claim:deg_in_456}
\item if $d(v) = 6$, then $d^+(v)=d^-(v)=3$; \label{claim:deg6=3-3}
\item if $v$ is incident to more than one digon, then  $d(v) = 4$. \label{claim:digons_deg4}
\end{enumerate}
\end{claim}

\begin{proofclaim}
By Claim~\ref{claim:result_discharging}~\ref{cas1}, $\Sigma(D) = \sum_{v \in V(D)} w^*(v) \leq w^*(v)$ for every vertex $v$ in $D$.
In particular, since  $\Sigma(D) \geq \rho(D) \geq -1$, we have $w^*(v) \geq -1$ for every vertex $v$ in $D$.
Hence, the three outputs are implied by respectively Claim~\ref{claim:result_discharging}~\ref{7digon},  Claim~\ref{claim:result_discharging}~\ref{deg6-in2}, and Claim~\ref{claim:result_discharging}~\ref{deg6-2digons} together with Lemma~\ref{lemma:vertex_no_almost_only_digons},
\end{proofclaim}

We shall now deduce even more structure on $D$.

\begin{claim}\label{claim:component_of_B(D)} 
Let $C$ be a connected component of $B(D)$. Then $C$ is a path of length at most 2. Moreover:
\begin{itemize}
    \item if $C$ has length $1$, then its two vertices have degree $5$ in $D$.
    \item if $C$ is a path of length $2$, then its internal vertex has degree $4$ in $D$. 
\end{itemize}
\end{claim}

\begin{proofclaim}
Let $C$ be a connected component of $B(D)$. 
By Claim~\ref{claim:degree_prop_after_discharging}~\ref{claim:digons_deg4},  every vertex of $B(D)$ has degree at most $2$ in $B(D)$, and if it has degree $2$ in $B(D)$, then it has degree $4$ in $D$. 
Let $b, c$ be two vertices of degree $2$ in $B(D)$, and assume $[b,c] \subseteq A(D)$. 
Let $a$ be the other neighbour of $b$ in $B(D)$ and $d$ the other neighbour of $c$. 
If $a = d$, then $\{a,b,c\}$ induces a bidirected $3$-cycle in $D$, a contradiction. 
So we may assume that $a \neq d$ and thus $[a,b,c,d]$ is a $3$-thread, a contradiction to  Lemma~\ref{lemma:subdiv_digon}. 
Hence, the set of vertices of degree $2$ in $B(D)$ is an independent set.  
So $C$ is a path of length at most $2$, and if it is a path of length $2$ then its internal vertex has degree $4$. 

Assume that $C$ is a path of length one, say $C=[u,v]$. 
By Claim~\ref{claim:degree_4_not_3_neighbours}, both $u$ and $v$ have degree at least $5$.
 If one of $u$, $v$, say $u$, has degree at least $6$, then by~\ref{claim:result_discharging}~\ref{deg6-1digon}, $w^*(u) \leq -\frac{1}{2}$, and then, by~Claims~\ref{claim:epsilon_correct} 
 $\rho(D) \leq \Sigma(D) - 1 \leq w^*(u) -1 \leq -\frac{1}{2} - 1 < -1$,
a contradiction. 
\end{proofclaim}

\begin{claim}\label{claim:no_A_conf}
$D$ contains no A$^+$- nor A$^-$-configuration on a chelou arc.
As a consequence, any connected component of $D[\Vquatre]$
is either an isolated vertex incident to no chelou arc, or an isolated arc,  end-vertices of which are incident to no other chelou arc 
\end{claim}

\begin{proofclaim}
Suppose for contradiction that $D$ has an A$^+$-configuration $H$
on an out-chelou arc $xy$ (The proof is similar for an  A$^-$-configuration).
By Claim~\ref{claim:component_of_B(D)}, the path of digons
in $H$ is equal to $[y_1,y_2]$, where $y_1$, $y_2$ are two out-neighbours of $y$.
Hence by Claims~\ref{claim:epsilon_correct} and~\ref{claim:result_discharging},
$0 \leq \rho(D) + 1 \leq \Sigma(D)  \leq w^*(y)$.
If $y$ has degree at least $5$, as $xy$ is out-chelou, $y$ is incident to
no digon, so by Claim~\ref{claim:result_discharging}~\ref{deg5-nodigon}
we have $w^*(y) \leq - \frac{1}{6}$, a contradiction.
Hence $y$ has degree $4$.

By Claim~\ref{claim:digon_degre_5_plus}, $d^+(y_1) = d^+(y_2) = 3$. By Claim~\ref{deg6-1digon}, if $y_1$ or $y_2$ has degree $6$, then it has final charge at most $-\frac{1}{2}$ and thus, as above, $\rho(D) < -1$, a contradiction. 
So $d^-(y_1)=d^-(y_2) = 2$ and thus both $yy_1$ and $yy_2$ are chelou arc, a contradiction to Claim~\ref{claim:path_nice}. 
\end{proofclaim}

\begin{claim}\label{claim:2thread}
Let $[y_1,x,y_2]$ be a $2$-thread.
There exists a vertex $z$ such that $y_1z, zy_2 \in A(D)$ or
$y_2z,zy_1 \in A(D)$.
\end{claim}

\begin{proofclaim}
First note that by Claims~\ref{claim:cycle_minus_one_arc} $y_1$ and $y_2$ are not adjacent, and by Claim~\ref{claim:component_of_B(D)} $x$ has degree $4$. 
Let $D'$ be the digraph obtained from $D$ by removing $x$ and merging
$y_1,y_2$ into a single vertex $y$.
It is enough to prove that $y$ is in a digon in $D'$ 

If $D'$ has a  $2$-dicolouring $\phi$, then 
setting $\phi(x) \neq \phi(y_1) = \phi(y_2) = \phi(y)$ yields a 
$2$-dicolouring of $D$, a contradiction.
Hence $D'$ contains a $3$-dicritical digraph $\Tilde{D}$, which necessarily contains $y$ because $D$ is $3$-dicritical.
Let $U = V(\Tilde{D}) \setminus\{y\} \cup\{y_1,y_2,x\}$.

Suppose for a contradiction that $y$ is incident to no digon in $\tilde{D}$.
Then $\tilde{D}$ is neither a bidirected odd cycle nor an odd 3-wheel.
Hence, by minimality of $D$, $\rho(\Tilde{D}) \leq -2$.
Moreover, if $M$ is a matching of
digons in $\Tilde{D}$,  then $M \cup [y_1,x]$ is a matching of digons  in $D$. So $\pi(D) \geq \pi(\Tilde{D}) + 1$.
Thus $\rho_D(U) \leq \rho(\Tilde{D}) + 2 \times 7 - 3\times 4  -2=  \rho(\Tilde{D})$.
Now  $U=V(D)$, for otherwise, by Claim~\ref{claim:high_potential_number_one},   $4\leq \rho_D(U) \leq  \rho(\Tilde{D}) \leq -2$, a contradiction.
But then $-1 \leq \rho(D) \leq \rho_D(U) \leq  \rho(\Tilde{D}) \leq -2$, a contradiction.
\end{proofclaim}

In view of Claim~\ref{claim:2thread}, we call {\bf apex} of a $2$-thread $[y_1,x,y_2]$ a vertex $z$ such that $y_1z, zy_2 \in A(D)$ or
$y_2z,zy_1 \in A(D)$.

\begin{claim}\label{claim:apex}
Let $z$ be the apex of a $2$-thread. Then
\begin{enumerate}[label=(\roman*)]
\item $z$ has degree at least $5$, \label{no-apex-deg4}
\item  $z$ is the apex of a unique $2$-thread, \label{apex-of-1}
\item $z$ is incident to no digon , and\label{apex-no-digon}
\item $w^*(z) \leq -\frac{2}{3}$.\label{apex-2/3} 
\end{enumerate}
As a consequence $D$ contains at most one $2$-thread.
\end{claim}

\begin{proofclaim}
Let $[y^-_1,x_1,y^+_1]$ be a $2$-thread such that $y^-_1z, zy^+_1 \in A(D)$. 
\medskip

\ref{no-apex-deg4} Suppose for a contradiction that $z$ has degree $4$. Then $z \in \Vquatre$ for otherwise, by Claim~\ref{claim:degree_4_not_3_neighbours}, both $[z,y_1^-]$ and  $[z,y_1^+]$ must be digons of $D$, and thus $B(D)$ contains a cycle, a contradiction to Lemma~\ref{lem:forestB}. 
Let us denote by $y^-_2$ (resp. $y^+_2$) the in-neighbour (resp. out-neighbour) of $z$ distinct from $y^-_1$ (resp. $y^+_1$).

Consider the digraph $D' = (D - z) \cup y^-_1 y^+_1$ which has a 
$2$-dicolouring $\phi$ by Claim~\ref{claim:2colouring_minus_v_plus_e}.
Without loss of generality, we may assume that $\phi(x_1)=1$.
Thus $\phi(y^-_1)=\phi(y^+_1)=2$.
If $\phi(y^-_2)=2$ or $\phi(y^+_2)=2$, then setting $\phi(z) =1$ yields $2$-dicolouring of $D$, a contradiction.
Thus  $\phi(y^-_2)=\phi(y^+_2)=1$. Set $\phi(z) =2$.
Then there is no monochromatic dicycle $C$ in $D$, for otherwise it must contain $(y^-_1,z,y^+_1)$ and thus the dicycle $C'$ obtained from $C$ by replacing this subpath by $(y^-_1,y^+_1)$ is a monochromatic dicycle in $D'$.
Hence $\phi$ is  a $2$-dicolouring of $D$, a contradiction. This proves~\ref{no-apex-deg4}. 
\medskip

So we may assume from now on that $z$ has degree at least $5$, and by Claim~\ref{claim:degree_prop_after_discharging}~\ref{claim:deg_in_456}, $z$ has degree $5$ or $6$. Without loss of generality, we may assume that $d^+(z)=3$ and $d^-(z) \in \{2,3\}$.
We denote by $y^+_1, y^+_2,y^+_3$ the out-neighbours of $z$ and if $d^-(z) =2$ we denote by $y^-_1$, $y^-_2$ the in-neighbours of $z$.
\medskip

\ref{apex-of-1} Suppose for a contradiction that $z$ is the apex of a $2$-thread distinct from $[y^-_1,x_1,y^+_1]$. By Claim~\ref{claim:component_of_B(D)}, this $2$-thread is disjoint from $[y^-_1,x_1,y^+_1]$. So we may assume without loss of generality that it is $[y^-_2,x_2,y^+_2]$. 
\medskip

If $z$ is incident to a digon, then  $d(z) = 6$ and each of its simple neighbours are incident to a digon, so it receives no charge.
Thus $w^*(z) \leq w(z) = -\frac{5}{2} < -1$, a contradiction. 
Henceforth $z$ is incident to no digon.
\medskip 

If  $z$ has degree $6$,  then, by (R2),  its receives  $\frac{1}{3}$ from its neighbours incident to no digon, which are only two here.
Thus $w^*(z) \leq w(z) + 2\times \frac{1}{3} = -\frac{4}{3}< -1$, a contradiction.
Henceforth $z$ has degree $5$, and in particular $d^-(z) = 2$.
\medskip

If both $y^-_1$ and $y^-_2$ have out-degree at least $3$, then $z$ is a simple standard neighbour of both $y^-_1$ and $y^+_1$ and since they are both incident to a digon, $z$ sends $\frac{1}{2}$ to each of them by (R1).
Moreover, $z$ has at most one standard neighbour incident to no digon, namely $y^+_3$, and so it receives at most $\frac{1}{3}$ by (R2).
Hence $w^*(z) \leq w(z) -2\times\frac{1}{2} + \frac{1}{3} = -\frac{7}{6} < -1$,
a contradiction.

So one of  $y^-_1$ or $y^-_2$, say $y^-_1$, has out-degree $2$. In particular, since $z$ is incident to no digon and $d^-(z) = 2$, $y^-_1z$ is an out-chelou arc.  
Moreover, since $y_1^-$ is incident to a digon, $y_1^-$ is nice. 
So, by Claim~\ref{claim:chelou_arc_one}~\ref{it:nice}, we have either
an A$^+$-configuration or a B$^+$-configuration on $y^-_1z$.
By Claim~\ref{claim:no_A_conf}, it must be a B$^+$-configuration.
But then $y^+_2y^-_2 \in A(D)$, and so $D[\{y^-_2,x_2,y^+_2\}]$ contains a bidirected $3$-cycle minus one arc, a contradiction to Claim~\ref{claim:cycle_minus_one_arc}.
This proves~\ref{apex-of-1}.
\medskip 

We now prove the following assertion:
\medskip

\noindent($\star$) If $z$ is incident to no digon, then $w^*(z) \leq -\frac{2}{3}$. 
\medskip

\noindent{\it Proof of ($\star$)}: Assume that $z$ is incident to no digon. 

If $z$ has degree $6$, then, by (R2),  $z$ receives nothing from $y_1^-$ and $y_1^+$ since these two vertices are incident to a digon, and thus receives $\frac{1}{3}$  from at most four neighbours.  So $w^*(z)\leq w(z) + 4\times\frac{1}{3} = -\frac{2}{3}$. 
So we may assume that $z$ has degree $5$. 
\medskip

If $y^-_2$ is a nice chelou neighbour of $z$, then by Claims~\ref{claim:chelou_arc_one}~\ref{it:nice}
and~\ref{claim:no_A_conf}, there is a B$^+$-configuration on $y^-_2z$. In particular $y^+_1y^-_1 \in A(D)$, so
$D[\{y^-_1,x_1,y^+_1\}]$ contains a bidirected $3$-cycle
minus one arc, a contradiction to Claim~\ref{claim:cycle_minus_one_arc}.
Henceforth $y^-_2$ is not a nice chelou neighbour of $z$.
If $y^-_2$ is  a bad chelou neighbour of $z$, then $z$ sends $\frac{1}{3}$
to $y^-_2$ by (R3). Otherwise, $z$ is a standard neighbour of $y^-_2$ and
is incident to no digon, so $z$ sends at least $\frac{1}{3}$ to $y^-_2$ by 
(R1) or (R2). In both cases, $z$ sends at least $\frac{1}{3}$ to $y^-_2$.
\medskip

If $d^+(y_1^-) \geq 3$, then $z$ has a standard neighbour incident to a digon (namely $y^+_1$), and is a standard neighbour of a vertex incident to a digon (namely $y^-_1$). 
So Claim~\ref{claim:result_discharging}~\ref{deg5-nodigon-particular-case} applied with $\ell_2, \ell_3 \geq 1$ and we have $w^*(z) \leq -\frac{1}{2} - \frac{1}{3} \leq -\frac{2}{3}$.
Henceforth, we may assume that $d^+(y^-_1) = 2$.
Moreover, since $z$ is incident to no digon,  $y_1^-z$ is an out-chelou arc, and since $d(y_1^-) \geq 5$, we get that $y_1^-$ is a nice chelou neighbour of $z$.  
\medskip

So, by Claims~\ref{claim:chelou_arc_one}~\ref{it:nice} and~\ref{claim:no_A_conf},
there is a B$^+$-configuration on $y^-_1z$. 
In particular, $d^-(y^+_1) \geq 3$, so $z$ is a simple standard 
neighbour of $y^+_1$ and thus $z$ sends $\frac{1}{2}$ to $y_1^+$ by (R1). 


Finally, by (R2) $z$ receives $\frac{1}{3}$ from each of its standard neighbour incident to no  digon which are at most $2$ (since $y_1^+$ is incident to a digon). 
So $w^*(v) \leq w(v) -\frac{1}{2} - \frac{1}{3} + \frac{2}{3} \leq  - \frac{2}{3}$. 
This proves ($\star$). 

\medskip

\ref{apex-no-digon} Suppose for a contradiction that $z$ is incident to a digon. 

If $z$ has degree $6$, then $z$ has at most two simple standard neighbours incident to no digon, so the total charge it receives is at most $2\times \frac{1}{2}$ by (R1) and thus $w^*(z) \leq w(z) + 1 = -\frac{3}{2}<-1$, a contradiction. 
Henceforth $z$ has degree $5$, and is incident to a unique digon by Lemma~\ref{lemma:vertex_no_almost_only_digons}.

Now $z$ has at most one simple standard neighbour incident 
to no digon and so it receives at most once $\frac{1}{2}$ by (R1).
As a consequence $w^*(z) \leq w(z) + \frac{1}{2} = -\frac{1}{2}$.

Now observe that $D$ has no isolated digon as otherwise Claim~\ref{claim:epsilon_correct} gives
$-1 \leq \rho(D) \leq \Sigma(D) -1 \leq - \frac{1}{2} -1 $,
a contradiction. Hence $z$ is in a $2$ thread $[z,x',z']$,  which by Claim~\ref{claim:2thread} has an apex, say $v$.

By Claim~\ref{claim:component_of_B(D)}, $\{y_1^-,x_1,y_1^+\} \cap \{z,x',z'\} = \emptyset$. 

If $v$ is in no digon, then by ($\star$),  $w^*(v) \leq -\frac{2}{3}$ and so
$\Sigma(D) \leq w^*(z) + w^*(v) \leq - \frac{7}{6} < -1$, a contradiction.
Hence $v$ is incident to a digon and thus it is in a $2$-thread because there is no isolated digon.
If this $2$-thread is not  $[y^-_1,x_1,y^+_1]$ (i.e. $v \notin \{y_1^-, y_2^-\}$), let $u$ be its apex. By~\ref{apex-of-1}, $u$, $v$ and $z$ are pairwise distinct.
Moreover, similarly to $z$, $w^*(v)\leq -\frac{1}{2}$, and similarly to $v$, $u$ is incident to a digon and $w^*(u)\leq -\frac{1}{2}$. Thus
$ \Sigma(D) \leq w^*(z) + w^*(v) +w^*(u) \leq - \frac{3}{2} < -1$, a contradiction.
 Henceforth $v \in \{y_1^-, y_2^-\}$ and we may assume without loss of generality that $v=y^-_1$ and $z'y^-_1 \in A(D)$.

By Claim~\ref{claim:2colouring_minus_v_plus_e}, $(D-z)\cup y^-_1z'$
has a  $2$-dicolouring $\phi$. 
Without loss of generality $\phi(z')=\phi(x_1)=1$
and $\phi(y^-_1)=\phi(y^+_1)=\phi(x')=2$.
We set $\phi(z)=1$.
There is no directed cycle coloured $1$ containing $z$, because three of the four neighbours
of $z$, namely $x'$, $y^-_1$ and $y^+_1$, are coloured $2$. Thus $\phi$ is a  $2$-dicolouring of $D$, a contradiction.
This proves~\ref{apex-no-digon}. 
\medskip

\ref{apex-2/3} Follows directly from~\ref{apex-no-digon} and ($\star$). 

\medskip

To see that $D$ has at most one $2$-thread, suppose for contradiction that it has at least two $2$-threads with apices $z$ and $z'$, which are distinct by~\ref{apex-of-1}.
Then by~\ref{apex-2/3}, $w^*(z),w^*(z') \leq -\frac{2}{3}$ and therefore $\rho(D) \leq \Sigma(D) \leq w^*(z) + w^*(z') \leq -\frac{4}{3}<-1$, a contradiction.
\end{proofclaim}

\begin{claim}\label{claim:no6digon}
A vertex of degree $6$ is incident to no digon. 
\end{claim}

\begin{proofclaim}
Suppose for a contradiction that $v$ is a vertex of degree $6$ incident to a digon. 
By Claim~\ref{claim:degree_prop_after_discharging}~\ref{claim:digons_deg4}, $v$ is incident to exactly one digon. 
By Claim~\ref{claim:result_discharging}~\ref{deg6-1digon}, we have
$w^*(v) \leq -\frac{1}{2}$.

If $v$ is in an isolated digon, then by Claim~\ref{claim:epsilon_correct}, $-1 \leq \rho(D) \leq \Sigma(D) - 1 \leq w^*(v) -1 \leq  -\frac{3}{2} <-1$,
a contradiction.
Henceforth $v$ is in a $2$-thread. By Claims~\ref{claim:2thread}, this $2$-thread has an apex $z$ (which is distinct from $v$) and $w^*(z) \leq -\frac{2}{3}$ by Claim~\ref{claim:apex}.
Thus $-1 \leq \Sigma(D) \leq w^*(z) + w^*(v) \leq -\frac{1}{2}-\frac{2}{3}
=-\frac{7}{6} < -1$, a contradiction.
\end{proofclaim}

\begin{claim}\label{claim:no_B_conf}
$D$ contains no B$^+$-configuration on an out-chelou arc $xy$ such that $x$ is nice.
\end{claim}

\begin{proofclaim} 
Suppose for a contradiction that there is a B$^+$-configuration $H$ on a
out-chelou arc $xy$ such that $x$ is nice. Let $N^+(y) = \{y_1, y_2, y_3\}$ and let $z$ be the in-neighbour of $y$ distinct from $z$.
In particular, $y_1z,y_2z,y_3z \in A(D)$ and so $d^-(z) \geq 3$.
See Figure~\ref{subfig:3_wheel}.

Since $xy$ is chelou, $d^-(y) = 2$ and thus, by Claim~\ref{claim:degree_prop_after_discharging}~\ref{claim:deg6=3-3}, $d(y) = 5$. 
So $w^*(y) \leq - \frac{1}{6}$ by Claim~\ref{claim:result_discharging}~\ref{deg5-nodigon}.

If $zy$ is chelou, then  $z$ is a nice chelou neighbour of $y$ because 
$d(z) \geq 5$, a contradiction to
Claim~\ref{claim:nice_chelou_neighbour_of_degree_5}.  
So $zy$ is not chelou, and thus $d^+(z) = 3$ and  $d(z) = 6$.

Observe that if $H$ spans $D$, then $A(D) \setminus A(H)$ contains
at least $4$ arcs because $d^+_D(z) = 3$ and $d^-_D(x) \geq 2$.
Hence $\rho(D) \leq \rho(H) - 4 \times 3 \leq -2$, a contradiction.
Therefore $V(H) \neq V(D)$.

Assume that some $y_i$, say $y_1$, is in a $2$-thread $[y_1,x,y'_1]$ in the B$^+$-configuration $H$.
We have $d(z) = 6$, and $z$ has $y_1$ as a simple standard neighbour  which is incident to a digon. 
So we may apply Claim~\ref{claim:result_discharging}~\ref{deg6-nodigon} on $z$ with $\ell_1 \geq 1$, so  $w^*(z) \leq - \frac{1}{3}$. 
By Claim~\ref{claim:2thread}, $[y_1,x,y'_1]$ has an apex $v$, and $w^*(v) \leq - \frac{2}{3}$ by
Claim~\ref{claim:apex}. 
If $v\notin \{y,z\}$, then $-1 \leq \Sigma(D) \leq
w^*(v) + w^*(z) + w^*(y) \leq -\frac{2}{3} -\frac{1}{3} - \frac{1}{6} =-\frac{7}{6} < -1$, a contradiction. Henceforth, $v\in \{y,z\}$. 
In particular, there is an arc between $\{y,z\}$ and $y'_1$. Set $R=V(H)\setminus \{x\}$. Then $D[R]$ has at least as many arcs as $H$ and one vertex less. Thus $\rho(R) \leq \rho(H) -7 =2$ by Claim~\ref{claim:B+pot}, a contradiction to Claim~\ref{claim:high_potential_number_one}. Henceforth, none of the $y_i$ belongs to a $2$-thread in $H$. 

\medskip

Hence the $y_i$ are incident to no digon in $H$, so $(y_1,y_2, y_3, y_1)$ is a directed $3$-cycle.
Moreover, every $y_i$ has degree at least $4$ in $H$, and because vertices
of degree $6$ in $D$ are incident to no digon by
Claim~\ref{claim:no6digon},
we deduce that $y_i$ is in no digon of the form $[y_i,u]$ in $D$ with $u \not\in V(H)$.

Furthermore, if $y_i$ is in a digon $[y_i,u]$ with $u \in V(H-x)$, then
$D[V(H-x)]$ has an arc more than $H-x$, and so by Lemma~\ref{lemma:potential_subgraph_not_induced} $\rho_D(V(H-x)) \leq \rho(H-x)-3 \leq 5 - 3 < 4$, contradicting Claim~\ref{claim:high_potential_number_one}.

Finally, if $[y_i,x]$ is a digon in $D$ for some $i \in [3]$, 
then $D[V(H)]$ has two more arcs than $H$, and so $\rho_D(V(H)) \leq 
\rho(H) - 2\times 3 \leq 9 - 6 < 4$, 
contradicting Claim~\ref{claim:high_potential_number_one} (because $H$ does not spans $D$).

Altogether, this proves that for every $i \in [3]$, $y_i$ is in no digon in $D$.

\medskip 

First suppose that some $y_i$, say $y_1$, has degree $6$. 
Then $y$ is a simple standard neighbour of $y_1$. 
Then by Claim~\ref{claim:result_discharging}~\ref{deg6-nodigon},
we get $w^*(z) \leq - \frac{1}{3}$ when applied on $z$ with $\ell_1+\ell_2 \geq 1$,
and also $w^*(y_1) \leq -\frac{2}{3}$ when applied on $y_1$ with $\ell_1 + \ell_2 \geq 2$ (because $y_1$ is a simple standard neighbour of both $y$ and $z$). 
Recall that $w^*(y) \leq -\frac{1}{6}$. Thus
$-1 \leq \Sigma(D) \leq -\frac{7}{6} < -1$, a contradiction.
Henceforth each $y_i$ has degree $4$ or $5$. 
\medskip

Assume now that some $y_i$, say $y_1$, has degree $4$.
Then by Claim~\ref{claim:ind-cycle} and its directional dual, one of $y_2,\ y_3$ has out-degree $3$ and the other has in-degree $3$. 
Hence both $y_2$ and $y_3$ have degree $5$. 
If $y_1$ has a chelou neighbour, then it must be a nice chelou neighbour (because all neighbours of $y_1$ has degree at least $5$) and thus, by Claim~\ref{claim:chelou_arc_one}~\ref{it:nice}, there is an $A$-configuration on $y$ (it cannot be a $B$-configuration since we assumed $d(y_1) = 4$). A contradiction to  Claim~\ref{claim:no_A_conf}. 
So $y_1$ has no chelou neighbour. 

Thus $d^-(y_2)=d^+(y_3)=3$ and by Claim~\ref{claim:result_discharging}%
~\ref{deg4-nodigon-nochelou} $w^*(y_1) \leq - \frac{1}{3}$.
Moreover, as $d^+(y_3)=3$, by Claim~\ref{claim:result_discharging}%
~\ref{deg6-nodigon} with $\ell_1 \geq 1$, $w^*(z) \leq -\frac{1}{3}$.
Finally, by Claim~\ref{claim:result_discharging}~\ref{deg5-nodigon},
$w^*(y),w^*(y_2),w^*(y_3) \leq -\frac{1}{6}$. 
Thus
$\rho(D) \leq \Sigma(D) \leq -\frac{1}{3}-\frac{1}{3} -
3\times\frac{1}{6} = -\frac{7}{6} < -1$, a contradiction.
Henceforth,  for $i\in [3]$, $d(y_i)=5$ and so $w^*(y_i) \leq -\frac{1}{6}$ by Claim~\ref{claim:result_discharging}~\ref{deg5-nodigon}.
\medskip  

If two of the $y_i$ have out-degree $3$, then they are simple standard neighbours of $z$, so  Claim~\ref{claim:result_discharging}~\ref{deg6-nodigon} applied on $z$ with $\ell_1 \geq 2$, so $w^*(z) \leq -\frac{2}{3}$. 
As a consequence $\Sigma(D)  \leq w^*(z)+
w^*(y_1) + w^*(y_2) + w^*(y_3) \leq -\frac{2}{3}- 3\times \frac{1}{6} = -\frac{7}{6} < -1$, a contradiction.
Hence at least two of the $y_i$, say $y_1,y_2$ have in-degree $3$.
If $y_3$ has out-degree $3$, then by Claim~\ref{claim:result_discharging}%
~\ref{deg6-nodigon} with $\ell_1 \geq 1$, $w^*(z) \leq -\frac{1}{3}$.
Moreover, $y$ receives $\frac{1}{3}$ from each of its three out-neighbours
by (R2), and sends $\frac{1}{3}$ to each of $z,y_1,y_2$ by (R2).
Hence $w^*(y) \leq w(y) +3\times\frac{1}{3} - 3\times\frac{1}{3} = -\frac{1}{2}$,
and  $\Sigma(D)  \leq w^*(y) + w^*(y_1) + w^*(y_2) + w^*(y_3) + w^*(z)
= -\frac{1}{2} - 3\times \frac{1}{6} - \frac{1}{3}
= -\frac{4}{3} < -1$, a contradiction.
Henceforth, for $i\in [3]$, $d^+(y_i)=2$. 

Now, the directed $3$-cycle $(y_1,y_2,y_3,y_1)$ contradicts the directional dual of Claim~\ref{claim:not_wheel_3}.
\end{proofclaim}

\begin{claim}\label{claim:chelou_only_in_V4}
All chelou arcs have both end-vertices in $\Vquatre$.
\end{claim}

\begin{proofclaim}
Let $xy$ be a chelou arc, and assume without loss of generality that it is out-chelou. So  $d^+(x) = d^-(y) = 2$ and $y$ is incident to no digon. 
By Claims~\ref{claim:chelou_arc_one}, \ref{claim:no_A_conf} and~\ref{claim:no_B_conf}, $x$ must be a bad vertex. 
In particular $d(x) = 4$ and since $x$ is incident to at least one simple arc (namely $xy$), $x$ cannot be incident to a digon  by Claim~\ref{claim:degree_4_not_3_neighbours} and thus $x \in \Vquatre$. 
Similarly, if $d^+(y) = 3$, then $y$ is nice and since $x$ is incident to no digon, $xy$ is an in-chelou arc and we reach a contradiction by Claims~\ref{claim:chelou_arc_one}, \ref{claim:no_A_conf} and~\ref{claim:no_B_conf}. Hence both $x$ and $y$ are in $\Vquatre$.
\end{proofclaim}

\begin{claim}\label{claim:no_isolated_digon}
There is no isolated digon in $D$.
\end{claim}

\begin{proofclaim} 
Suppose for contradiction that $D$ has an isolated digon $[y_1,y_2]$. By Claims~\ref{claim:degree_4_not_3_neighbours} and~\ref{claim:no6digon}, we must have $d(y_1) = d(y_2) = 5$. 
By Claims~\ref{claim:epsilon_correct} and~\ref{claim:result_discharging}~\ref{cas1}, we have $ \rho(D) \leq \Sigma(D) -1 \le -1$ and thus $\rho(D) = -1$. 
Hence $[y_1,y_2]$ is the unique isolated digon and for every vertex $v$, $w^*(v) = 0$. 
In particular $D$ has no $2$-threads by Claim~\ref{claim:apex}, and thus $[y_1, y_2]$ is the only digon of $D$. 
By Claim~\ref{claim:result_discharging}~\ref{deg5-nodigon} there is 
no vertex of degree $5$ other than $y_1$ and $y_2$.

If $D$ has no vertex of degree $4$, then $w(x) \leq -\frac{1}{2}$ for every vertex $x$, and thus $\rho(D) = -1 \leq -\frac{1}{2}n(D)$, which implies $n(D) \leq 2$, a contradiction.
Therefore, there is vertex $v$ of degree $4$, which must be in $\Vquatre$ as  $[y_1,y_2]$ is the unique digon of $D$. Since we must have $w^*(v) = 0$, $v$ must be incident to a unique chelou neighbour by~\ref{claim:result_discharging}~\ref{deg4-nodigon-nochelou} and~\ref{deg4-nodigon-2chelou}. 
By Claim~\ref{claim:chelou_arc_one}~\ref{it:triangle}, $v$ is in a directed $3$-cycle $(u,v,w,u)$ where $u$ and $w$ are not the chelou neighbour of $v$ and thus must have degree at least $5$.
If both $u$ and $w$ have degree $6$, then $w^*(u) \leq -\frac{1}{3}$
by Claim~\ref{claim:result_discharging}~\ref{deg6-nodigon} with $\ell_1 \geq 1$, a contradiction.
Hence, one of  $u$, $w$, say $u$, has degree $5$, and so is $y_1$ or $y_2$,
say $y_1$.
We deduce that every vertex of degree $4$ is a simple standard  neighbour of
$y_1$ or $y_2$, and therefore $|\Vquatre| \leq 4$.
But using the initial charges, observing that $w(y_1) = w(y_2) = -1$ and the initial charge of a vertex of degree $6$ is $7 - 3 \times \frac{6}{2} = -2$, we have $0 \leq \Sigma(D) =  |\Vquatre| -  2 - 2|V_6| \leq 2- 2  |V_6|$, where $V_6$ is the set of vertices of degree $6$. 
As a consequence $|V_6| \leq 1$. 
Let $z$ this possible vertex of degree $6$. 
Altogether, we get that $D$ has two vertices of degree $5$, namely $y_1$ and $y_2$, at most one vertex of degree $6$, $z$ if it exists, and all the other vertices have degree $4$. 

Consider a vertex  in $\Vquatre$. 
As said previously, it is incident with exactly one chelou arc, so it is the standard neighbour of three vertices. Each of these vertices must have degree at least $5$, so  they are $y_1,\ y_2$ and $z$. 
Hence,  $z$ exists and so $|V_6| =1$, and all vertices in $\Vquatre$ are adjacent to  $y_1,\ y_2$ and $z$. 
But $y_1$ has only two simple standard  neighbours, so $|\Vquatre| \leq 2$.
Using again the initial charge, we get $0 \leq \Sigma(D) = 1\times 2 - 1 \times 2 - 2|V_6| =-2$, a contradiction.
%
\end{proofclaim}

\subsection{Dicolouring} \label{subsec:colouring}

By Claims~\ref{claim:degree_prop_after_discharging}, \ref{claim:component_of_B(D)} and \ref{claim:no6digon}, 
$V(D)$ is partitioned into the following sets: 
\begin{itemize}
\item $\Vquatre$, the set of vertices of degree $4$ incident to no digon,
\item $V^d_4$, the set of vertices of degree $4$ which are the middle vertex of a 
    $2$-thread,
\item $V_5^+$ (resp. $V_5^-$), 
    the set of vertices of degree $5$ that are incident to 
    no digon and have out-degree $3$ (resp. in-degree $3$), 
    
\item $V^{d+}_5$ (resp. $V^{d-}_5$), the set of
    vertices of degree $5$ that are the end-vertex of a $2$-thread and have out-degree $3$
    (resp. in-degree $3$), 
\item $V_6$, the set of vertices of degree $6$ incident to no digon and with in- and out-degree $3$. 
\end{itemize}

As the sum of the in-degrees equals the sum of the out-degrees, we have
$|V^-_5| + |V^{d-}_5| = |V^+_5| + |V^{d+}_5|$. 

\begin{claim}\label{claim:cycles_in_V5_V4}
Let $C$ be a directed cycle with vertices in $\Vquatre \cup V_5^+ \cup V^{d+}_5 \cup V_5^-
\cup V^{d-}_5$, then either
\begin{itemize}
\item $V(C)$ is included in $V^-_5 \cup V^{d-}_5$, or
\item $V(C)$ is included in $V^+_5 \cup V^{d+}_5$.
\end{itemize}
In particular, $V(C)\cap \Vquatre = \emptyset$. 
\end{claim}

\begin{proofclaim}
Let $C$ be a directed cycle with vertices in $\Vquatre \cup V_5^+ \cup V^{d+}_5 \cup V_5^- \cup V^{d-}_5$. 
Assume for a contradiction that $V(C)$ is not included in $V^-_5 \cup V^{d-}_5$
nor in $V^+_5 \cup V^{d+}_5$. 
By Claim~\ref{claim:V4_tree}, $V(C)$ is not included in $\Vquatre$. 
by Claim~\ref{claim:chelou_only_in_V4} all chelou arcs have their end-vertices in $\Vquatre$. This implies that there is no arc from $V_5^-$ to $\Vquatre \cup V_5^+ \cup V_5^{d+}$, nor from $\Vquatre \cup V_5^- \cup V_5^{d-}$ to $V_5^+$ (because such arcs would be chelou arcs). 
Hence, $C$ contains an arc $xy$ between $ V^{d-}_5$ and $V_5^{d+}$. 

Since $D$ has no isolated digon by Claim~\ref{claim:no_isolated_digon}, both $x$ and $y$ are extremities of some $2$-thread . 
By Claim~\ref{claim:apex}, $D$ has only one $2$-thread.  
So there exists a vertex $z$ such that $[x,z,y]$ is a $2$-thread of $D$. 
But as $xy \in A(D)$, $\{x,y,z\}$ induces a bidirected $3$-cycle minus one arc, a contradiction to Claim~\ref{claim:cycle_minus_one_arc}.
\end{proofclaim}

We are now ready to conclude.
We distinguish three cases depending on whether $V_6$ is a stable set
or not and whether $D$ has a digon or not.
Recall that $|V_5^+ \cup V^{d+}_5| = |V_5^- \cup V^{d-}_5|$.


{\medskip \noindent \bf Case 1:} There is an arc $z_1 z_2 \in A(D)$ with $z_1,z_2 \in V_6$.

Since vertices in $V_6$ are incident to no digon (Claim~\ref{claim:no6digon}), $z_1$ is a simple standard neighbour of  $z_2$. 
Therefore, we can apply Claim~\ref{claim:result_discharging}~\ref{deg6-nodigon} on $z_1$ with $\ell_1 \geq 1$, so $w^*(z_1) \leq -\frac{1}{3}$. 
Similarly, $z_2$ is a simple standard neighbour of $z_1$, so $w^*(z_2) \leq -\frac{1}{3}$. 


Assume that $D$ contains a digon. 
Then, by Claims~\ref{claim:component_of_B(D)} and~\ref{claim:no_isolated_digon}, $D$ contains a $2$-thread.
By Claim~\ref{claim:2thread}, this $2$-thread has an apex $v$, and $w^*(v) \leq -\frac{2}{3}$ by Claim~\ref{claim:apex}~\ref{apex-2/3}.
If $v\notin \{z_1,z_2\}$, then $\rho(D) \le \Sigma(D)  \leq w^*(v) + w^*(z_1) + w^*(z_2) \leq -\frac{2}{3} - \frac{1}{3} - \frac{1}{3} < -1$, a contradiction.
So $v\in  \{z_1,z_2\}$ and by directional duality, we may assume $v=z_1$. 
So $z_1$ has two neighbours incident to a digon. 
Hence, by Claim~\ref{claim:result_discharging}~\ref{deg6-nodigon} applied
on $z_1$ with $\ell_1 \geq 1$ (because $z_1$ is a simple standard neighbour of $z_2$) and  $\ell_3 \geq 2$ (because $z_1$ has two standard neighbours incident to a digon), $w^*(z) \leq -\frac{1}{3} - \frac{2}{3} = -1$.
Thus $\rho(D) \le \Sigma(D)  \leq w^*(z_1) + w^*(z_2) \leq -1 - \frac{1}{3} < -1$, a contradiction.
Henceforth $D$ has no digon.

Assume $D[V_6]$ contains a dicycle $C$. 
Then $C$ has at least three vertices. Moreover,  each vertex in $C$ is a simple standard neighbour of its two neighbours in $C$. 
So by Claim~\ref{claim:result_discharging}~\ref{deg6-nodigon} with $\ell_1 \geq 1$,
every vertices in $C$ has final charge at most $-\frac{2}{3}$.
So $\rho(D) \leq \Sigma(D) \leq \sum_{v \in V(C)} w^*(v) \leq - \frac{2|C|}{3} <-1$, a contradiction.
Henceforth $D[V_6]$ is acyclic.

Every vertex in $V^+_5\cup V^-_5$ has final charge at most $-\frac{1}{6}$
by Claim~\ref{claim:result_discharging}~\ref{deg5-nodigon}. Hence,  
$\rho(D) \leq \Sigma(D) \leq w^*(z_1) + w^*(z_2) + \sum_{v \in V^+_5\cup V^-_5}
w^*(v) \leq  -\frac{2}{3} -\frac{|V_5|}{6}$.  
Therefore $|V^+_5\cup V^-_5| \leq 2$. 
Since there is no digon, $V^{d+}_5\cup V^{d-}_5$ is empty and so $|V_5^+| = |V_5^-|$.
Thus $|V_5^+|, |V_5^-| \leq 1$. 
Therefore there is no directed cycle in $D[V_5^+]$ nor in $D[V_5^-]$, and so $D[\Vquatre \cup V_5^- \cup V_5^+]$ is acyclic by~Claim~\ref{claim:cycles_in_V5_V4}.
Hence colouring the vertices of $\Vquatre \cup V_5^- \cup V_5^+$ with $1$ and those of $V_6$ with $2$ yields a $2$-dicolouring of $D$,
a contradiction. 

{\medskip \noindent \bf Case 2:} $D[V_6]$ is stable and there is no digon in $D$.

If both $D[V_5^-]$ and $D[V_5^+]$ are acyclic subdigraphs, 
then so is $D[\Vquatre \cup V^-_5\cup V^+_5]$ by Claim~\ref{claim:cycles_in_V5_V4}.
Moreover, by assumption, $D[V_6]$ is acyclic.
Hence colouring the vertices of  $\Vquatre \cup V^-_5\cup V^+_5$ with $1$ and those of $V_6$ with $2$ yields a $2$-dicolouring of $D$, a contradiction.

Henceforth, without loss of generality, we may assume that $V_5^-$ is not acyclic and thus has size at least $3$.
But $|V_5^-| = |V_5^+|$ and $|V_5^-| + |V_5^+| \leq 6$ because every vertex of degree $5$ has final charge at most $-\frac{1}{6}$.
It follows that  $|V_5^-| = |V_5^+| =3$ and $D[V^-_5]$ is a directed $3$-cycle. 
Pick an arbitrary vertex $v^- \in V_5^-$. As $V_5^-$ is a directed $3$-cycle,
$v^-$ has two in-neighbours out of $V_5^-$. In particular,
there exists a vertex $v^+ \in V_5^+$ which is not adjacent to $v^-$
(recall that there are no arcs from $V_5^-$ to $V_5^+$ in $D$ as such an arc would be a chelou arc, a contradiction to~Claim~\ref{claim:chelou_only_in_V4}).
Now, colour $\Vquatre \cup V_5^- \cup V_5^+ \setminus \{v^-,v^+\}$ with colour $1$,
and $V_6 \cup \{v^-,v^+\}$ with colour $2$.
Since $D$ is $3$-dicritical, there must be a monochromatic directed cycle.
But $D[V_5^- \setminus \{v^-\}]$ and $D[V_5^+ \setminus \{v^+\}]$ are acyclic subdigraphs because they are of size $2$.
So $D[(\Vquatre \cup V_5^- \cup V_5^+) \setminus \{v^-,v^+\}]$ is also acyclic
by Claim~\ref{claim:cycles_in_V5_V4}.
Therefore, there is a dicycle in $D[V_6 \cup \{v^-,v^+\}]$ which must contain $v^+$ and $v^-$ because $V_6$ and $\{v^-,v^+\}$ are stable sets.
Thus $v^+$ has an out-neighbour $u$ in $V_6$. But then $v^+$ is a simple standard neighbour of $u$. Thus Claim~\ref{claim:result_discharging}~\ref{deg6-nodigon} applies to $u$ with $\ell_1 \geq 1$, so $w^*(u) \leq - \frac{1}{3}$. Consequently, $-1\leq \Sigma(D) \leq \sum_{v \in V_5^- \cup V_5^+}w^*(v) + w^*(u) \leq -6\times \frac{1}{6} - \frac{1}{3}  < -1$, a contradiction.

{\medskip \noindent \bf Case 3:} $D[V_6]$ is stable and
    there is a digon in $D$. 

By Claims~\ref{claim:component_of_B(D)} and~\ref{claim:no_isolated_digon}, each digon is in a $2$-thread. 
By Claim~\ref{claim:apex}, $D$ contains exactly one $2$-thread. 
Let us denote it by $[v_1,x,v_2]$.
Vertices $v_1$ and $v_2$ do not have degree $6$ by Claim~\ref{claim:no6digon} nor degree $4$ by Claim~\ref{claim:component_of_B(D)}. 
So $\{v_1, v_2\}\subseteq V^{d+}_5 \cup V^{d-}_5$.

Let $z$ be the apex of $[v_1,x,v_2]$. 
By Claim~\ref{claim:apex}~\ref{apex-2/3}, $w^*(z) \leq -\frac{2}{3}$. 
As vertices in $V_5^+ \cup V_5^-$ have final charge at most $-\frac{1}{6}$, $|V_5^+ \cup V_5^-|\leq 2$.
This implies that $V^{d-}_5 \cup V_5^-$ and $V^{d+}_5 \cup V_5^+$ both have
size at most $2$ and so induce acyclic subdigraphs.
Therefore $D[\Vquatre \cup  V_5^+ \cup V_5^-\cup V^{d+}_5 \cup V^{d-}_5]$ is acyclic
by Claim~\ref{claim:cycles_in_V5_V4}.
Since the two neighbours ($v_1$ and $v_2$) of $x$ are not in $V_6$, and $V_6$ is a stable set, then 
$V_6 \cup \{x\}$ is a stable set and so $D[V_6\cup \{x\}]$ is acyclic.
Hence, colouring $\Vquatre \cup  V_5^+ \cup V_5^-\cup V'^+_5 \cup V'^-_5$ 
with $1$ and $V_6 \cup \{x\}$ with $2$, we get a $2$-dicolouring of $D$, a contradiction.

This concludes the proof of Theorem~\ref{thm:main_potential}. \qed

\section{Conclusion and further research}

The main question investigated in this paper is to get upper and lower bounds on $o_k(n)$, the minimum number of arcs in a $k$-dicritical oriented graph of order $n$.

By Theorem~\ref{theorem:improved_lower_bound_on_ok} and the results in Subsection~\ref{sec:lowerk}, we have
$$\pth{k - \frac{3}{4}-\frac{1}{4k-6}} n(D) + \frac{3}{4(2k-3)} \leq o_k(n) < (2k-3)n$$
for every $k\geq 3$ and $n$ large enough.

A first step would be to determine the order of $o_k(n)$ that is a constant $c_k$ such that
$o_k(n) = c_k n + o(n)$. We believe that the constructions given to establish the upper bound are nearly optimal. In particular, we conjecture that $o_k(n)$ is getting closer and closer to the upper bound as $k$ tends to infinity.

\begin{conjecture}
 $\frac{c_k}{2k -3} \rightarrow 1$ when $k \rightarrow +\infty$.
\end{conjecture}

We mainly focused on the case $k=3$. Proposition~\ref{prop:O} and Theorem~\ref{thm:main_thm_oriented_critical} yield 
$$\frac{7n+2}{3} \leq o_3(n) \leq \left\lceil\frac{5n}{2}\right\rceil$$
for every even $n \geq 12$.
In the conclusion of \cite{bang2019haj}, Bang-Jensen, Bellitto, Schweser, and Stiebitz  conjectured that every oriented 3-dicritical digraph of $n$ vertices has at least $\frac 5 2 n$ arcs. 
Together with Corollary~\ref{coro:construction_3_dicritical_odd_and_even_order},
this implies $o_3(n) = \lceil\frac{5}{2}n\rceil$ for every $n \geq 12$.
Hence, a natural question is whether the lower bound of Theorem \ref{thm:main_thm_oriented_critical} can be improved. 
We have no reason to think that our bound is tight and we tend to believe the conjecture of \cite{bang2019haj} to be true. 
During this project, we used the same method to obtain lower bounds on $o_3(n)$ that we improved incrementally until we reached the one of Theorem \ref{thm:main_thm_oriented_critical}.  We believe that our method could still be pushed further and provide better lower bounds on $o_3(n)$, at the cost of a longer and sharper case analysis of the structure of the $3$-dicritical oriented graphs. 
However, the question remains as to how far our method can be pushed and whether it can provide a tight bound.

\medskip

Another question that has been touched in this paper is to determine the set of orders $N_k$ for which there exists a $k$-dicritical oriented graph.

The first particular case consists in determining the minimum order $n_k$ of a $k$-dicritical oriented graph.
This is equivalent to determining the minimum integer $n_k$ such that there exists a tournament of order $n_k$ and dichromatic number $k$.
The directed cycles of order at least $3$ are the $2$-dicritical oriented graphs, so  $n_2=3$ and $N_2 = \{n\in \mathbb{N} \mid n\geq 3\}$.
Neumann Lara~\cite{Neumann_lara_small_tournament} proved that $n_3=7$ and $n_4=11$.
Bellitto, Bousquet, Kabela, and Pierron~\cite{BBKP} recently established $n_5=19$. 
For larger values of $k$, $n_k$ is unknown.

Let $D_1$ and $D_2$ be two oriented graphs.
 The  oriented graph $\triangle(1,D_1,D_2)$ is the oriented graph obtained from the disjoint union of $D_1$ and $D_2$ by adding a new vertex $u_0$ and all the arcs from $u_0$ to $D_1$, all the arcs from $D_1$ to $D_2$, and all the arcs from $D_2$ to $u_0$. 
 
 \begin{lemma}[Hoshino and Kawarabayashi~\cite{HoKa15}]\label{lem:triangle}
 Let $k\geq 2$ be an integer.
 If $D_1$ and $D_2$ are $k$-dicritical, then $\triangle(1,D_1,D_2)$ is $(k+1)$-dicritical.
 \end{lemma}

\begin{proof}
Let $D_1$ and $D_2$ be two $k$-dicritical digraphs and $D = \Delta(1,D_1,D_2)$.

\medskip

Let us first show that $D$ is not $(k-1)$-dicolourable.\\
Suppose for a contradiction that there is a $(k-1)$-dicolouring $\phi$ of $D$.
Then, as $D_1$ and $D_2$ are not $(k-2)$-dicolourable, there exist $u_1 \in V(D_1)$
and $u_2 \in V(D_2)$ such that $\phi(u_0)=\phi(u_1)=\phi(u_2)$.
Then $(u_0,u_1,u_2,u_0)$ is a monochromatic directed $3$-cycle in $D$, a contradiction.

\medskip

Let us show that $D\setminus e$ is $(k-1)$-dicolourable, for any arc $e=xy \in A(D)$.\\
If $x,y \in V(D_1)$, then $D_1\setminus e$ has a $(k-2)$-dicolouring $\phi_1$ by assumption.
Moreover, $D_2$ has a $(k-1)$-dicolouring $\phi_2$.
By relabelling the colours, we may assume that $\phi_1$ takes values in $[k-2]$
and $\phi_2$ in $[k-1]$.
Then colour $u_0$ with $k-1$.

Suppose for a contradiction that the resulting colouring $\phi$ has a monochromatic
dicycle $C$ in $D\setminus e$.
If $C$ is coloured $k-1$, then it does not intersect $D_1$.
Moreover, $u_0$ has no out-neighbours coloured $k-1$, so $u_0 \not\in C$.
Thus $C$ is included in $D_2$, contradicting the fact that $\phi_2$ is
a $(k-1)$-dicolouring of $D_2$.
If $C$ is not coloured $k-1$, then $C$ is included in $V(D_1) \cup V(D_2)$.
But there is no arc from $D_2$ to $D_1$ in $D-u_0$.
Hence $C$ is included either in $D_1 \setminus e$ or in $D_2$, a contradiction in both cases.
If $x,y \in V(D_2)$, the proof is identical.

Now suppose $x \in V(D_1)$ and $y \in V(D_2)$.
By assumption, $D_1 -x$ (resp. $D_2 - y$) has a $(k-2)$-dicolouring $\phi_1$
(resp. $\phi_2$), with colours taken in $[k-2]$.
Then we colour $u_0,x$ and $y$ with $k-1$.
Suppose for a contradiction that the resulting colouring $\phi$ has a monochromatic
dicycle $C$ in $D \setminus xy$.
If $C$ is not coloured $k-1$, then it is either included in $D_1-x$ or in $D_2-y$
as there is no arc from $D_2$ to $D_1$.
But this contradicts the fact that $\phi_1$ is a dicolouring of $D_1$
and $\phi_2$ a dicolouring of $D_2$.
Hence $C$ is coloured $k-1$. But the only vertices coloured $k-1$ are $u_0,x$ and $y$,
which do not induce a dicycle in $D \setminus xy$.

The only remaining case (up to symmetry) is $e = u_0 y$ for some $y \in D_1$.
Consider a $(k-2)$-dicolouring $\phi_1$ of $D-y$ with colours taken
in $[k-2]$, and a $(k-1)$-dicolouring $\phi_2$ of $D_2$ with colours taken in
$[k-1]$.
Then colour $u_0$ and $y$ by $k-1$.
Suppose for a contradiction that the resulting colouring $\phi$ has a monochromatic
cycle $C$ in $D\setminus e$.
If $C$ is not coloured $k-1$, then it must be included in $D_1$ or $D_2-y$,
a contradiction. Hence $C$ is coloured $k-1$. Moreover, $C$ must contains $u_0$.
But no out-neighbour of $u_0$ in $D\setminus e$ is coloured $k-1$, a contradiction.

This proves that $D\setminus e$ is $(k-1)$-dicolourable for every arc $e$,
and so $D$ is $k$-dicritical.
\end{proof}

 Lemma~\ref{lem:triangle} implies that $\triangle(1,\vec{C}_3,\vec{C}_{n-4})$ is $3$-dicritical for all $n \geq 7$. Hence $N_3 = \{n\in \mathbb{N} \mid n\geq 7\}$.
 Moreover, by induction it implies the following.
 \begin{corollary}
 There exists a smallest integer $p_k$ such that there exists a $k$-dicritical oriented graph of order $n$ for any $n\geq p_k$.
 \end{corollary}

In view of $N_2$ and $N_3$, one might be tempted to believe that $p_k = n_k$, that is $N_k= \{n\in \mathbb{N} \mid n\geq n_k\}$ for all $k\geq 2$.
This is unfortunately false for $k=4$ as shown by the following proposition.

\begin{proposition}
There is no $4$-dicritical oriented graph on $12$ vertices.
\end{proposition}
\begin{proof}
Let us denote by $P_{11}$ the Paley tournament on 11 vertices, whose vertices are labelled from $0$ to $10$ in which $ij$ is an arc if and only if $j-i$ is in $\{1,3,4,5,9\}$ modulo 11.
Neumann-Lara~\cite{Neumann_lara_small_tournament} proved that $P_{11}$ is the unique $4$-dichromatic tournament of order 11, and
Bellitto, Bousquet, Kabela, and Pierron~\cite{BBKP} proved that every $4$-dichromatic tournament of order $12$ contains $P_{11}$.

Assume for a contradiction that there exists a $4$-dicritical oriented graph $D$ of order $12$.
$D$ does not contain $P_{11}$. As long as there are two non-adjacent vertices $a,b$ such that adding the arc $ab$ does not make a $P_{11}$ appear, we add $ab$.
This results in an oriented graph $D'$ (of order 12) such that adding any arc between two non-adjacent vertices make a $P_{11}$ appear. Note that $\dic(D') = 4$ because $D'$ contains $D$ and $n_5 > 12$.

Since $D'$ does not contain $P_{11}$, by Bellitto et al. result, $D'$ is not a tournament.
Therefore there are two non-adjacent vertices $a$ and $b$.

By construction, if we add the arc $ab$ to $D'$, then a $P_{11}$ appears. 
Label the vertices of this $P_{11}$ by $u_0, u_1, \dots, u_{10}$ so that $u_iu_j$ is an arc if and only if $j-i$ is in $\{1,3,4,5,9\}$ modulo 11, $a=u_0$ and $b=u_1$.
This is possible because $P_{11}$ is arc-transitive.

Similarly, adding the arc $ba$ to $D'$ let a $P_{11}$ appear, whose vertices can be labelled 
$v_0, v_1, \dots, v_{10}$ so that $v_iv_j$ is an arc if and only if $j-i$ is in $\{1,3,4,5,9\}$ modulo 11, $b=v_0$ and $a=v_1$.

Set $I^{++} = N^+(a)\cap N^+(b)$, $I^{--} = N^-(a)\cap N^-(b)$,
$I^{+-} = N^+(a)\cap N^-(b)$, and $I^{-+} = N^-(a)\cap N^+(b)$.
Note that those four sets are disjoint and contained in $V(D')\setminus \{a,b\}$.
Now $\{u_4, u_5\}\subseteq I^{++}$, $\{v_4, v_5\}\subseteq I^{++}$,
$\{u_7, u_8\}\subseteq I^{--}$, $\{v_7, v_8\}\subseteq I^{--}$,
$\{v_2, v_6, v_{10}\}\subseteq I^{+-}$,
and $\{u_2, u_6, u_{10}\}\subseteq I^{-+}$.
Since $|V(D')\setminus \{a,b\}|=10$, it follows that $(I^{++}, I^{--}, I^{+-}, I^{-+})$ is a partition of $V(D')\setminus \{a,b\}$, 
 $I^{++}= \{u_4, u_5\} = \{v_4, v_5\}$, 
$I^{--} = \{u_7, u_8\} = \{v_7, v_8\}$,
$ I^{+-} = \{v_2, v_6, v_{10}\} $,
and $I^{-+} = \{u_2, u_6, u_{10}\}$.

Since $u_7u_8$ and $v_7v_8$ are arcs we have $u_7=v_7$, and $u_8=v_8$.
Moreover $v_3$ is in $I^{-+} = \{u_2, u_6, u_{10}\}$.
But $v_3$ dominated both $v_7=u_7$ and $v_8=u_8$, while $u_2$ and $u_6$ do not dominate $u_8$ 
and $u_{10}$ does not dominate $u_7$.
This is a contradiction.
\end{proof}

\section*{Acknowledgements} 
This research was partially supported by the french Agence Nationale de la Recherche under contract Digraphs ANR-19-CE48-0013-01 and contract DAGDigDec (JCJC) ANR-21-CE48-0012, and by the group Casino/ENS Chair on Algorithmics and Machine Learning.

\end{document}